\documentclass[12pt]{article}
\usepackage{amsmath,amsthm,amsfonts,amssymb,times,xcolor,hyperref}
\usepackage{enumitem,cleveref}

\font\smallit=cmti10

\newtheorem{theorem}{Theorem}[section]
\newtheorem{lemma}[theorem]{Lemma}
\newtheorem{corollary}[theorem]{Corollary}

\theoremstyle{definition}
\newtheorem{example}[theorem]{Example}
\newtheorem{comment}[theorem]{Comment}

\newtheorem{conjecture}[theorem]{Conjecture}
\newtheorem*{notation}{Notation}

\newcommand{\ediv}{\mathrel{\|}}
\DeclareMathOperator{\real}{Re}
\newcommand{\mmod}{\ \text{mod}\ }

\newcommand\blfootnote[1]{%
  \begingroup
  \renewcommand\thefootnote{}\footnote{#1}%
  \addtocounter{footnote}{-1}%
  \endgroup
}

\allowdisplaybreaks

\begin{document}

\begin{center}
\textbf{\large On the factorization of 
\boldmath $x^{d_{r}} + x^{d_{r-1}} + \cdots + x^{d_{2}} + x^{d_{1}} + p$ \unboldmath} \\[4pt]  
\textbf{\large and a conjecture of C.~Nicol}
\vskip 20pt
{Michael Filaseta}\\
{\smallit Dept.~Mathematics, 
University of South Carolina, 
Columbia, SC, 29208}\\
{\tt filaseta@math.sc.edu}\\ 
\vskip 10pt
{Lilit Martirosyan}\\
{\smallit Dept.~Mathematics and Stat., 
University of North Carolina Wilmington, 
Wilmington, NC 28403--5970}\\
{\tt martirosyanl@uncw.edu}\\ 
\vskip 10pt
{London Cameron Swan}\\
{\smallit Dept.~Mathematics and Stat., 
University of North Carolina Wilmington, 
Wilmington, NC 28403--5970}\\
{\tt lcs2908@uncw.edu}\\ 
\end{center}

\smallskip
\begin{abstract}
For an arbitrary prime $p$ and arbitrary positive integers  $r$ and $d_{1}, \ldots, d_{r}$ with $d_{r} > \cdots > d_{1}$, we show that $x^{d_{r}} + x^{d_{r-1}} + \cdots + x^{d_{2}} + x^{d_{1}} + p$ removed of all of its monic irreducible reciprocal factors is irreducible.  As a consequence, we establish a number of results related to a conjecture of Charles Nicol that, over the rationals, the sum of two cyclotomic polynomials $\Phi_{n}(x) + \Phi_{m}(x)$ is a product of cyclotomic polynomials times either $2$ or an irreducible polynomial, where $n$ and $m$ are integers exceeding $1$.  We also establish similar results for $\Phi_{n}(x)\Phi_{m}(x)+1$.  
\end{abstract}

\blfootnote{\text{2020 Mathematics Subject Classification}: Primary 11R09; Secondary 11C08, 12E05} 
\blfootnote{\text{Keywords:  irreducible polynomial, cyclotomic polynomial, reciprocal polynomial}}

\medskip
\section{Introduction}

The \textit{reciprocal} of a polynomial $w(x) \in \mathbb R[x]$ is the polynomial $\widetilde{w}(x) = x^{\deg w} w(1/x)$.  A  $w(x) \in \mathbb R[x]$ is a \textit{reciprocal polynomial} if $\widetilde{w}(x) = \pm w(x)$.  A polynomial  $w(x) \in \mathbb R[x]$ is reciprocal if and only if for every $\alpha \in \mathbb C$ that is a root of $w(x)$ to some multiplicity $k \ge 1$, we have $\alpha \ne 0$ and $1/\alpha$ is a root of $w(x)$ with multiplicity $k$.  For $u(x), v(x)$ and $w(x) \in \mathbb R[x]$, we note that 
\begin{itemize}
\item
if $f(x) = u(x)v(x)$, then $\tilde{f}(x) = \tilde{u}(x) \tilde{v}(x)$;
\item
$\deg \widetilde{w}(x) \le \deg w(x)$ with equality if and only if $w(0) \ne 0$; and
\item
if $w(0) \ne 0$, then the reciprocal of $\widetilde{w}(x)$ is $w(x)$.
\end{itemize}
The cyclotomic polynomials $\Phi_{n}(x)$ are examples of reciprocal polynomials. 

We will be interested in the factorization of polynomials in $\mathbb Z[x]$.  The non-reciprocal part of $w(x) \in \mathbb Z[x]$ is $w(x)$ removed of all of its irreducible reciprocal factors in $\mathbb Z[x]$ having positive leading coefficient.  For example, the non-reciprocal part of
\[
(1-x)^{3} (x^{3} + x + 1) (x^{2} + 3x + 1)
\]
is $-(x^{3} + x + 1)$ obtained by removing the factor $x-1$ three times and removing the factor $x^{2} + 3x + 1$.  The non-cyclotomic part of $w(x) \in \mathbb Z[x]$ is $w(x)$ removed of all of its cyclotomic factors.  The non-cyclotomic part of the polynomial in the example just given is $-(x^{3} + x + 1)(x^{2} + 3x + 1)$. 

Let 
\begin{equation}\label{basicFeq}
F(x) = x^{d_{r}} + x^{d_{r-1}} + \cdots + x^{d_{2}} + x^{d_{1}} + p,
\quad \text{with } r, d_{j} \in \mathbb Z^{+}, d_{r} > \cdots > d_{1}, \text{ and } p \text{ prime}.
\end{equation}
Since $F(x)$ is monic and $F(0) > 1$, the polynomial $F(x)$ is not reciprocal, and it follows that the non-reciprocal part of $F(x)$ has positive degree.  One of our goals is to show the following.  

\begin{theorem}\label{factorizationthm}
The non-reciprocal part of the polynomial $F(x)$ given in \eqref{basicFeq} is irreducible.
\end{theorem}

After establishing this theorem, we will turn to results associated with the following.

\begin{conjecture}[Charles Nicol \cite{nicol}]\label{nicolconj}
If $n$ and $m$ are integers exceeding $1$, then the non-cyclotomic part of the polynomial $\Phi_{n}(x) + \Phi_{m}(x)$ is either $2$ or irreducible in $\mathbb Q[x]$.
\end{conjecture}

In the above conjecture, irreducibility is in $\mathbb Q[x]$ rather than in $\mathbb Z[x]$ as otherwise exceptions occur.  For example, the polynomial $\Phi_{42}(x) + \Phi_{21}(x)$ has the two irreducible non-cyclotomic factors $2$ and $x^{12}-x^{8}+x^{6}-x^{4}+1$ in $\mathbb Z[x]$.  Also, the examples $\Phi_{2}(x) + \Phi_{2}(x)$ and $\Phi_{3}(x) + \Phi_{6}(x)$ show that the non-cyclotomic part of $\Phi_{n}(x) + \Phi_{m}(x)$ can be $2$.
On the other hand, the non-cyclotomic part of $\Phi_{n}(x) + \Phi_{m}(x)$ can be $2$ only if $\phi(n) = \phi(m)$ since otherwise the leading coefficient of $\Phi_{n}(x) + \Phi_{m}(x)$ is $1$.  

Computations were done by the authors verifying that the non-cyclotomic part of the polynomial $\Phi_{n}(x) + \Phi_{m}(x)$ is either $2$ or irreducible in $\mathbb Q[x]$ for all $(n,m)$ with $2 \le n \le 500$ and $2 \le m \le 500$.  As the number of prime divisors of $n$ and $m$ and the gaps between them can alter the structure of $\Phi_{n}(x)$ and $\Phi_{m}(x)$ considerably, we view this as only moderate evidence for the conjecture.

We illustrate here the connection between Theorem~\ref{factorizationthm} and this conjecture.  For positive integers $n$ and $m$ with $\phi(n) \ne \phi(m)$, we will see in Lemma~\ref{recipfactf(x)} that every irreducible reciprocal factor of $\Phi_{n}(x) + \Phi_{m}(x)$ with positive leading coefficient is cyclotomic.  Thus, the non-reciprocal and non-cyclotomic parts of $\Phi_{n}(x) + \Phi_{m}(x)$, for $\phi(n) \ne \phi(m)$, are the same.  
For primes $p$ (different from $p$ in \eqref{basicFeq}), $q$ and $\ell$ satisfying $p > q > \ell$, one can show that (see \eqref{factorphik:eq1})
\begin{equation}\label{introeq:ellpqeq}
\big( \Phi_{\ell p}(x) + \Phi_{\ell q}(x) \big) \Phi_{\ell}(x)
= \Phi_{\ell}(x^{p}) + \Phi_{\ell}(x^{q}) 
= \sum_{i=1}^{\ell-1} x^{p i} + \sum_{j=1}^{\ell-1} x^{q j} + 2.
\end{equation}
The condition $p > q > \ell$ on the primes $p$, $q$ and $\ell$ imply that the exponents on the powers of $x$ above are distinct so that $\Phi_{\ell}(x^{p}) + \Phi_{\ell}(x^{q})$ is of the form given for $F(x)$ in \eqref{basicFeq}.   Therefore, Theorem~\ref{factorizationthm} implies that the non-reciprocal part of $\Phi_{\ell p}(x) + \Phi_{\ell q}(x)$ is irreducible.  Hence, in the case that $n = \ell p$ and $m = \ell q$ with primes $p > q > \ell$, the conclusion of Conjecture~\ref{nicolconj} follows.

One goal of ours is to generalize this example to polynomials of the form $\Phi_{\ell^{t} p}(x) + \Phi_{\ell^{t} q}(x)$ and to completely characterize the cyclotomic factors that can occur in this case.  More precisely, we will show the following.

\begin{theorem}\label{kpq:mainthm}
Let $p$, $q$ and $\ell$ be primes such that $p > q > \ell$, and let $k = \ell^{t}$ for some integer $t \ge 0$.  Then the non-cyclotomic part of $\Phi_{k p}(x) + \Phi_{k q}(x)$ is irreducible.  Furthermore, let  $D$ be a positive integer, and define $d$ and $j$ by $D = d \ell^{j}$, where $j \ge 0$ is an integer, $d \in \mathbb Z^{+}$ and $\ell \nmid d$.  Then $\Phi_{D}(x)$ divides $\Phi_{k p}(x) + \Phi_{k q}(x)$ if and only if $j \le t-1$ and one of the following hold.
\begin{enumerate}[label=(\roman*)]
\item
There is a positive integer $s$ such that $2^{s} \ediv (\ell-1)$, $d = 2^{s+1} d_{o}$ for some odd divisor $d_{o}$ of $\ell-1$, and $q \equiv -p \pmod{d}$. \label{mainthmfirstitem}
\item
The prime $\ell$ satisfies $\ell \equiv 3 \pmod{4}$, $d$ is a multiple of $4$ that divides $\ell+1$, and $q \equiv p + (d/2) \pmod{d}$. \label{mainthmseconditem}
\end{enumerate}
\end{theorem}

Based on the above theorem, we will see that the density of primes $p$, $q$ and $\ell$ satisfying $\ell < q < p \le X$ for which $\Phi_{\ell p}(x) + \Phi_{\ell q}(x)$ is irreducible tends to $7/12$ as $X \rightarrow \infty$ (see Corollary~\ref{seventwelfthcor}).  Thus, more often than not these polynomials do not have cyclotomic factors.

We will also address further results that come from finding multiples of $\Phi_{n}(x) +\Phi_{m}(x)$ that are of the form \eqref{basicFeq}.  We will see that this approach is also connected to the study of the factorization of the polynomials $\Phi_{n}(x)\,\Phi_{m}(x) + 1$.  This will allow us to address similar results for these polynomials such as the following.

\begin{theorem}\label{21case:thm}
If $p_{1}, p_{2}$ and $q_{1}$ are distinct odd primes and $e$ is a nonnegative integer, then the non-cyclotomic part of $\Phi_{2^{e}p_{1}p_{2}}(x) \,\Phi_{2^{e}q_{1}}(x) + 1$ is irreducible.
\end{theorem}

With that in mind, it is reasonable (or as reasonable as with Conjecture~\ref{nicolconj}) to make the following conjecture.

\begin{conjecture}\label{newconj}
If $n$ and $m$ are integers exceeding $1$, then the non-cyclotomic part of the polynomial $\Phi_{n}(x)\,\Phi_{m}(x)+1$ is irreducible.
\end{conjecture}

Observe that, with $n$ and $m$ exceeding $1$, the polynomial $\Phi_{n}(x)\,\Phi_{m}(x)+1$ is monic with constant term $2$, so it will have an irreducible non-reciprocal factor in $\mathbb Q[x]$.  In the case of $\Phi_{n}(x) + \Phi_{m}(x)$ and $\phi(n) = \phi(m)$, the same was not true.
We note also that we verified computationally that the non-cyclotomic part of the polynomial $\Phi_{n}(x)\,\Phi_{m}(x)+1$ is irreducible for all $(n,m)$ with $2 \le n \le 500$ and $2 \le m \le 500$.

We turn now to a couple more general results, but first give some notation.

\begin{notation}
We use the standard big-oh notation.  Thus, the expression $O_{t,s}\big(G_{1}(X)\big)$ indicates a function $G_{2}(X)$ for which $|G_{2}(X)| \le C_{t,s} G_{1}(X)$, where $C_{t,s}$ is a positive constant depending only on $t$ and $s$ and where $X \ge X_{0}(t,s)$ for some constant $X_{0}(t,s)$ depending only on $t$ and $s$.  The notation $O\big(G_{1}(X)\big)$ has the same meaning but with the constant instead being an absolute constant.  We will also use Landau's notation $G_{2}(X) \ll G_{1}(X)$ to mean the same as $G_{2}(X) = O(G_{1}(X))$.  We will also make use of the little-oh notation $G_{2}(X) = o\big(G_{1}(X)\big)$ to mean that $G_{2}(X)/G_{1}(X)$ tends to $0$ as $X$ goes to infinity.  Finally, we make use of the standard notation $\omega(n)$ to denote the number of distinct prime divisors of a positive integer $n$.
\end{notation}

\begin{theorem}\label{arbitrarynumprimes:thm}
Fix positive integers $t$ and $s$.  
Then for almost all choices of $t+s$ distinct odd primes $p_{1}, \ldots, p_{t}, q_{1}, \ldots, q_{s}$, the numbers $n = 2^{e} p_{1} \cdots p_{t}$ and $m = 2^{e} q_{1} \cdots q_{s}$, where $e$ is an arbitrary nonnegative integer, 
satisfy the following.
\begin{enumerate}[label=(\roman*)]
\item
If $t \equiv s {\hskip -2pt}\pmod{2}$, then the non-cyclotomic part of $\Phi_{n}(x)+\Phi_{m}(x)$ is irreducible.  \label{arbitrarynumprimesfirstitem}
\item
If $t \not\equiv s {\hskip -2pt}\pmod{2}$, then the non-cyclotomic part of $\Phi_{n}(x)\,\Phi_{m}(x) + 1$ is irreducible.  \label{arbitrarynumprimesseconditem}
\end{enumerate} 
More precisely, the number of choices for $t+s$ distinct odd primes $p_{1}, \ldots, p_{t}, q_{1}, \ldots, q_{s}$ each at most $X$ such that there exists a nonnegative integer $e$ for which the applicable conclusion in \ref{arbitrarynumprimesfirstitem} or \ref{arbitrarynumprimesseconditem} fails for $n = 2^{e} p_{1} \cdots p_{t}$ and $m = 2^{e} q_{1} \cdots q_{s}$ is $O_{t,s}\big(\pi(X)^{t+s-1}\big) = O_{t,s}\big(X^{t+s-1}/(\log X)^{t+s-1}\big)$.
\end{theorem}

\begin{theorem}\label{maingeneralthm2}
Uniformly for $t$ and $s$ positive integers satisfying
\[
t+s = o\big( \sqrt{\log\log X} \,\big),
\]
the density of pairs of odd positive integer pairs $(n,m)$, with $n \le X$, $m \le X$, $nm$ squarefree, $\omega(n)= t$ and $\omega(m) = s$ such that the non-cyclotomic part of $\Phi_{n}(x)+\Phi_{m}(x)$ is irreducible if $t \equiv s \pmod{2}$ and the non-cyclotomic part of $\Phi_{n}(x)\Phi_{m}(x)+1$ is irreducible if $t \not\equiv s \pmod{2}$ is $1$ as $X \rightarrow \infty$.
\end{theorem}

Observe that if $t$ and $s$ are fixed positive integers, Theorem~\ref{maingeneralthm2} implies that for almost all positive squarefree relatively prime integers $n$ and $m$, with $\omega(n) = t$ and $\omega(m)= s$, the non-cyclotomic part of $\Phi_{n}(x)+\Phi_{m}(x)$ is irreducible if $t \equiv s \pmod{2}$ and the non-cyclotomic part of $\Phi_{n}(x)\Phi_{m}(x)+1$ is irreducible if $t \not\equiv s \pmod{2}$.

Elaborating on an earlier comment, we will see that the polynomials $\Phi_{n}(x) +\Phi_{m}(x)$ and $\Phi_{n}(x)\,\Phi_{m}(x)+1$ we investigate here will have no irreducible reciprocal factors occurring with multiplicity $> 1$ and all their irreducible reciprocal factors will be cyclotomic.  We note that this is not the case for general $F(x)$ of the form given in \eqref{basicFeq}, even in the case that $p = 2$ in \eqref{basicFeq} that we will take advantage of throughout the paper.  In the way of examples, we have
\[
x^{10}+x^7+x^6+x^5+x^3+x+2 
= (x + 1)^2 (x^2 + 1) (x^6 - 2x^5 + 2x^4 - x^3 + 2x^2 - 3x + 2)
\]
and
\begin{align*}
x^{19}+x^{18}+x^{16}&+x^{11}+x^7+x^5+x^3+x+2 \\[5pt]
&= (x^{10} + x^7 + x^5 + x^3 + 1) \,(x^9 + x^8 - x^5 - x^4 - x^3 + x + 2),
\end{align*}
where in the latter case $x^{10} + x^7 + x^5 + x^3 + 1$ is an irreducible reciprocal polynomial that is not cyclotomic.  Also, in the case that the constant term of $F(x)$ in \eqref{basicFeq} is replaced by a prime power, we observe that Theorem~\ref{factorizationthm} is no longer true as the examples
\begin{gather*}
x^{4} + 4 = (x^{2} + 2x + 2)(x^{2} - 2x + 2) \intertext{and} 
x^{4} + x^{2} + 25 = (x^2+3x+5)(x^2-3x+5)
\end{gather*}
demonstrate.

The remainder of this paper is organized as follows.  In the next section, we establish Theorem~\ref{factorizationthm}.  In Section~3 and Section~4, we provide a proof of Theorem~\ref{kpq:mainthm}.  In Section~5, we show that that the density of primes $p$, $q$ and $\ell$ satisfying $\ell < q < p \le X$ for which $\Phi_{\ell p}(x) + \Phi_{\ell q}(x)$ is irreducible tends to $7/12$ as $X \rightarrow \infty$.  In Section~6, we give some background regarding multiples of $\Phi_{n}(x)+\Phi_{m}(x)$ of the form \eqref{basicFeq}.  It is in this section that we will see $\Phi_{n}(x)\,\Phi_{m}(x)+1$ naturally arising.  In Section~7, we piece together a number of applications including proofs of Theorem~\ref{21case:thm}, Theorem~\ref{arbitrarynumprimes:thm} and Theorem~\ref{maingeneralthm2}.  We end the paper with some minor observations in Section~8.

\section{A proof of Theorem~\ref{factorizationthm}}\label{sect2}

Our approach will use an idea of Ljunggren \cite{ljunggren} and other ideas described in \cite{fil1999}, but first we record a lemma based on the first paragraph of Section~2 in \cite{filsol1999}.

\begin{lemma}\label{f=uvlem}
Suppose the non-reciprocal part of $f(x) \in \mathbb Z[x]$ is reducible.  Then there exist non-reciprocal polynomials $u(x)$ and $v(x)$ in $\mathbb Z[x]$ such that $f(x) = u(x) v(x)$.
\end{lemma}

Lemma~\ref{f=uvlem} may seem obvious, and its proof is not difficult.  However, one should note that if $f(x)$ has at least two irreducible non-reciprocal factors $w_{1}(x)$ and $w_{2}(x)$ in $\mathbb Z[x]$, one cannot simply take $u(x)$ and $v(x)$ so that $w_{1}(x) \mid u(x)$ and $w_{2}(x) \mid v(x)$.  In particular, if the reciprocals $\widetilde{w}_{1}(x)$ and $\widetilde{w}_{2}(x)$ of $w_{1}(x)$ and $w_{2}(x)$, respectively, also divide $f(x)$, then without more care, this construction of $u(x)$ and $v(x)$ could lead to $u(x)$ divisible by $w_{1}(x)\widetilde{w}_{1}(x)$ and $v(x)$ divisible by $w_{2}(x) \widetilde{w}_{2}(x)$, in which case both $u(x)$ and $v(x)$ could be reciprocal.  As indicated above, a correct argument for the proof of Lemma~\ref{f=uvlem} can be found in \cite[Section 2]{filsol1999}.

Suppose now that $F(x)$ is as in Theorem~\ref{factorizationthm} and has a reducible non-reciprocal part.  By Lemma~\ref{f=uvlem}, there exist non-reciprocal $u(x)$ and $v(x)$ in $\mathbb Z[x]$ such that $F(x) = u(x) v(x)$.  As $F(x)$ is monic, we can take $u(x)$ and $v(x)$ to be monic.  Since $F(x)$ cannot have positive real roots (as all its non-zero coefficients are positive), neither $u(x)$ nor $v(x)$ can have a positive real root.  As $u(x)$ and $v(x)$ are monic, we obtain that both $u(0)$ and $v(0)$ are positive.  
Since $F(0) = p$, we can take $u(0) = p$ and $v(0) = 1$.  
We define
\begin{equation}\label{wdef}
W(x) = u(x) \tilde{v}(x).
\end{equation}

\begin{lemma}\label{Winfo}
Suppose $F(x)$ in \eqref{basicFeq} has a reducible non-reciprocal part.
Then the polynomial $W(x)$ above satisfies $W(x) \ne F(x)$, $\deg W = \deg F = d_{r}$ and 
\begin{equation}\label{WWtildeFFtilde}
W(x)\,\widetilde{W}(x) = F(x)\,\widetilde{F}(x).
\end{equation}
Furthermore, like $F(x)$, the polynomial $W(x)$  consists of $r+1$ terms, with constant term $p$ and $r$ other terms each having coefficient $1$.  
\end{lemma}

\begin{proof}
Since $v(x)$ is non-reciprocal, we have $v(x) \ne \tilde{v}(x)$ so that $W(x) = u(x) \tilde{v}(x) \ne u(x) v(x) = F(x)$.  
The constant terms of $u(x)$ and $v(x)$ are non-zero, so $\deg \tilde{u} = \deg u$, $\deg \tilde{v} = \deg v$, and the reciprocals of $\tilde{u}(x)$ and $\tilde{v}(x)$ are $u(x)$ and $v(x)$, respectively.  The definition of $W(x)$ now implies $\deg W = \deg F = d_{r}$ and 
\[
W(x) \widetilde{W}(x) = \big( u(x) \tilde{v}(x) \big) \big( \tilde{u}(x) v(x) \big)
= \big( u(x) v(x) \big) \big( \tilde{u}(x) \tilde{v}(x) \big) = F(x) \widetilde{F}(x).
\]
The last part of the lemma will be based on some further properties of $W(x)$ that we develop next.

Since $v(1) = \tilde{v}(1)$, we deduce that $W(1) = F(1)$.
Also, by construction, $W(x)$ is monic and $W(0) = u(0) \tilde{v}(0) = p = F(0)$.
Thus, we can write 
\[
W(x) = x^{d_{r}} + \sum_{j=1}^{t} w_{j} x^{e_{j}} + p, 
\  \text{ where } t \in \mathbb Z^{+}, \text{ and } w_{j} \in \mathbb Z\backslash\{ 0 \} \text{ and } 1 \le e_{j} < d_{r} \text{ for each $j$},
\]
and $W(1) = F(1)$ implies
\[
\sum_{j=1}^{t} w_{j} = r-1.
\]
In \eqref{WWtildeFFtilde}, already established, we equate the values of the coefficients of $x^{d_{r}}$ on both sides.  In general, one can see that for a polynomial $A(x) \in \mathbb R[x]$ with $A(0) \ne 0$, the coefficient of $\deg A$ in $A(x) \widetilde{A}(x)$ is the sum of the squares of the coefficients of $A(x)$.  Using $W(x)$ and $F(x)$ are monic with the same constant terms, we deduce that
\[
\sum_{j=1}^{t} w_{j}^{2} = r-1.
\]
As $w_{j}^{2} \ge w_{j}$ for each $j$ and
\[
\sum_{j=1}^{t} w_{j}^{2} = \sum_{j=1}^{t} w_{j} = r-1,
\]
we deduce that $w_{j}^{2} = w_{j}$ for each $j$.  Since each $w_{j} \ne 0$, we obtain $w_{j} = 1$ for $j \in \{ 1, 2, \ldots, t \}$ and $t = r-1$.
The lemma follows.
\end{proof}

For the proof of Theorem~\ref{factorizationthm}, we present an argument given to us by Alexander Borisov (private communication) that simplified our original approach.

\begin{proof}[Proof of Theorem~\ref{factorizationthm}]
Assume the non-reciprocal part of $F(x)$ is reducible, and let $W(x)$ be as in Lemma~\ref{Winfo}. 
From \eqref{WWtildeFFtilde}, we see that
\begin{equation}\label{borisovpoint}
F(x)\big(\widetilde{W}(x)-\widetilde{F}(x)\big)
= \big(F(x)-W(x)\big)\widetilde{W}(x).
\end{equation}
Since $W(x) \ne F(x)$, both the left-hand side and right-hand side of \eqref{borisovpoint} are non-zero.  The properties of $W(x)$ in Lemma~\ref{Winfo} imply that the terms in $\widetilde{W}(x)-\widetilde{F}(x)$ and in $F(x)-W(x)$ all have coefficients that are $\pm 1$.  Furthermore, the lowest-degree term on the left-hand side of \eqref{borisovpoint} has coefficient $\pm p$ and the lowest-degree term on the right-hand side has coefficient $\pm 1$.  Hence, we obtain a contradiction, finishing the proof.
\end{proof}

\section{Simple applications of Theorem~\ref{factorizationthm}}

We begin this section by showing the polynomials in Conjecture~\ref{nicolconj}, in the case that $\phi(n) \ne \phi(m)$, cannot have irreducible reciprocal factors that are not cyclotomic. 

\begin{lemma}\label{recipfactf(x)}
If $w(x)$ is a reciprocal factor of $\Phi_{n}(x) + \Phi_{m}(x)$ for some positive integers $n$ and $m$ with $\phi(n) > \phi(m)$, then $w(x) \mid \big( x^{\phi(n) - \phi(m)} - 1 \big)$.  In particular, if $\phi(n) \ne \phi(m)$, then the irreducible reciprocal factors of $\Phi_{n}(x) + \Phi_{m}(x)$ are cyclotomic and appear with multiplicity $1$.
\end{lemma}

\begin{proof}
Since $w(x)$ is a reciprocal factor of $\Phi_{n}(x) + \Phi_{m}(x)$, we deduce that $w(x)$ also divides the reciprocal of $\Phi_{n}(x) + \Phi_{m}(x)$.  Given $\phi(n) > \phi(m)$, this reciprocal is
\begin{align*}
x^{\phi(n)} \big( \Phi_{n}(1/x) + \Phi_{m}(1/x) \big)
&= x^{\phi(n)}\Phi_{n}(1/x) + x^{\phi(n)-\phi(m)} \cdot x^{\phi(m)} \Phi_{m}(1/x) \\[5pt]
&= \Phi_{n}(x) + x^{\phi(n)-\phi(m)} \Phi_{m}(x).
\end{align*}
As $w(x)$ also divides $\Phi_{n}(x) + \Phi_{m}(x)$, we deduce that $w(x)$ divides
\[
\big( \Phi_{n}(x) + x^{\phi(n)-\phi(m)} \Phi_{m}(x) \big) - \big( \Phi_{n}(x) + \Phi_{m}(x) \big) = (x^{\phi(n) - \phi(m)} - 1) \Phi_{m}(x).
\]
Note that $\gcd(\Phi_{n}(x) + \Phi_{m}(x), \Phi_{m}(x)) = \gcd(\Phi_{n}(x), \Phi_{m}(x)) = 1$ since $\Phi_{n}(x)$ and $\Phi_{m}(x)$ are irreducible polynomials of different degrees.  Since  $w(x)$ divides $\Phi_{n}(x) + \Phi_{m}(x)$, we deduce that $\gcd(w(x), \Phi_{m}(x)) = 1$.  The conclusion that $w(x)$ divides $x^{\phi(n) - \phi(m)} - 1$ follows.  The second statement in the theorem follows from the first since $x^{\phi(n) - \phi(m)} - 1$ is a product of distinct cyclotomic polynomials and we can interchange the roles of $n$ and $m$.
\end{proof}

For later purposes in this paper, we establish here the following result, similar to Lemma~\ref{recipfactf(x)} but for the polynomials $\Phi_{n}(x) \,\Phi_{m}(x) + 1$.

\begin{lemma}\label{recipfactphinphimplus1}
If $w(x)$ is a reciprocal factor of $\Phi_{n}(x) \,\Phi_{m}(x) + 1$ for some positive integers $n$ and $m$, then $w(x) \mid \big( x^{\phi(n)+\phi(m)} - 1\big)$.  In particular, the irreducible reciprocal factors of $\Phi_{n}(x)\,\Phi_{m}(x)+1$ are cyclotomic and appear with multiplicity $1$.
\end{lemma}

\begin{proof}
The factor $w(x)$ is necessarily a factor also of the reciprocal of $\Phi_{n}(x) \,\Phi_{m}(x) + 1$ which, given $\Phi_{n}(x)$ and $\Phi_{m}(x)$ are reciprocal, is 
\[
x^{\phi(n)+\phi(m)} \big( \Phi_{n}(1/x) \,\Phi_{m}(1/x) + 1 \big)
= \Phi_{n}(x) \,\Phi_{m}(x) + x^{\phi(n)+\phi(m)}.
\]
Thus, $w(x)$ divides
\[
\big( \Phi_{n}(x) \,\Phi_{m}(x) + x^{\phi(n)+\phi(m)} \big) - \big( \Phi_{n}(x) \,\Phi_{m}(x) + 1 \big)
= x^{\phi(n)+\phi(m)} - 1,
\]
and the lemma follows.
\end{proof}

\begin{comment}
Lemma~\ref{recipfactphinphimplus1} implies that Conjecture~\ref{newconj} holds if and only if for all integers $n$ and $m$ exceeding $1$, the polynomial $\Phi_{n}(x) \,\Phi_{m}(x) + 1$ is a product of \textit{distinct} cyclotomic polynomials and an irreducible polynomial.  There is a similar implication of Lemma~\ref{recipfactf(x)} in the case of Conjecture~\ref{nicolconj}, but one has to take into account that Lemma~\ref{recipfactf(x)} only applies when $\phi(n) \ne \phi(m)$.  We do not know whether  cyclotomic factors of $\Phi_{n}(x) + \Phi_{m}(x)$ can occur with multiplicity $> 1$ when $\phi(n) = \phi(m)$.
\end{comment}

We state next some well-known identities for the cyclotomic polynomials $\Phi_{n}(x)$.  We omit the proofs but see, for example, \cite[Section 4.3]{berlekamp} and \cite[Exercise~5, p.~308]{hungerford}.

\begin{lemma}\label{lem:cyclotomic-properties}
For $n$ a positive integer, we have
\[
\Phi_{n}(x) = \prod_{d \mid n} \big( x^{d} - 1 \big)^{\mu(n/d)},
\qquad
\text{with }
\mu(n/d) = 
\begin{cases}
(-1)^{\omega(n/d)} &\text{if $n/d$ is squarefree}, \\
0 &\text{otherwise},
\end{cases}
\]
where, as in the introduction, $\omega(m)$ denotes the number of distinct prime divisors of a positive integer $m$.
If $p$ is a prime, then we also have
\[
\Phi_{pn}(x)=
\begin{cases}
\Phi_{n}(x^{p}) &\text{if } p \mid n, \\
\Phi_n(x^p)/\Phi_n(x) &\text{if } p\nmid n.
\end{cases}
\]
Furthermore, in the case that $p = 2$ and $2 \nmid n$, we have
$\Phi_{2n}(x) = \Phi_{n}(x^{2})/\Phi_{n}(x) = \Phi_{n}(-x)$.
\end{lemma}

Observe that Lemma~\ref{lem:cyclotomic-properties} implies the classical result that if $p$ is a prime then 
\[
\Phi_{p}(x) = (x^{p}-1)/(x-1) = x^{p-1} + \cdots + x + 1.
\]  
Furthermore, with $p$ and $\ell$ distinct primes, we see that $\Phi_{\ell p}(x) \,\Phi_{\ell}(x) = \Phi_{\ell}(x^{p})$.  The equation \eqref{introeq:ellpqeq} follows, and given Lemma~\ref{recipfactf(x)} and Theorem~\ref{factorizationthm}, the proof given in the introduction that the non-cyclotomic part of $\Phi_{\ell p}(x) + \Phi_{\ell q}(x)$ is irreducible when $p$, $q$ and $\ell$ are primes satisfying $p > q > \ell$ is complete.  From Lemma~\ref{lem:cyclotomic-properties}, we also see that, for every positive integer $t$, analogous to \eqref{introeq:ellpqeq}, we have
\begin{equation}\label{elltpq:eq}
\big( \Phi_{\ell^{t}p}(x) + \Phi_{\ell^{t}q}(x) \big) \Phi_{\ell^{t}}(x)
= \big( \Phi_{\ell p}\big(x^{\ell^{t-1}}\big) + \Phi_{\ell q}\big(x^{\ell^{t-1}}\big) \big) \Phi_{\ell}\big(x^{\ell^{t-1}}\big)
= \Phi_{\ell}\big(x^{\ell^{t-1}p}\big) + \Phi_{\ell}\big(x^{\ell^{t-1}q}\big)
\end{equation}
is in the form \eqref{basicFeq}, where once again $p$, $q$ and $\ell$ are primes satisfying $p > q > \ell$.  Furthermore, from $\Phi_{p}(x) = (x^{p}-1)/(x-1)$, for odd primes $p$ and $q$ satisfying $p > q$, we see that
\begin{equation}\label{2pq:identity}
\big( \Phi_{p}(-x) + \Phi_{q}(-x) \big) (x+1)
= x^{p} + x^{q} + 2
\end{equation}
is in the form \eqref{basicFeq}.  The non-cyclotomic part of $\Phi_{p}(x) + \Phi_{q}(x)$ is irreducible if and only if the non-cyclotomic part of $\Phi_{p}(-x) + \Phi_{q}(-x)$ is irreducible.  Thus, an application of Theorem~\ref{factorizationthm} with Lemma~\ref{recipfactf(x)} gives the first assertion in Theorem~\ref{kpq:mainthm}, namely the following.

\begin{theorem}\label{pqell:thm}
Let $p$, $q$ and $\ell$ are primes satisfying $p > q > \ell$, and let $t \ge 0$ be an integer.  Then the non-cyclotomic part of $\Phi_{\ell^{t} p}(x) + \Phi_{\ell^{t} q}(x)$ is irreducible.
\end{theorem}

This leads to

\begin{corollary}\label{pq2:cor}
Let $p$ and $q$ be distinct odd primes, and let $t \ge 0$ be an integer.  Then $\Phi_{2^{t} p}(x) + \Phi_{2^{t} q}(x)$ is irreducible.
\end{corollary}

\begin{proof}
Since $\Phi_{p}(x) + \Phi_{q}(x)$ is irreducible if and only if $\Phi_{p}(-x) + \Phi_{q}(-x) = \Phi_{2 p}(x) + \Phi_{2 q}(x)$ is irreducible, we may and do suppose $t \ge 1$.
Assume that there is a root of unity $\zeta$ for which $\Phi_{2^{t} p}(\zeta) + \Phi_{2^{t} q}(\zeta) = 0$.  Then Lemma~\ref{lem:cyclotomic-properties} implies
\[
\Phi_{2 p}\big(\zeta^{2^{t-1}}\big) + \Phi_{2 q}\big(\zeta^{2^{t-1}}\big) =
\Phi_{2^{t} p}(\zeta) + \Phi_{2^{t} q}(\zeta) = 0,
\]
which implies $\Phi_{2p}(x) + \Phi_{2q}(x) = \Phi_{p}(-x) + \Phi_{q}(-x)$ has a root which is a root of unity, namely $\zeta' = \zeta^{2^{t-1}}$.  From \eqref{2pq:identity}, we deduce that $\zeta'$ is a root of $x^{p} + x^{q} + 2$.  By the triangle inequality for complex numbers, we see that 
\[
2 \le \big| \zeta'{}^{p} + \zeta'{}^{q} \big| \le \big| \zeta'{}^{p} \big| + \big| \zeta'{}^{q} \big| = 1 + 1 = 2,
\]
so equality holds and we deduce that $\zeta'{}^{p} = \zeta'{}^{q} = -1$.  
There exist integers $u$ and $v$ such that $pu + qv = 1$, which implies then that
\[
\zeta' = \zeta'{}^{pu+qv} = (\zeta'{}^{p})^{u} (\zeta'{}^{q})^{v} = (-1)^{u+v} = \pm 1. 
\]
Since $1$ is not a root of $x^{p} + x^{q} + 2$, we obtain $\zeta' = -1$.  Recall that $\zeta'$ is a root of $\Phi_{p}(-x) + \Phi_{q}(-x)$, giving a contradiction as $\Phi_{p}(1) + \Phi_{q}(1) = p+q \ne 0$. 
We deduce that $\Phi_{2^{t} p}(x) + \Phi_{2^{t} q}(x)$ has no cyclotomic factors.  The corollary now follows from Theorem~\ref{pqell:thm}.
\end{proof}

Alexander Borisov has given a proof in \cite{nicol}, using Mahler measure, of the special case of Corollary~\ref{pq2:cor} that $\Phi_{p}(x)+\Phi_{q}(x)$ is irreducible for arbitrary distinct primes $p$ and $q$.  It is perhaps worth noting that there is a fairly simple proof of Corollary~\ref{pq2:cor} independent of Mahler measure and Theorem~\ref{factorizationthm}.  We view the arguments for Lemma~\ref{recipfactf(x)} and the argument in the proof of Corollary~\ref{pq2:cor} showing that $f(x) = \Phi_{2^{t} p}(x) + \Phi_{2^{t} q}(x)$ has no cyclotomic factors as such simple arguments, so what is left to justify is that the non-reciprocal part of $f(x)$ is irreducible without referring to Theorem~\ref{factorizationthm}.  As in the proof of Corollary~\ref{pq2:cor}, we see that if $\alpha$ is a root of $f(x)$, then $\beta = \alpha^{2^{t-1}}$ is a root of $x^{p} + x^{q} + 2$.  If $|\beta| = 1$, then the proof there leads to $\beta = -1$ is a root of $\Phi_{p}(-x) + \Phi_{q}(-x)$ and a contradiction.  The triangle inequality now implies that $|\beta| > 1$ since otherwise $|\beta| < 1$ and we obtain $2 = |\beta^{p}+\beta^{q}| \le |\beta|^{p} + |\beta|^{q} < 2$.  We deduce that every root $\alpha$ of $f(x) = \Phi_{2^{t} p}(x) + \Phi_{2^{t} q}(x)$ satisfies $|\alpha| > 1$.  Now, if we assume $f(x) = g(x)h(x)$ for some polynomials $g(x)$ and $h(x)$ in $\mathbb Z[x]$ of degree at least $1$, from $f(0) = 2$, we see that one of $g(0)$ or $h(0)$ is $\pm 1$.  Since $f(x)$ is monic, the leading coefficient of $g(x)$ is $\pm 1$ and the leading coefficient of $h(x)$ is $\pm 1$.  As a consequence, one of $g(x)$ or $h(x)$ has the absolute value of the product of its roots equal to $1$, contradicting that all the roots of $f(x)$, and hence $g(x)$ and $h(x)$, have absolute value $> 1$.  Hence, Corollary~\ref{pq2:cor} follows.

\section{The proof of Theorem~\ref{kpq:mainthm}}

In the previous section in Theorem~\ref{pqell:thm}, we showed that the non-cyclotomic part of the polynomial $\Phi_{\ell^{t} p}(x) + \Phi_{\ell^{t} q}(x)$ is irreducible given primes $p > q > \ell$ and $t$ a nonnegative integer.  Before addressing more applications of Theorem~\ref{factorizationthm} to Conjecture~\ref{nicolconj}, we address in this section a complete description of the cyclotomic factors of $\Phi_{\ell^{t} p}(x) + \Phi_{\ell^{t} q}(x)$ as given by Theorem~\ref{kpq:mainthm}.  For $t = 0$ and for $\ell = 2$, we already know from Corollary~\ref{pq2:cor} that $\Phi_{\ell^{t} p}(x) + \Phi_{\ell^{t} q}(x)$ is irreducible, which agrees with Theorem~\ref{kpq:mainthm}, so we only consider $t > 0$ and $\ell > 2$ here.

Throughout the rest of this section, $p$, $q$ and $\ell$ are primes satisfying $p > q > \ell > 2$ and $t$ is a positive integer.  We set $k = \ell^{t}$,
\[
f(x) = \Phi_{kp}(x) + \Phi_{kq}(x)
\qquad \text{and} \qquad
F(x) = f(x) \Phi_{k}(x).
\]
The next lemma follows directly from two applications of Lemma~\ref{lem:cyclotomic-properties}.

\begin{lemma}\label{lemF(x)}
Given the above notation, we have 
\begin{equation}\label{factorphik:eq1}
F(x) = \Phi_{k}(x^{p})+\Phi_{k}(x^{q})
= \Phi_{\ell}(x^{p \ell^{t-1}}) + \Phi_{\ell}(x^{q \ell^{t-1}}).
\end{equation}
\end{lemma}

\begin{lemma}\label{factorphik}
Given the notation above, $\Phi_{k}(x)$ does not divide $f(x)$.
\end{lemma}

\begin{proof}
Since $F(x) = f(x) \Phi_{k}(x)$, it suffices to show that $\Phi_{k}(x)^{2}$ does not divide $F(x)$.  Since $\Phi_{k}(x)$ is irreducible, we obtain $\gcd(\Phi_{k}(x), \Phi'_{k}(x)) = 1$.  The polynomial $\Phi_{k}(x)$ has all its roots on the unit circle.  The Gauss-Lucas theorem \cite[Section~2.1.6]{sheilsmall} states that, in the complex plane, the roots of the derivative of a polynomial lie inside the convex hull of the roots of the polynomial.  We deduce that the roots of $\Phi'_{k}(x)$ lie in $\{ z \in \mathbb C : |z| < 1 \}$.  In particular, we obtain that $\Phi'_{k}(x)$ cannot be $0$ at a root of unity.

Assume $\Phi_{k}(x) \mid f(x)$ or, equivalently, $\Phi_{k}(x)^{2} \mid F(x)$.  
The condition $p > q > \ell$ implies $\gcd(pq, k) = 1$.  
Thus, we see that 
\[
\phi(k p) - \phi(k q) = \phi(k) (\phi(p) - \phi(q)) =  \phi(k) (p-q) = \ell^{t-1} (\ell-1) (p-q).
\]
From Lemma~\ref{recipfactf(x)}, we obtain $\Phi_{k}(x)$ divides $x^{\ell^{t-1} (\ell-1) (p-q)} - 1$, so we deduce that $k$ must divide $\ell^{t-1} (\ell-1) (p-q)$.  As $k = \ell^{t}$ with $\ell$ prime, we see that $\ell \mid (p-q)$.  As a consequence, we have $\zeta_{\ell}^{p} = \zeta_{\ell}^{q}$, which we will take advantage of momentarily. 

From $\Phi_{k}(x)^{2} \mid F(x)$, we obtain $\Phi_{k}(x) \mid F'(x)$ so that $F'(\zeta_{k}) = 0$.  
From \eqref{factorphik:eq1}, we obtain
\[
F'(x) = p\ell^{t-1} x^{p\ell^{t-1}-1} \Phi'_{\ell}(x^{p\ell^{t-1}}) + q\ell^{t-1} x^{q\ell^{t-1}-1} \Phi'_{\ell}(x^{q\ell^{t-1}}).
\]
Since $k = \ell^{t}$, we obtain
\[
\zeta_{k}^{\ell^{t-1}} = e^{2 \pi i \ell^{t-1}/k} = e^{2 \pi i/\ell} = \zeta_{\ell}.
\]
Therefore,
\begin{align*}
F'(\zeta_{k}) &= p\ell^{t-1} \big( \zeta_{k}^{\ell^{t-1}} \big)^{p} \zeta_{k}^{-1} \Phi'_{\ell}\big(\big(\zeta_{k}^{\ell^{t-1}}\big)^{p}\big) + q\ell^{t-1} \big( \zeta_{k}^{\ell^{t-1}} \big)^{q} \zeta_{k}^{-1} \Phi'_{\ell}\big(\big(\zeta_{k}^{\ell^{t-1}}\big)^{q}\big) \\[5pt]
&= p\ell^{t-1} \zeta_{\ell}^{p} \zeta_{k}^{-1} \Phi'_{\ell}(\zeta_{\ell}^{p}) + q\ell^{t-1} \zeta_{\ell}^{q} \zeta_{k}^{-1} \Phi'_{\ell}(\zeta_{\ell}^{q}).
\end{align*}
Recalling $\zeta_{\ell}^{p} = \zeta_{\ell}^{q}$, we deduce $\Phi'_{\ell}(\zeta_{\ell}^{q}) = \Phi'_{\ell}(\zeta_{\ell}^{p})$ so that
\[
F'(\zeta_{k}) = \big( p \zeta_{\ell}^{p} + q \zeta_{\ell}^{q} \big) \ell^{t-1} \zeta_{k}^{-1} \Phi'_{\ell}(\zeta_{\ell}^{p}).
\]
At the beginning of the proof, we showed $\Phi'_{\ell}(x)$ cannot be $0$ at a root of unity.  Hence, the last factor $\Phi'_{\ell}(\zeta_{\ell}^{p})$ is non-zero.  As $|p \zeta_{\ell}^{p}| = p \ne q = |q \zeta_{\ell}^{q}|$, we also have $p \zeta_{\ell}^{p} + q \zeta_{\ell}^{q} \ne 0$.  We deduce then that $F'(\zeta_{k}) \ne 0$, giving a contradiction.
\end{proof}

\begin{corollary}\label{fiffF}
Let $d \in \mathbb Z^{+}$.  Then $\Phi_{d}(x) \mid f(x)$ if and only if $d \ne k$ and $\Phi_{d}(x) \mid F(x)$.
\end{corollary}

For the moment, we restrict our attention to $t = 1$, so $k = \ell$.  
We will make our analysis later for $t  > 1$ from this case.  Our aim in the next lemma is to show that with $k = \ell$, we necessarily have \ref{mainthmfirstitem} or \ref{mainthmseconditem} of Theorem~\ref{kpq:mainthm} if $\Phi_{d}(x) \mid f(x)$.  After the lemma, we will show that each of the conditions \ref{mainthmfirstitem} and \ref{mainthmseconditem} is sufficient for $\Phi_{d}(x) \mid f(x)$.  
The proof of our next lemma is long, but the result is essential.  Note that we elaborate on the main idea of the argument for this lemma after some background leading to \eqref{mnequatdboundproof} below.

\begin{lemma}\label{dbound}
Let $d$ be a positive integer such that $\Phi_{d}(x) \mid f(x)$, where $k = \ell$ is a prime.  Then $d \le 2\ell-2$.  Furthermore, either \ref{mainthmfirstitem} or \ref{mainthmseconditem} of Theorem~\ref{kpq:mainthm} holds.
\end{lemma}

\begin{proof}
Since $\Phi_{d}(x) \mid f(x)$, we have $\Phi_{d}(x) \mid F(x)$.  From Lemma~\ref{recipfactf(x)} with $t=1$, we know $d$ divides
\[
\phi(\ell p) - \phi(\ell q) = (\phi(p)-\phi(q))\phi(\ell) = (p-q) (\ell-1).
\]
Since $p > q > \ell$ and $\gcd(pq, p-q) = 1$, we deduce that $\gcd(pq, d) = 1$.  Let $\zeta = \zeta_{d} = e^{2 \pi i/d}$.  Then for each integer $j \in [1,d]$ with $\gcd(j,d) = 1$, the number $\zeta^{j}$ is a root of $\Phi_{d}(x)$ and, hence, $F(x)$.  
We choose $j \in [1,d]$ relatively prime to $d$, as we can, so that $qj \equiv 1 \pmod{d}$.  
We make use of \eqref{factorphik:eq1} with $t = 1$.  
Given $\zeta^{qj} = \zeta$ and $F(\zeta^{j}) = 0$, we see that
\begin{equation}\label{contradtobe}
0 = F(\zeta^{j}) = \Phi_{\ell}(\zeta^{pj}) + \Phi_{\ell}(\zeta^{qj}) = \Phi_{\ell}(\zeta^{pj}) + \Phi_{\ell}(\zeta)
= \sum_{u=0}^{\ell-1} \zeta^{p j u} + \sum_{v=0}^{\ell-1} \zeta^{v}.
\end{equation}
Set $a = pj$, and note that $\gcd(a,d) = 1$.

Let $\ell_{0}$ be the integer in $[0,d)$ satisfying $\ell_{0} \equiv \ell \pmod{d}$. 
By \eqref{factorphik:eq1}, we see that $F(x)$ has positive coefficients so that $F(1) \ne 0$, so $d \ne 1$. 
From Corollary~\ref{fiffF}, we have $d \ne \ell$ which now, since $\ell$ is prime, implies $\ell_{0} \ge 1$. 
Let $b$ be an integer relatively prime to $d$ so that $\zeta^{b}$ is a primitive $d^{\rm th}$ root of unity.  
For every integer $j' \ge 0$, we have
\[
\zeta^{bj'} + \zeta^{b(j'+1)} + \cdots + \zeta^{b(j'+d-1)} = \dfrac{\zeta^{bj'} (\zeta^{bd} - 1)}{\zeta^{b}-1} = 0.
\]
Thus, in \eqref{contradtobe}, with $b = a = pj$ in the first sum and $b = 1$ in the second sum, we can eliminate the terms where $u$ and $v$ are larger than $\ell_{0}-1$ since they will sum to $0$.  We obtain
\begin{equation}\label{newcontradtobe}
\sum_{u=0}^{\ell_{0}-1} \zeta^{a u} + \sum_{v=0}^{\ell_{0}-1} \zeta^{v} = 0.
\end{equation}
By the definition of $\ell_{0}$ and since $\ell_{0} \ge 1$, we have $1 \le \ell_{0} < d$. 

Let $\ell'_{1} = \lfloor (\ell_{0}-1)/2 \rfloor$ and $\ell'_{2} = \lceil (\ell_{0}-1)/2 \rceil$.  
Fix $j_{0} \in \{ 1, 2 \}$, and let $\ell' = \ell'_{j_{0}}$.  We divide by $\zeta^{\ell'}$ in \eqref{newcontradtobe} to deduce that
\[
\sum_{u=0}^{\ell_{0}-1} \zeta^{a u - \ell'} + \sum_{v=0}^{\ell_{0}-1} \zeta^{v-\ell'} = 0.
\]
Let $a_{u}$ be the integer in $[0,d)$ satisfying $a_{u} \equiv a u - \ell' \pmod{d}$.  
Since $\gcd(a,d) = 1$ and $\ell_{0} < d$, the values $a_{u}$ for $0 \le u \le \ell_{0}-1$ are distinct modulo $d$.  Similarly, the values of $v-\ell'$ are distinct modulo $d$ for $0 \le v \le \ell_{0}-1$.  On the other hand, we see that
\begin{equation}\label{mnequatdboundproof}
\sum_{v=0}^{\ell_{0}-1} \zeta^{v-\ell'} = -\sum_{u=0}^{\ell_{0}-1} \zeta^{a_{u}}.
\end{equation}

The role of $\ell'$ (whether $j_{0} = 1$ or $j_{0} = 2$) is that it is chosen so that the sum on the left of \eqref{mnequatdboundproof} has maximal real part over all sums of $\ell_{0}$ unequal powers of $\zeta$.  More precisely, the powers of $\zeta$ appearing in the left sum are the ones that lie on the unit circle as much to the right on the unit circle as possible.  The term coming from the choice $v = \ell'$ gives the summand $\zeta^{0} = 1$.  
If $\ell_{0}$ is odd, then $\ell' = (\ell_{0}-1)/2$, and we will also have $(\ell_{0}-1)/2$ terms coming from $\zeta^{v-\ell'}$ with $v > \ell'$ that are the $(\ell_{0}-1)/2$ right-most powers of $\zeta$ above the real axis and $(\ell_{0}-1)/2$ terms coming from $\zeta^{v-\ell'}$ with $v < \ell'$ that are the $(\ell_{0}-1)/2$ right-most powers of $\zeta$ below the real axis.  If $\ell_{0}$ is even and $j_{0} = 1$, then $\ell' = (\ell_{0}-2)/2$, and we will have $\ell_{0}/2$ terms coming from $\zeta^{v-\ell'}$ with $v > \ell'$ that are the $\ell_{0}/2$ right-most powers of $\zeta$ above the real axis and $(\ell_{0}/2)-1$ terms coming from $\zeta^{v-\ell'}$ with $v < \ell'$ that are the $(\ell_{0}/2)-1$ right-most powers of $\zeta$ below the real axis.  If $\ell_{0}$ is even and $j_{0} = 2$, then $\ell' = \ell_{0}/2$, and we will have $(\ell_{0}/2)-1$ terms coming from $\zeta^{v-\ell'}$ with $v > \ell'$ that are the $(\ell_{0}/2)-1$ right-most powers of $\zeta$ above the real axis and $\ell_{0}/2$ terms coming from $\zeta^{v-\ell'}$ with $v < \ell'$ that are the $\ell_{0}/2$ right-most powers of $\zeta$ below the real axis.  In each of these cases, the sum on the left of \eqref{mnequatdboundproof} has maximal real part that the sum of $\ell_{0}$ distinct powers of $\zeta$ on the unit circle can have.  As the sum on the right of \eqref{mnequatdboundproof} is also a sum of $\ell_{0}$ distinct powers of $\zeta$, we obtain a severe restriction on the values of $a_{u}$. 

We can be more explicit on the value of the real part of the left sum in \eqref{mnequatdboundproof}.  If $\ell_{0}$ is odd so $\ell_{0} = 2\ell'+1$, then we use
\[
z + z^{2} + \cdots + z^{\ell'} = \dfrac{z (z^{\ell'}-1)}{z - 1}  = \dfrac{z^{1/2} (z^{\ell'}-1)}{z^{1/2} - z^{-1/2}}
\]
and take $z = e^{\pm i 2\pi/d}$ to rewrite the left sum in \eqref{mnequatdboundproof} as
\begin{align*}
1 + \sum_{v'=1}^{\ell'} (e^{i 2\pi v'/d} + e^{-i2\pi v'/d}) &= 1 + \dfrac{e^{i \pi/d} (e^{i 2\pi \ell'/d}-1)}{e^{i \pi/d}-e^{-i \pi/d}}
+ \dfrac{e^{-i \pi/d} (e^{-i 2\pi \ell'/d}-1)}{e^{-i \pi/d}-e^{i \pi/d}} \\[5pt]
&= 1 + \dfrac{e^{i \pi (2\ell'+1)/d}-e^{-i \pi (2\ell'+1)/d}-e^{i \pi/d}+e^{-i \pi/d}}{e^{i \pi/d}-e^{-i \pi/d}} \\[5pt]
&= 1 +  \dfrac{\sin(\pi (2\ell'+1)/d) - \sin(\pi/d)}{\sin(\pi/d)} = \dfrac{\sin(\pi \ell_{0}/d)}{\sin(\pi/d)}.
\end{align*}
With $\ell_{0}$ even, we use the same evaluation of the last sum above with $\ell'$ replaced by $(\ell_{0}/2) - 1$ and note the real part of $\zeta^{\pm \ell_{0}/2}$ is $\cos(\pi \ell_{0}/d)$ to deduce that the real part of the left sum in \eqref{mnequatdboundproof} is
\[
1 + \cos(\pi \ell_{0}/d) + \sum_{v'=1}^{(\ell_{0}/2) - 1} (e^{i 2\pi v'/d} + e^{-i2\pi v'/d})
= 1 + \cos(\pi \ell_{0}/d) + \dfrac{\sin(\pi (\ell_{0}-1)/d)}{\sin(\pi/d)}.
\]
Thus, we deduce that 
\begin{equation}\label{realpartleftsum}
\real\bigg( \sum_{v=0}^{\ell_{0}-1} \zeta^{v-\ell'} \bigg)
= \begin{cases}
\dfrac{\sin(\pi \ell_{0}/d)}{\sin(\pi/d)} &\text{if $\ell_{0}$ is odd}, \\[9pt]
1 + \cos(\pi \ell_{0}/d) + \dfrac{\sin(\pi (\ell_{0}-1)/d)}{\sin(\pi/d)} &\text{if $\ell_{0}$ is even}.
\end{cases}
\end{equation}

Next, we look at how large the real part of the right-hand side of \eqref{mnequatdboundproof} can be, or equivalently how small the sum on the right of \eqref{mnequatdboundproof} can be.
Suppose first that $d$ is odd.  Then $\zeta^{a_{u}} = e^{2\pi i a_{u}/d} \ne -1$ no matter what the integer value of $a_{u}$ is.  As the left side of \eqref{mnequatdboundproof} has a term $\zeta^{0} = 1$, this suggests that perhaps the real part of the right side of \eqref{mnequatdboundproof} cannot be as large as the real part of the left side.  This however requires some justification.  We bound the real part of the right side of \eqref{mnequatdboundproof} by taking the distinct powers of $\zeta$ appearing as terms as those powers of $\zeta$ that are as left on the unit circle as possible.  

We begin by showing that, with $d$ odd, we cannot have $\ell_{0}$ even.  
So suppose $\ell_{0}$ is even.  Then with $d = 2d'-1$ odd, this gives us
\begin{align*}
\real \bigg( -\sum_{u=0}^{\ell_{0}-1} \zeta^{a_{u}} \bigg)
&\le \real \bigg( - \sum_{j=0}^{(\ell_{0}/2)-1} \big( e^{i 2\pi (d'+j)/d} +  e^{-i 2\pi (d'+j)/d} \big)\bigg) \\[5pt]
&= - \sum_{j=0}^{(\ell_{0}/2)-1} \big( e^{i 2\pi (d'+j)/d} +  e^{-i 2\pi (d'+j)/d} \big).
\end{align*}
We can evaluate this sum in a similar manner as we did the left sum in \eqref{mnequatdboundproof}.   We obtain
\begin{align*}
\sum_{j=0}^{(\ell_{0}/2)-1} \big( e^{i 2\pi (d'+j)/d} +  e^{-i 2\pi (d'+j)/d} \big)
&= e^{i 2\pi d'/d} \sum_{j=0}^{(\ell_{0}/2)-1} e^{i 2\pi j/d}  + e^{-i 2\pi d'/d} \sum_{j=0}^{(\ell_{0}/2)-1} e^{-i 2\pi j/d} \\[5pt]
&= \dfrac{e^{i 2\pi d'/d} (e^{i \pi \ell_{0}/d} - 1)}{e^{i 2\pi /d} - 1} 
+  \dfrac{e^{-i 2\pi d'/d} (e^{-i \pi \ell_{0}/d} - 1)}{e^{-i 2\pi /d} - 1} \\[5pt]
&= \dfrac{e^{i \pi (2d'-1)/d} (e^{i \pi \ell_{0}/d} - 1)}{e^{i \pi /d} - e^{-i \pi /d}} 
+ \dfrac{e^{-i \pi (2d'-1)/d} (e^{-i \pi \ell_{0}/d} - 1)}{e^{-i \pi /d}-e^{i \pi /d}} \\[5pt]
&= \dfrac{-(e^{i \pi \ell_{0}/d} - 1)}{e^{i \pi /d} - e^{-i \pi /d}} 
+ \dfrac{(e^{-i \pi \ell_{0}/d} - 1)}{e^{i \pi /d}-e^{-i \pi /d}} \\[5pt]
&= \dfrac{-\sin(\pi \ell_{0}/d)}{\sin(\pi/d)}.
\end{align*}
We deduce from \eqref{mnequatdboundproof} and \eqref{realpartleftsum} that 
\[
1 + \cos(\pi \ell_{0}/d) + \dfrac{\sin(\pi (\ell_{0}-1)/d)}{\sin(\pi/d)}
\le \dfrac{\sin(\pi \ell_{0}/d)}{\sin(\pi/d)}.
\]
Since 
\[
\sin(\pi \ell_{0}/d) = \sin(\pi (\ell_{0}-1)/d) \cos(\pi/d) + \sin(\pi/d) \cos(\pi (\ell_{0}-1)/d),
\]
we see that 
\[
\dfrac{\sin(\pi \ell_{0}/d)}{\sin(\pi/d)} 
= \dfrac{\sin(\pi (\ell_{0}-1)/d) \cos(\pi/d)}{\sin(\pi/d)} + \cos(\pi (\ell_{0}-1)/d).
\]
As $d > \ell_{0}$ and $\ell_{0}$ is even so $\ell_{0} \ge 2$, the sine values are positive.  Since $d > 1$ and odd, we deduce that $\cos(\pi/d) < 1$.  We therefore can establish that $d$ cannot be odd if $\ell_{0}$ is even if we can show
\[
1 + \cos(\pi \ell_{0}/d) \ge \cos(\pi (\ell_{0}-1)/d).
\]
We are done if $\cos(\pi (\ell_{0}-1)/d) < 0$, so we consider $\cos(\pi (\ell_{0}-1)/d) \ge 0$.
The above inequality is equivalent to
\[
1 + \cos(\pi (\ell_{0}-1)/d) \cos(\pi/d) - \sin(\pi (\ell_{0}-1)/d) \sin(\pi/d) \ge \cos(\pi (\ell_{0}-1)/d).
\]
As $d > 1$, we have $d \ge 3$ so that
\[
\cos(\pi/d) \ge \cos(\pi/3) = \dfrac{1}{2}
\qquad \text{and} \quad
\sin(\pi/d) \le \sin(\pi/3) = \dfrac{\sqrt{3}}{2}.
\]
It suffices therefore to show
\[
1 + \dfrac{\cos(\pi (\ell_{0}-1)/d)}{2} \ge \cos(\pi (\ell_{0}-1)/d) + \dfrac{\sqrt{3}\,\sin(\pi (\ell_{0}-1)/d)}{2}
\]
or equivalently
\[
\dfrac{\cos(\pi (\ell_{0}-1)/d)}{2} + \dfrac{\sqrt{3}\,\sin(\pi (\ell_{0}-1)/d)}{2} \le 1.
\]
The left-hand side of this last inequality is
\[
\sin(\pi/6) \,\cos(\pi (\ell_{0}-1)/d) + \sin(\pi (\ell_{0}-1)/d) \,\cos(\pi/6)
= \sin\bigg( \dfrac{\pi}{6} + \dfrac{\pi (\ell_{0}-1)}{d} \bigg) \le 1,
\]
so we deduce that, with $d$ odd, we cannot have $\ell_{0}$ is even.

Next, we show that with $d$ odd, we cannot have that $\ell_{0}$ is odd by completing a similar analysis.  In this case, by taking the distinct values of $\zeta^{a_{u}}$ as left on the unit circle as possible and using the previous notation $\ell' = (\ell_{0}-1)/2$ and $d = 2d'-1$, we see that
\begin{align*}
\real \bigg( -\sum_{u=0}^{\ell_{0}-1} \zeta^{a_{u}} \bigg)
&\le - \sum_{j=0}^{\ell'-1} \big( e^{i 2\pi (d'+j)/d} +  e^{-i 2\pi (d'+j)/d} \big) - \real \big(e^{i 2\pi (d'+\ell')/d}\big) \\[5pt]
&= - \sum_{j=0}^{\ell'-1} \big( e^{i 2\pi (d'+j)/d} +  e^{-i 2\pi (d'+j)/d} \big) - \cos\big( 2\pi (d'+\ell')/d \big) \\[5pt]
&= \dfrac{\sin(2\pi \ell'/d)}{\sin(\pi/d)} - \cos\big( 2\pi (d'+\ell')/d \big) \\[5pt]
&= \dfrac{\sin(\pi (\ell_{0}-1)/d)}{\sin(\pi/d)} - \cos\big( \pi (d+\ell_{0})/d \big) \\[5pt]
&= \dfrac{\sin(\pi (\ell_{0}-1)/d)}{\sin(\pi/d)} + \cos\big( \pi \ell_{0}/d \big).
\end{align*}
From \eqref{mnequatdboundproof} and \eqref{realpartleftsum}, we deduce that
\[
\dfrac{\sin(\pi \ell_{0}/d)}{\sin(\pi/d)} \le \dfrac{\sin(\pi (\ell_{0}-1)/d)}{\sin(\pi/d)} + \cos\big( \pi \ell_{0}/d \big).
\]
As $\sin(\pi (\ell_{0}-1)/d) = \sin(\pi \ell_{0}/d) \,\cos(\pi /d) - \sin(\pi /d) \,\cos(\pi \ell_{0}/d)$, we obtain
\[
\dfrac{\sin(\pi \ell_{0}/d)}{\sin(\pi/d)} \le \dfrac{\sin(\pi \ell_{0}/d) \,\cos(\pi /d)}{\sin(\pi/d)},
\]
which is impossible since $d > \ell_{0}$ implies the trigonometric values shown are positive and since $\cos(\pi /d) < 1$.  Thus, we see that if $d$ is odd, we cannot have $\ell_{0}$ odd.  Therefore, $d$ cannot be odd.

We consider now the case that $d$ is even, say $d = 2d'$.  In this case, we have that $\zeta^{d'} = -1$.  The analysis in this case is similar to how we treated the left sum in \eqref{mnequatdboundproof}.  In fact, we can see by comparing real and then imaginary parts that, for \eqref{mnequatdboundproof} to hold, we must have
\[
\{ \zeta^{a_{u}} :  0 \le u \le \ell_{0}-1 \} = \{ -\zeta^{v-\ell'} :  0 \le v \le \ell_{0}-1 \}.
\]
Equivalently, since $\zeta^{d'} = -1$, we obtain that modulo $d$, the $\ell_{0}$ distinct values of the $a_{u}$ in some order are the same as the $\ell_{0}$ distinct values of $v-\ell'+d'$.
Recalling 
\[
a_{u} \equiv au - \ell' \equiv pju - \ell' \equiv p q^{-1} u - \ell' \pmod{d},
\]
where $q^{-1}$ refers to the inverse of $q$ modulo $d$, we see that the set of the $\ell_{0}$ distinct values of $p q^{-1} u$ modulo $d$ is the same as the set of the $\ell_{0}$ distinct values of $v+d'$ modulo $d$.  Furthermore, if these sets of values are the same, then necessarily \eqref{mnequatdboundproof} holds and \eqref{contradtobe} will follow as well as that $\Phi_{d}(x)$ divides $F(x)$.  

To formalize this discussion better, for arbitrary integers $A$ and $D$ with $D > 0$, we define $A \mmod D$ to be the unique integer $B \in [0,D)$ such that $B \equiv A \pmod{D}$.  With this notation, we have reduced solving \eqref{contradtobe} with $d = 2d'$ to determining when
\begin{equation}\label{setseqmodd}
\{ (p q^{-1} u) \mmod d : 0 \le u \le \ell_{0}-1 \} = \{ (v + d') \mmod d : 0 \le v \le \ell_{0}-1 \}.
\end{equation}
As an immediate consequence, we see that $\ell_{0} \ne 1$, implying $d \nmid (\ell-1)$.
Also, $d'$ belongs to the set on the right and, hence, to the set on the left.  Thus, $p q^{-1} u_{0} \equiv d' \pmod{d}$ for some $u_{0} \in (0,\ell_{0}-1]$.  Recall $d' \mid d$ and $p$ and $q^{-1}$ are relatively prime to $d$.  Since $d \mid (p q^{-1} u_{0} - d')$, we deduce that $d' \mid u_{0}$.  As $2d' = d > \ell_{0} > u_{0}$, we obtain $u_{0} = d'$.  Since $u_{0} \le \ell_{0}-1$, we see then that $d' \le \ell_{0}-1$.  Thus, $d' < \ell_{0} < d = 2d'$.

Since $d$ is even and $\ell$ is an odd prime such that $d \mid (\ell-\ell_{0})$, we know that $\ell_{0}$ is odd.  Recall $\ell_{0} \ne 1$.  
Therefore, there are positive integers $s_{1}$ and $s_{2}$ for which 
\[
2^{s_{1}} \ediv (\ell_{0}-1)
\qquad \text{and} \qquad
2^{s_{2}} \ediv d.
\]
Let $s = \min\{ s_{1}, s_{2} \}$.  Observe that since $2^{s} \mid d$, if some integer $c$ is divisible by $2^{s}$, then $c \mmod d$ is an integer that is divisible by $2^{s}$.  Hence, by taking $u \in \{ 0, 2^{s}, \ldots, \ell_{0}-1 \}$ (multiples of $2^{s}$), we see that there are $(\ell_{0}-1)/2^{s} + 1$ multiples of $2^{s}$ on the left-hand side of \eqref{setseqmodd}.  Thus, there must be $(\ell_{0}-1)/2^{s} + 1$ multiples of $2^{s}$ on the right-hand side of \eqref{setseqmodd}.  Since the choices of $v$ for these must be spaced $2^{s}$ apart, this is only possible if 
\[
2^{s} \mid d'
\qquad \text{and, hence,} \qquad
2^{s+1} \mid d.
\]  
We deduce then that $s = s_{1}$.  
Since $(\ell-1) - (\ell_{0}-1)$ is divisible by $d$ and $2^{s} \ediv (\ell_{0}-1)$, we deduce that 
\[
2^{s} \ediv (\ell-1).
\]  
Hence, for $s$ such that $2^{s} \ediv (\ell-1)$, we have $2^{s+1} \mid d$.  In particular, we now know $d \ge 4$ and $d' \ge 2$.
We define
\[
a' = a \mmod d = pq^{-1} \mmod d.
\]  
Since $\gcd(a,d) = 1$, we have $\gcd(a',d) = 1$.  
Recall $d' < \ell_{0} < d = 2d'$.  

Suppose $\ell_{0} = d'+1$.  Then the right side of \eqref{setseqmodd} are the $d'$ integers $d', d'+1, \ldots, d-1$ together with $0$.  
The left side includes $0$ and the integers $a' u \mmod d$ for $1 \le u \le d'$.  
In particular, we must have $a' \in [d',d-1]$.  Since $\gcd(a',d) = 1$, we know $a' \ne d'$.  
Let $a'' = d-a'$.  Since $a' \in [d'+1,d-1]$, we have $a'' \in [1,d'-1]$.  Also, $\gcd(a'',d) = 1$. 

Assume $a'' \ge 2$.  Then the number $u' = \lceil d'/a'' \rceil$ is at most $(d'/2)+1 \le d'$.  In this case, since $\gcd(a'',d) = 1$, we have $u' a'' \ge d'+1$.  Also, $u' a'' < ((d'/a'')+1) a'' = d' + a'' \le 2d' - 1$.  As a consequence, taking $u = u'$, we see that 
\[
a' u \mmod d = (d-a'') u' \mmod d = d - a'' u' \in [1,d'-1]. 
\]
This is impossible as \eqref{setseqmodd} implies $a' u \mmod d \in [d', d-1]$.  Hence, we deduce that $a'' = 1$.

Since $a'' = d-a'$, we obtain $a' = d-1$.  As $a' \equiv pq^{-1} \pmod{d}$, we obtain $p+q$ is divisible by $d$.  Recall we are in the case that $\ell_{0} = d'+1$.  Since $d$ and hence $d'$ divides $\ell-\ell_{0} = \ell-1-d'$, we deduce that $d' \mid (\ell-1)$.  Summarizing, for $\ell_{0} = d'+1$, we have a positive integer $s$ such that $2^{s} \ediv (\ell-1)$, $d = 2^{s+1} d_{o}$ for some odd divisor $d_{o}$ of $\ell-1$ (since $d' = 2^{s} d_{o}$ divides $\ell-1$), and $q \equiv -p \pmod{d}$.  In other words, in this case, Theorem~\ref{kpq:mainthm}~\ref{mainthmfirstitem} holds.
We are now left with considering $d' + 1 < \ell_{0} < d = 2d'$.  

Suppose now $\ell_{0} = d-1$.  
As $4 \mid d$ and $\ell \equiv \ell_{0} \equiv -1 \pmod{d}$, we deduce that $\ell \equiv 3 \pmod{4}$.  Observe that $d$ is a multiple of $4$ that divides $\ell+1$.  Recall $\gcd(pq,d) = 1$ and $d = 2d'$.  To show that \ref{mainthmseconditem} holds, we need only show that $q \equiv p + d' \pmod{d}$.  We will establish that $a' = pq^{-1} \mmod d = d'+1$.  Once this has been done, \ref{mainthmseconditem} follows, as $q \equiv p + d' \pmod{d}$ if and only if 
\[
pq^{-1} \mmod d = (1-d'q^{-1}) \mmod d = (1-d') \mmod d = d' + 1,
\]
where we have used that $q^{-1} \mmod d$ is relatively prime to $d$ and therefore necessarily odd.

With $\ell_{0} = d-1$, we see that the right-hand side of \eqref{setseqmodd} contains every integer in $\{ 0, 1, \ldots, d-1 \}$ except $(d-1+d') \mmod d = d'-1$.  Comparing with the left-hand side of \eqref{setseqmodd}, we deduce that for $u = d-1$, we must have $p q^{-1} u \mmod d = d'-1$.  In other words, we see that
\[
d'-1 \equiv p q^{-1} (d-1) \equiv - p q^{-1} \pmod{d}.
\]
Hence, we obtain
\[
a' = pq^{-1} \mmod d \equiv (-d'+1) \mmod d = d' + 1,
\]
and Theorem~\ref{kpq:mainthm}~\ref{mainthmseconditem} follows.

To finish the proof, we need to justify that we cannot have $d' + 2 \le \ell_{0} \le d-2$. 
With the plan of obtaining a contradiction, we assume $d' + 2 \le \ell_{0} \le d-2$. 
Since $d' = d/2$, we are done unless $(d/2)+2 \le d - 2$ which implies $d \ge 8$ and, hence, $d' \ge 4$ and $\ell_{0} \ge 6$.
Recall $a' = pq^{-1} \mmod d$.  We work instead with 
\[
b = (-a') \mmod d = (-pq^{-1}) \mmod d.
\]
The residue classes represented modulo $d$ on each side of \eqref{setseqmodd} are distinct, and the idea is to work with $b$ to say something about the residue classes modulo $d$ that are not represented there.  As $\gcd(b,d) = 1$, comparing those not represented on both sides, we obtain
\begin{align*}
\{ (bu) \mmod d : 1 \le u \le d-\ell_{0} \} 
&= \{ (-a'u) \mmod d : 1 \le u \le d-\ell_{0} \} \\
&= \{ (a'u) \mmod d : \ell_{0} \le u \le d-1 \} \\
&= \{ (v + d') \mmod d : \ell_{0} \le v \le d-1 \} \\
&= \{ (v - d') \mmod d : \ell_{0} \le v \le d-1 \} \\
&= \{ \ell_{0}-d', \ell_{0}-d'+1, \ldots, d'-1 \}.
\end{align*}
We successively consider $u$ on the left noting that we deduce from the above that
\begin{equation}\label{bueq}
\ell_{0} - d' \le (bu) \mmod d \le d'-1, 
\qquad \text{for } 1 \le u \le d-\ell_{0}.
\end{equation}
We see from $u = 1$ and $0 \le b \le d-1$ that $\ell_{0} - d' \le b \le d'-1$.  
By way of induction, suppose we know for some integer $u = [1, d-\ell_{0}-1]$ that $b \le (d'-1)/u$, as is the case for $u = 1$.  Then
\[
b(u+1) \le \dfrac{(u+1)(d'-1)}{u} \le 2(d'-1) < d.
\]
Therefore, $b(u+1) \mmod d = b(u+1)$, and we deduce from \eqref{bueq} that
$b(u+1) \le d'-1$ so that $b \le (d'-1)/(u+1)$.  
This finishes our induction, and we obtain $b \le (d'-1)/u$ for all $u = [1, d-\ell_{0}]$. 
In particular, with $u = d-\ell_{0}$, we see that
\[
b \le \dfrac{d'-1}{d-\ell_{0}}.
\]
Recall when we took $u = 1$ in \eqref{bueq} we obtained $b \ge \ell_{0} - d'$.  Therefore, we see that
\[
\ell_{0} - d' \le \dfrac{d'-1}{d-\ell_{0}},
\]
which, given $d = 2d' > 0$, can be written as
\[
d^{2} - (3\ell_{0}-1)d + 2\ell_{0}^{2}-2 \ge 0.
\]
The left side of this inequality factors, and we have
\begin{equation}\label{dell0:ineq}
(d - 1 - \ell_{0})(d + 2 - 2\ell_{0}) \ge 0.
\end{equation}
Observe that $d' + 2 \le \ell_{0} \le d-2$ implies
\[
d + 2 - 2\ell_{0} < 2d' + 4 - 2\ell_{0} \le 0
\qquad \text{and} \qquad
d - 1 - \ell_{0} > d - 2 - \ell_{0} \ge 0,
\]
contradicting \eqref{dell0:ineq}.  It is worth noting that $\ell_{0} = d'+1$ and $\ell_{0} = d-1$ previously considered lead to $0$ on the left-hand side of \eqref{dell0:ineq}, so no contradiction would come from \eqref{dell0:ineq} in those cases.

We have finished proving that \ref{mainthmfirstitem} or \ref{mainthmseconditem} must hold.  In the case of \ref{mainthmfirstitem}, we see that $d \le 2(\ell-1) = 2\ell-2$.  In the case of \ref{mainthmseconditem}, we obtain $d \le \ell+1 \le 2\ell - 2$ provided $\ell \ge 3$.  In this case, though, $\ell \equiv 3 \pmod{4}$, so we have $\ell \ge 3$.  Thus, the proof of Lemma~\ref{dbound} is now complete.
\end{proof}

We now turn to showing that in the case of \ref{mainthmfirstitem} and \ref{mainthmseconditem} in Lemma~\ref{dbound}, the cyclotomic polynomial $\Phi_{d}(x)$ divides $\Phi_{\ell p}(x) + \Phi_{\ell q}(x)$.

\begin{lemma}\label{largestcyclfactor}
Let $\ell$ be an odd prime, and let $p$ and $q$ be odd primes such that $p > q > \ell$.
Let $s \in \mathbb Z^{+}$ such that $2^{s} \ediv (\ell-1)$, 
and let $d_{o}$ be an odd divisor of $\ell-1$. 
Let $d = 2^{s+1} d_{o}$, and suppose $q \equiv -p \pmod{d}$.  Then $\Phi_{d}(x)$ divides $\Phi_{\ell p}(x) + \Phi_{\ell q}(x)$.
\end{lemma}

\begin{proof}
By Corollary~\ref{fiffF} with $k = \ell$, it suffices to show under the conditions of the lemma that $\Phi_{d}(x)$ divides $F(x) = \Phi_{\ell}(x^{p}) + \Phi_{\ell}(x^{q})$.  
Let $\zeta = \zeta_{d}$.  Since $p$ is an odd prime and $p > \ell$, we have $\gcd(p, d) = 1$. Therefore, there is a $j \in \mathbb Z$ such that $p j \equiv 1 \pmod{d}$.  Necessarily, we have $\gcd(j, d) = 1$ so that $\zeta^{j}$ is a root of $\Phi_{d}(x)$, and it suffices to show $F(\zeta^{j}) = 0$.  
Now, $q \equiv -p \pmod{d}$ implies $\zeta^{q} = \zeta^{-p}$.  Therefore, by the definition of $j$, we obtain 
\[
F(\zeta^{j}) = \Phi_{\ell}(\zeta^{pj}) + \Phi_{\ell}(\zeta^{-pj}) 
= \Phi_{\ell}(\zeta) + \Phi_{\ell}(\zeta^{-1}).
\]
Since $\Phi_{\ell}(x)$ is reciprocal and of degree $\ell-1$, we have $\Phi_{\ell}(x^{-1}) = x^{-(\ell-1)}\Phi_{\ell}(x)$ so that
\[
F(\zeta^{j}) = \Phi_{\ell}(\zeta) + \zeta^{1-\ell}\Phi_{\ell}(\zeta) = \Phi_{\ell}(\zeta) \zeta^{1-\ell} \big( \zeta^{\ell-1} + 1 \big).
\]
The conditions in the lemma imply $d/2$ divides $\ell-1$ and $(\ell-1)/(d/2)$ is odd.  Hence, we see that $x^{d/2} + 1$ is a factor of $x^{\ell-1} + 1$.  As $x^{d}-1 = (x^{d/2} - 1)(x^{d/2} + 1)$, we deduce that $\Phi_{d}(x) \mid (x^{d/2} + 1)$, implying $\zeta^{\ell-1} + 1 = 0$.  Thus, we obtain $F(\zeta^{j}) = 0$, and the result follows.
\end{proof}

\begin{lemma}\label{othercyclfactors}
Let $\ell \equiv 3 \pmod{4}$ be a prime, and let $p$ and $q$ be odd primes such that $p > q > \ell$.  Let $d$ be a multiple of $4$ that divides $\ell+1$, and suppose $q \equiv p + (d/2) \pmod{d}$.  Then $\Phi_{d}(x)$ divides $\Phi_{\ell p}(x) + \Phi_{\ell q}(x)$.
\end{lemma}

\begin{proof}
As in the proof of Lemma~\ref{largestcyclfactor}, we want to show $\Phi_{d}(x)$ divides $F(x) = \Phi_{\ell}(x^{p}) + \Phi_{\ell}(x^{q})$, we take $\zeta = \zeta_{d}$, and we use that $\gcd(p, d) = 1$.  Observe that $\zeta^{q} = \zeta^{d/2}\zeta^{p} = -\zeta^{p}$ and 
\[
\Phi_{\ell}(x) + \Phi_{\ell}(-x) = 2\big( x^{\ell-1} + x^{\ell-3} + \cdots + x^{2} + 1 \big)
= \dfrac{2 (x^{\ell+1} - 1)}{x^{2}-1}.
\]
Since $d$ divides $(\ell+1)p$ but not $2p$, we deduce
\[
F(\zeta) = \Phi_{\ell}(\zeta^{p}) + \Phi_{\ell}(\zeta^{q}) 
= \Phi_{\ell}(\zeta^{p}) + \Phi_{\ell}(-\zeta^{p})
= \dfrac{2 (\zeta^{(\ell+1)p} - 1)}{\zeta^{2p}-1} = 0.
\]
The lemma follows.
\end{proof}

We now address the case that $k = \ell^{t}$ for some integer $t \ge 1$.  Once the following lemma is established, Theorem~\ref{kpq:mainthm} follows.

\begin{lemma}\label{Phi_{ell}_suffices}
Let $p$, $q$ and $\ell$ be odd primes such that $p > q > \ell$, and let $k = \ell^{t}$ for some integer $t \ge 1$.  Let $D$ be a positive integer, and define $j$ by $D = d \ell^{j}$, where $j \ge 0$ is an integer, $d \in \mathbb Z^{+}$ and $\ell \nmid d$.  Then $\Phi_{D}(x)$ divides $\Phi_{k p}(x) + \Phi_{k q}(x)$ if and only if $j \le t-1$ and $\Phi_{d}(x)$ divides $\Phi_{\ell p}(x) + \Phi_{\ell q}(x)$.
\end{lemma}

\begin{proof}
Suppose first that $\Phi_{D}(x)$ divides $\Phi_{k p}(x) + \Phi_{k q}(x)$.  
Assume $j \ge t$.  Then Lemma~\ref{lem:cyclotomic-properties} implies $\Phi_{D}(x) = \Phi_{d \ell^{j}}(x)$ divides $\Phi_{k p}(x) + \Phi_{k q}(x) = \Phi_{\ell p}(x^{\ell^{t-1}}) + \Phi_{\ell q}(x^{\ell^{t-1}})$.  Thus, with $\zeta = \zeta_{D} = \zeta_{d \ell^{j}}$, we see that $\Phi_{D}(\zeta) = \Phi_{d \ell^{j}}(\zeta) = 0$ so that
\[
0 = \Phi_{k p}(\zeta) + \Phi_{k q}(\zeta) = \Phi_{\ell p}(\zeta^{\ell^{t-1}}) + \Phi_{\ell q}(\zeta^{\ell^{t-1}}),
\]implying $\zeta^{\ell^{t-1}} = \zeta_{d \ell^{j}}^{\ell^{t-1}} = \zeta_{d \ell^{j-t+1}}$ is a root of $\Phi_{\ell p}(x) + \Phi_{\ell q}(x)$.  In other words, $\Phi_{d'}(x)$ divides $f(x) = \Phi_{\ell p}(x) + \Phi_{\ell q}(x)$, where $d' = d \ell^{j-t+1}$.  Observe that, in this case where $j \ge t$, we have $d'$ is divisible by $\ell$.  This contradicts Lemma~\ref{dbound} as $\ell$ is an odd prime and in both \ref{mainthmfirstitem} and \ref{mainthmseconditem}, we have the only odd factors of $d$ are divisors of $\ell \pm 1$ and, hence, relatively prime to (and smaller than) $\ell$.  Therefore, $j \le t-1$.

Since $j \le t-1$ and $\ell \nmid d$, we see that
\[
\zeta_{d \ell^{j}}^{\ell^{t-1}} = \zeta_{d}^{\ell^{t-j-1}}
\]
is a primitive $d^{\rm \,th}$ root of unity.  By the definition of $F(x)$ and Lemma~\ref{lemF(x)}, we obtain
\[
0 = F(\zeta_{d \ell^{j}}) 
= \Phi_{\ell}\big( \zeta_{d\ell^{j}}^{p \ell^{t-1}} \big) + \Phi_{\ell}\big( \zeta_{d\ell^{j}}^{q \ell^{t-1}} \big)
= \Phi_{\ell}\big( \zeta_{d}^{p \ell^{t-j-1}} \big) + \Phi_{\ell}\big( \zeta_{d}^{q \ell^{t-j-1}} \big).
\]
Hence, the polynomial $\Phi_{\ell}(x^{p}) + \Phi_{\ell}(x^{q})$ is divisible by $\Phi_{d}(x)$.  
Note that $\Phi_{\ell}(x) \ne \Phi_{d}(x)$ since $d \ne \ell$, which follows from $\ell \nmid d$.  
Lemma~\ref{lemF(x)} with $k = \ell$ implies now that $\Phi_{d}(x)$ divides $\Phi_{\ell p}(x) + \Phi_{\ell q}(x)$.  This completes the proof of the ``only if" part of the lemma.

We suppose now that $j \le t-1$ and $\Phi_{d}(x)$ divides $\Phi_{\ell p}(x) + \Phi_{\ell q}(x)$.
Since $\Phi_{d}(x)$ divides $\Phi_{\ell p}(x) + \Phi_{\ell q}(x)$, we see that $\Phi_{d}(x^{\ell^{t-1}})$ divides $\Phi_{\ell p}(x^{\ell^{t-1}}) + \Phi_{\ell q}(x^{\ell^{t-1}})$.  By Lemma~\ref{lem:cyclotomic-properties}, the latter polynomial is $\Phi_{k p}(x) + \Phi_{k q}(x)$.  We continue our use of Lemma~\ref{lem:cyclotomic-properties}.  Since $\ell \nmid d$, we also have $\Phi_{d}(x^{\ell}) = \Phi_{d}(x) \,\Phi_{d\ell}(x)$ so that
\[
\Phi_{d}(x^{\ell^{t-1}}) = \Phi_{d}(x^{\ell^{t-2}}) \,\Phi_{d\ell}(x^{\ell^{t-2}}) 
= \Phi_{d}(x^{\ell^{t-2}}) \,\Phi_{d\ell^{t-1}}(x).
\]
By induction, we obtain
\[
\Phi_{d}(x^{\ell^{t-1}}) 
= \Phi_{d}(x) \,\Phi_{d\ell}(x) \cdots \,\Phi_{d\ell^{t-2}}(x)\,\Phi_{d\ell^{t-1}}(x).
\]
Since $\Phi_{d}(x^{\ell^{t-1}})$ divides $\Phi_{k p}(x) + \Phi_{k q}(x)$ and $j \le t-1$, we deduce that $\Phi_{D}(x) = \Phi_{d \ell^{j}}(x)$ divides $\Phi_{k p}(x) + \Phi_{k q}(x)$, completing the proof of the ``if" part of the lemma.
\end{proof}

\section{A density result for the irreducibility of \boldmath{$\Phi_{\ell p}(x) + \Phi_{\ell q}(x)$}}

We turn now to more information we can obtain on the factorization of $\Phi_{\ell p}(x) + \Phi_{\ell q}(x)$ when $p$, $q$ and $\ell$ are primes with $p > q > \ell$.  We begin with the following corollary to Theorem~\ref{kpq:mainthm}. 

\begin{corollary}\label{kpqmainthmcor}
Let $p$, $q$ and $\ell$ be primes satisfying $p > q > \ell$.  If $\Phi_{\ell p}(x) + \Phi_{\ell q}(x)$ has a cyclotomic factor, then it has a cyclotomic factor of the form $\Phi_{2^{j}}(x)$ where $j \ge 2$.  Furthermore, $\Phi_{2^{j}}(x)$ divides $\Phi_{\ell p}(x) + \Phi_{\ell q}(x)$ if and only if one of the following occurs.
\begin{enumerate}[label=(\roman*)]
\item
We have $2^{j-1} \ediv (\ell-1)$ and $q \equiv -p \pmod{2^{j}}$. \label{kpqmainthmcorfirstitem}
\item
We have $j \ge 3$, $2^{j} \mid (\ell+1)$ and $q \equiv p + 2^{j-1} {\hskip -2pt}\pmod{2^{j}}$. \label{kpqmainthmcorseconditem}
\end{enumerate}
\end{corollary}

\begin{proof}
Let $f(x) = \Phi_{\ell p}(x) + \Phi_{\ell q}(x)$.  The polynomial $f(x)$ is monic with constant term $2$, and, 
by Corollary~\ref{pq2:cor}, we know that $f(x)$ is irreducible if $\ell = 2$.  We deduce that $f(x)$ has no cyclotomic divisors if $\ell = 2$, and Corollary~\ref{kpqmainthmcor} follows in this case.  We consider now odd $\ell$.  

In Theorem~\ref{kpq:mainthm}, we take $k = \ell$.  Then $t = 1$ and $D = d$. 
For the case of Theorem~\ref{kpq:mainthm}~(i), we set $s$ to be the positive integer satisfying $2^{s} \ediv (\ell-1)$.
Suppose $f(x)$ has some cyclotomic divisor $\Phi_{d}(x)$ as in Theorem~\ref{kpq:mainthm}~(i).  Then, with $j = s+1$, we see that $2^{j-1} \ediv (\ell-1)$, $2^{j} \ediv d$ and $q \equiv -p \pmod{d}$.  In this case, since $2^{j} \mid d$, we obtain $q \equiv -p \pmod{2^{j}}$.  
For such primes $p$ and $q$, Theorem~\ref{kpq:mainthm}~(i) implies that $f(x)$ is divisible by $\Phi_{2^{j}}(x)$.  Thus, whenever a cyclotomic divisor $\Phi_{d}(x)$ divides $f(x)$ in Theorem~\ref{kpq:mainthm}~(i), we must also have $\Phi_{2^{j}}(x)$ divides $f(x)$ and such factors occur precisely when $2^{j-1} \ediv (\ell-1)$ and $q \equiv -p \pmod{2^{j}}$.

For the case of Theorem~\ref{kpq:mainthm}~(ii), we obtain that if $\Phi_{d}(x) \mid f(x)$ and $j \ge 2$ is the positive integer for which $2^{j} \ediv d$, then $2^{j} \mid (\ell+1)$ and $q \equiv p + (d/2) \pmod{d}$.  The latter implies
\[
q \equiv p + (d/2) \equiv p + 2^{j-1} \pmod{2^{j}}.
\]
Theorem~\ref{kpq:mainthm}~(ii) then implies that we also have $\Phi_{2^{j}}(x) \mid f(x)$, and such factors occur precisely when $2^{j} \mid (\ell+1)$ and $q \equiv p + 2^{j-1} \pmod{2^{j}}$ for some $j \ge 2$.  Thus, for $j \ge 3$, we are led to Corollary~\ref{kpqmainthmcor}~(ii).

The condition $j \ge 3$ in Corollary~\ref{kpqmainthmcor}~(ii) leaves out the possibility that $4 \mid (\ell+1)$ and $\Phi_{4}(x)$ is a factor of $f(x)$ occurring when $q \equiv p + 2 \pmod{4}$.  Since $p$ is odd, we see $2p+2 \equiv 0 \pmod{4}$ so the condition $q \equiv p + 2 \pmod{4}$ is equivalent to
\[
q \equiv p + 2 \equiv p+2 + (2p+2) \equiv 3p \equiv -p \pmod{4}.
\]
Since $4 \mid (\ell+1)$, we see that $2 \ediv (\ell-1)$.  Thus, in this case, we already obtain from Corollary~\ref{kpqmainthmcor}~(i) that $\Phi_{4}(x) \mid f(x)$, finishing the proof.
\end{proof}

The corollary above does not extend to show all $\Phi_{n}(x)+\Phi_{m}(x)$ which have a cyclotomic divisor have one of the form $\Phi_{2^{j}}(x)$.  For example, the only cyclotomic divisor of $\Phi_{2}(x) + \Phi_{10}(x)$ and of $\Phi_{17}(x)+\Phi_{55}(x)$ is $x^{2}+x+1$.

\begin{corollary}\label{seventwelfthcor}
The density of primes $p$, $q$ and $\ell$ satisfying $\ell < q < p \le X$ for which $\Phi_{\ell p}(x) + \Phi_{\ell q}(x)$ is irreducible tends to $7/12$ as $X \rightarrow \infty$.
\end{corollary}

\begin{proof}
Throughout the proof we consider $X$ to be large.  
By the Prime Number Theorem (or see \eqref{siegelwalfisz} below) the number of prime triples $(p,q,\ell)$ satisfying $\ell < q < p \le X$ is 
\[
\binom{\pi(X)}{3} = \dfrac{X^{3}}{6 \log^{3}X} + O\bigg( \dfrac{X^{3}}{\log^{4}X} \bigg).
\]
To establish the corollary, it suffices then to show that the number of such prime triples for which $\Phi_{\ell p}(x) + \Phi_{\ell q}(x)$ is reducible, say $R(X)$, satisfies
\begin{equation}\label{RofXgoal}
R(X) = \dfrac{5}{12} \cdot \dfrac{X^{3}}{6 \log^{3}X} + O\bigg( \dfrac{X^{3}\log\log X}{\log^{4}X} \bigg)
= \dfrac{5\,X^{3}}{72 \log^{3}X} + O\bigg( \dfrac{X^{3}\log\log X}{\log^{4}X} \bigg).
\end{equation}
Given Theorem~\ref{kpq:mainthm}, we can view $R(X)$ as the number of prime triples $(p,q,\ell)$ satisfying $\ell < q < p \le X$ for which $\Phi_{\ell p}(x) + \Phi_{\ell q}(x)$ has a cyclotomic factor; and from Corollary~\ref{kpqmainthmcor}, we deduce that this is the same as the number of prime triples $(p,q,\ell)$ satisfying $\ell < q < p \le X$ for which Corollary~\ref{kpqmainthmcor}~(i) or (ii) holds.

Below, we make use of the Siegel--Walfisz Theorem (eg., \cite[Lemma~2.9]{elliott}) in the form that the number of primes $\le Y$ that are in a reduced residue class $a$ modulo a positive integer $A \le \log^{3} Y$ is 
\begin{equation}\label{siegelwalfisz}
\pi(Y; A,a) = 
\dfrac{Y}{\phi(A) \log Y} + O\bigg( \dfrac{Y}{\phi(A) \log^{2} Y} \bigg),
\end{equation}
where the implied constant is absolute.  
Note that the asymptotic formula corresponds to an asymptotic formula for $\pi(Y)$ when $A = 2$ and $a = 1$.  

If $\ell - 1$ is divisible by $4$, then $8 \nmid (\ell+1)$.  Hence, the conditions on $\ell$ in Corollary~\ref{kpqmainthmcor}~(i) and (ii) are disjoint except when $2 \ediv (\ell-1)$, in which case we also have the possibility that $2^{j} \mid (\ell+1)$ for some $j \ge 3$.  In this case, the primes $p$ and $q$ in (i) satisfy $q \equiv -p \pmod{4}$ and 
the primes $p$ and $q$ in (ii) satisfy $q \equiv p + 2^{j-1} \pmod{2^{j}}$.  As $j \ge 3$, the latter implies $q \equiv p \pmod{4}$, which cannot occur at the same time as $q \equiv -p \pmod{4}$.  We deduce that the triples $(p,q,\ell)$ that occur in Corollary~\ref{kpqmainthmcor}~(i) are disjoint from the triples $(p,q,r)$ occurring in Corollary~\ref{kpqmainthmcor}~(ii).  

Furthermore, a triple $(p,q,r)$ satisfies Corollary~\ref{kpqmainthmcor}~(i) for at most one $j$ since $2^{j-1} \ediv (\ell-1)$ for at most one $j$.  In Corollary~\ref{kpqmainthmcor}~(ii), the condition $q \equiv p + 2^{j-1} {\hskip -2pt}\pmod{2^{j}}$ implies $2^{j-1} \ediv (q-p)$.  Thus, also, a triple $(p,q,\ell)$ satisfies Corollary~\ref{kpqmainthmcor}~(ii) for at most one $j$.  

Having shown that the conditions in Corollary~\ref{kpqmainthmcor}~(i) and (ii) are all disjoint as $j$ varies, our plan is to establish \eqref{RofXgoal} by summing up the different contributions from Corollary~\ref{kpqmainthmcor}~(i) and (ii).  However, we make some simplifications by showing first we need only consider triples satisfying 
\begin{equation}\label{qboundnubound}
\ell \ge \dfrac{X}{\log X} \qquad \text{and} \qquad 
\max \{ 2^{\nu_{2}(\ell-1)}, 2^{\nu_{2}(\ell+1)} \} \le \log X,
\end{equation}
where $\nu_{2}$ is the $2$-adic valuation giving the exponent in the largest power of $2$ that divides a number. 
For the former case in \eqref{qboundnubound}, observe that the Prime Number Theorem or \eqref{siegelwalfisz} implies that the number of prime triples $(p,q,\ell)$ for which $\ell < X/\log X$ is at most
\[
\pi(X) \cdot \pi(X) \cdot \pi(X/\log X) \ll \dfrac{X^{3}}{\log^{4}X},
\]
so such prime triples can be viewed as contributing only to the error term in \eqref{RofXgoal}.  For the second case in \eqref{qboundnubound}, let $j_{0}$ be the integer for which $2^{j_{0}-1} \le \log X < 2^{j_{0}}$.  If $2^{\nu_{2}(\ell-1)} > \log X$, then $\nu_{2}(\ell-1) \ge j_{0}$ or equivalently $2^{j_{0}} \mid (\ell-1)$.  
By \eqref{siegelwalfisz}, we see that the number of $\ell \le X$ for which $2^{j_{0}} \mid (\ell-1)$ is at most
\[
\pi(X; 2^{j_{0}},1) \ll \dfrac{X}{\phi(2^{j_{0}}) \log X} = \dfrac{X}{2^{j_{0}-1} \log X} \ll \dfrac{X}{\log^{2} X}.
\]
As the number of choices for the pair of primes $p$ and $q$ is $O(\pi(X)^{2}) = O(X^{2}/\log^{2}X)$, we deduce that the number of prime triples $(p,q,\ell)$ for which $2^{\nu_{2}(\ell-1)} > \log X$ is $\ll X^{3}/\log^{4} X$, and therefore these triples can be viewed as contributing only to the error term in \eqref{RofXgoal}.  A similar analysis can be done in the case that $2^{\nu_{2}(\ell+1)} > \log X$, finishing the justification that we now need only consider triples satisfying \eqref{qboundnubound}.

We continue to use the notation that $j_{0}$ is the integer for which $2^{j_{0}-1} \le \log X < 2^{j_{0}}$ and now restrict to prime triples $(p,q,\ell)$ satisfying $\ell < q < p \le X$ and \eqref{qboundnubound}.  In particular, \eqref{qboundnubound} implies that $2^{j_{0}} \nmid (\ell \pm 1)$.
Suppose we have a prime triple $(p,q,\ell)$ with $p > q > \ell$ for which Corollary~\ref{kpqmainthmcor}~(i) holds.  Since $2^{j-1} \ediv (\ell-1)$ in (i), we deduce that $j-1 < j_{0}$ or equivalently $j \le j_{0}$.
Observe that the condition $2^{j-1} \ediv (\ell-1)$ is equivalent to saying $\ell \equiv 2^{j-1}+1 \pmod{2^{j}}$.  
Thus, the number of prime triples $(p,q,\ell)$ with $p > q > \ell$ as in Corollary~\ref{kpqmainthmcor}~(i), with $j \in [2,j_{0}]$ fixed and $\ell \ge X/\log X$ as in \eqref{qboundnubound}, is
\begin{equation*}
\begin{split}
\sum_{\substack{X/\log X \le \ell \le X \\ \ell \equiv 2^{j-1}+1 {\hskip -6pt}\pmod{2^{j}}}} &{\hskip -3pt}\sum_{\ell < q \le X} {\hskip -3pt}\sum_{\substack{q < p \le X \\ p \equiv -q {\hskip -6pt}\pmod{2^{j}}}}  1 
= {\hskip -5pt}\sum_{\substack{\ell \le X \\ \ell \equiv 2^{j-1}+1 {\hskip -6pt}\pmod{2^{j}}}} {\hskip -3pt}\sum_{\ell < q \le X} \big( \pi(X; 2^{j}, -q) - \pi(q; 2^{j}, -q) \big) \\[5pt]
&= \sum_{\substack{X/\log X \le \ell \le X \\ \ell \equiv 2^{j-1}+1 {\hskip -6pt}\pmod{2^{j}}}} \sum_{\ell < q \le X} 
\bigg( \dfrac{X}{2^{j-1} \log X} - \dfrac{q}{2^{j-1} \log q} + O\bigg( \dfrac{X}{2^{j}\log^{2} X} \bigg) \bigg) \\[5pt]
&= \sum_{\substack{X/\log X \le \ell \le X \\ \ell \equiv 2^{j-1}+1 {\hskip -6pt}\pmod{2^{j}}}} \sum_{\ell < q \le X} 
\bigg( \dfrac{X-q}{2^{j-1} \log X} + O\bigg( \dfrac{X \log\log X}{2^{j}\log^{2} X} \bigg) \bigg) \\[5pt]
&= \sum_{\substack{X/\log X \le \ell \le X \\ \ell \equiv 2^{j-1}+1 {\hskip -6pt}\pmod{2^{j}}}} \sum_{\ell < q \le X} \dfrac{X-q}{2^{j-1} \log X} + O\bigg( \dfrac{X^{3} \log\log X}{2^{j}\log^{4} X} \bigg),
\end{split}
\end{equation*}
where we have used \eqref{siegelwalfisz} and taken advantage of $q > \ell \ge X/\log X$.  
For a given $Y$, we can see that
\begin{equation}\label{sumofprimestoY}
\begin{split}
\sum_{q \le Y} q &= \int_{2^{-}}^{Y} t \,d \,\pi(t)
= Y\pi(Y) - \int_{2^{-}}^{Y} \pi(t) \,dt \\[5pt]
&= Y\pi(Y) - \int_{Y/\log Y}^{Y} \pi(t) \,dt + O\bigg( \dfrac{Y^{2}}{\log^{2}Y} \bigg) \\[5pt]
&= Y\pi(Y) - \bigg( 1 + O\bigg( \dfrac{\log\log Y}{\log Y} \bigg)  \bigg) \dfrac{1}{\log Y} \int_{Y/\log Y}^{Y} t \,dt + O\bigg( \dfrac{Y^{2}}{\log^{2}Y} \bigg) \\[5pt]
&= \dfrac{Y^{2}}{2\log Y} + O\bigg( \dfrac{Y^{2} \log\log Y}{\log^{2}Y} \bigg).
\end{split}
\end{equation}
Therefore, using $\ell \ge X/\log X$, we obtain that
\begin{align*}
\sum_{\ell < q \le X} \dfrac{X-q}{2^{j-1} \log X} &= \dfrac{X \big(\pi(X) - \pi(\ell)\big)}{2^{j-1} \log X} - \dfrac{X^{2}-\ell^{2}}{2^{j}\log^{2} X} + O\bigg(  \dfrac{X^{2} \log\log X}{2^{j}\log^{3}X} \bigg)  \\[5pt]
&= \dfrac{X^{2}}{2^{j}\log^{2} X} - \dfrac{X \ell}{2^{j-1}\log^{2} X} + \dfrac{\ell^{2}}{2^{j}\log^{2} X} 
+ O\bigg(  \dfrac{X^{2} \log\log X}{2^{j}\log^{3}X} \bigg).
\end{align*}
Next, we want to sum over primes $\ell \in [X/\log X, X]$ where $\ell \equiv 2^{j-1}+1 {\hskip -3pt}\pmod{2^{j}}$.   
Recall $j \le j_{0}$ and $2^{j_{0}-1} \le \log X$, so the modulus $2^{j}$ is at most $2 \log X$.  
Similar to our analysis in \eqref{sumofprimestoY}, applying the Siegel--Walfisz Theorem as given in \eqref{siegelwalfisz}, we obtain
\begin{equation*} 
\begin{split}
\sum_{\substack{X/\log X \le \ell \le X \\ \ell \equiv 2^{j-1}+1 {\hskip -6pt}\pmod{2^{j}}}} \ell 
&= \int_{(X/\log X)^{-}}^{X} t \,d \,\pi(t; 2^{j}, 2^{j-1}+1) \\[5pt]
&= X \pi(X; 2^{j}, 2^{j-1}+1) - \int_{X/\log X}^{X} \pi(t; 2^{j}, 2^{j-1}+1) \,dt + O\bigg( \dfrac{X^{2}}{2^{j}\log^{2}X} \bigg)\\[5pt]
&= \dfrac{X^{2}}{2^{j-1}\log X}  -  \dfrac{1}{2^{j-1}\log X} \int_{X/\log X}^{X} t \,dt + O\bigg( \dfrac{X^{2} \log\log X}{2^{j}\log^{2}X} \bigg) \\[5pt]
&= \dfrac{X^{2}}{2^{j}\log X} + O\bigg( \dfrac{X^{2} \log\log X}{2^{j}\log^{2}X} \bigg),
\end{split}
\end{equation*}
and
\begin{equation*} 
\begin{split}
\sum_{\substack{X/\log X \le \ell \le X \\ \ell \equiv 2^{j-1}+1 {\hskip -6pt}\pmod{2^{j}}}} \ell^{2}
&= \int_{(X/\log X)^{-}}^{X} t^{2} \,d \,\pi(t; 2^{j}, 2^{j-1}+1) \\[5pt]
&= X^{2} \pi(X; 2^{j}, 2^{j-1}+1) - 2\int_{X/\log X}^{X} \pi(t; 2^{j}, 2^{j-1}+1) \,t\,dt + O\bigg( \dfrac{X^{3}}{2^{j}\log^{2}X} \bigg) \\[5pt]
&= \dfrac{X^{3}}{2^{j-1}\log X}  -  \dfrac{1}{2^{j-2}\log X} \int_{X/\log X}^{X} t^{2} \,dt + O\bigg( \dfrac{X^{3} \log\log X}{2^{j}\log^{2}X} \bigg) \\[5pt]
&= \dfrac{X^{3}}{3 \cdot 2^{j-1}\log X} + O\bigg( \dfrac{X^{3} \log\log X}{2^{j}\log^{2}X} \bigg).
\end{split}
\end{equation*}
Thus, 
\begin{align*}
&\sum_{\substack{X/\log X \le \ell \le X \\ \ell \equiv 2^{j-1}+1 {\hskip -6pt}\pmod{2^{j}}}} \sum_{\ell < q \le X} \dfrac{X-q}{2^{j-1} \log X} \\[5pt]
&\qquad = \dfrac{X^{2} \pi(X; 2^{j}, 2^{j-1}+1)}{2^{j} \log^{2} X}
- \dfrac{X^{3}}{2^{2j-1}\log^{3} X} 
+ \dfrac{X^{3}}{3\cdot 2^{2j-1}\log^{3} X} 
+ O\bigg(  \dfrac{X^{3} \log\log X}{2^{j}\log^{4}X} \bigg)  \\[5pt]
&\qquad = \dfrac{X^{3}}{\log^{3} X} \bigg( \dfrac{1}{2^{2j-1}} - \dfrac{1}{2^{2j-1}} + \dfrac{1}{3\cdot 2^{2j-1}} \bigg)
+ O\bigg(  \dfrac{X^{3} \log\log X}{2^{j}\log^{4}X} \bigg) \\[5pt]
&\qquad = \dfrac{X^{3}}{3\cdot 2^{2j-1}\log^{3} X} + O\bigg(  \dfrac{X^{3} \log\log X}{2^{j}\log^{4}X} \bigg).
\end{align*}
Summing over $j \in [2,j_{0}]$ where $2^{j_{0}} > \log X$, we deduce that the number of prime triples $(p,q,\ell)$ with $p > q > \ell$ as in Corollary~\ref{kpqmainthmcor}~(i) is
\[
\bigg( \sum_{2 \le j \le j_{0}} \dfrac{1}{3\cdot 2^{2j-1}} \bigg) \dfrac{X^{3}}{\log^{3} X}
+  O\bigg(  \dfrac{X^{3} \log\log X}{\log^{4}X} \bigg)
= \dfrac{X^{3}}{18 \,\log^{3} X} + O\bigg(  \dfrac{X^{3} \log\log X}{\log^{4}X} \bigg).
\]

Now, we want to make a similar analysis for Corollary~\ref{kpqmainthmcor}~(ii) to determine a main term and error term for the triple sum
\[
\sum_{\substack{X/\log X \le \ell \le X \\ \ell \equiv -1 {\hskip -4pt}\pmod{2^{j}}}} 
\,\sum_{\ell < q \le X}
\sum_{\substack{q < p \le X \\ p \equiv q+ 2^{j-1} {\hskip -4pt}\pmod{2^{j}}}} 1,
\] 
where here $3 \le j \le j_{0}-1$.  An almost identical analysis to that given for Corollary~\ref{kpqmainthmcor}~(i) allows us to deduce that the triple sum is
\[
\dfrac{X^{3}}{3\cdot 2^{2j-1}\log^{3} X} + O\bigg(  \dfrac{X^{3} \log\log X}{2^{j}\log^{4}X} \bigg),
\]
as before.  In this case, we are summing over $j \in [3,j_{0}-1]$, which gives then that the number of prime triples $(p,q,\ell)$ with $p > q > \ell$ as in Corollary~\ref{kpqmainthmcor}~(ii) is
\[
\bigg( \sum_{3 \le j \le j_{0}-1} \dfrac{1}{3\cdot 2^{2j-1}} \bigg) \dfrac{X^{3}}{\log^{3} X}
+  O\bigg(  \dfrac{X^{3} \log\log X}{\log^{4}X} \bigg)
= \dfrac{X^{3}}{72 \,\log^{3} X} + O\bigg(  \dfrac{X^{3} \log\log X}{\log^{4}X} \bigg).
\]
Summing the asymptotic formulas that we obtained in the case of prime triples $(p,q,\ell)$ with $p > q > \ell$ as in Corollary~\ref{kpqmainthmcor}~(i) and as in Corollary~\ref{kpqmainthmcor}~(ii), we deduce \eqref {RofXgoal}, finishing the proof.
\end{proof}

\section{Preliminaries for further applications of Theorem~\ref{factorizationthm}}

In the introduction, for primes $p > q > \ell$, we explained that there is a multiple in $\mathbb Z[x]$ of the polynomial $\Phi_{\ell p}(x) + \Phi_{\ell q}(x)$ which is of the form \eqref{basicFeq} needed to apply Theorem~\ref{factorizationthm}.  In this section, we address the existence of such a multiple for other values of $\Phi_{n}(x)+\Phi_{m}(x)$.  We begin with the following.

\begin{lemma}\label{multunique}
Let $u(x)$ be a polynomial in $\mathbb Z[x]$ for which $u(0)$ is prime and for which there exists a $v(x) \in \mathbb Z[x]$ such that $u(x) v(x)$ is of the form given in \eqref{basicFeq} for some prime $p$.  Then $v(x)$ is unique.  In other words, if $v_{1}(x)$ and $v_{2}(x)$ are in $\mathbb Z[x]$ with $u(x) v_{1}(x)$ of the form \eqref{basicFeq} for some prime $p = p_{1}$ and with $u(x) v_{2}(x)$ of the form \eqref{basicFeq} for some primes $p = p_{2}$, then $p_{1} = p_{2}$ and $v_{1}(x) = v_{2}(x)$.
\end{lemma}

\begin{proof}
Let $v(x) \in \mathbb Z[x]$ be such that $u(x) v(x)$ is of the form \eqref{basicFeq} for some prime $p$.  Observe that $u(0)$ is a prime that divides the constant term of $u(x)v(x)$, so we must have $p = u(0)$.  
Write
\[
u(x) = \sum_{j=0}^{\infty} a_{j} x^{j}
\qquad \text{and} \qquad
v(x) = \sum_{j=0}^{\infty} b_{j} x^{j},
\]
where the $a_{j}$ and $b_{j}$ are integers and all but finitely many are $0$.  We view the $a_{j}$ as fixed and the $b_{j}$ as indeterminate integers.  The condition $u(0) = p$ implies $a_{0} = p$.  
Given $u(x) v(x)$ is of the form given in \eqref{basicFeq}, we see that $a_{0}b_{0} = p$ so that $b_{0} = 1$.  Suppose now that we know the values of $b_{0}, \ldots, b_{j}$ for some integer $j \ge 0$.  Given $u(x) v(x)$ is of the form given in \eqref{basicFeq}, the coefficient of $x^{j+1}$ in $u(x) v(x)$ is 
\begin{equation}\label{multunique:prfeq}
b_{j+1} a_{0} + b_{j} a_{1} + \cdots + b_{0} a_{j+1} \in \{ 0, 1 \}.
\end{equation}
Since $a_{0} = p$, we deduce that
\[
b_{j+1} = 
\begin{cases}
-(b_{j} a_{1} + \cdots + b_{0} a_{j+1})/p &\text{if } b_{j} a_{1} + \cdots + b_{0} a_{j+1} \equiv 0 {\hskip -6pt}\pmod{p} \\
-(b_{j} a_{1} + \cdots + b_{0} a_{j+1}-1)/p &\text{if } b_{j} a_{1} + \cdots + b_{0} a_{j+1} \equiv 1 {\hskip -6pt}\pmod{p}.
\end{cases}
\]
Since $u(x) v(x)$ is in the form \eqref{basicFeq}, exactly one of the two possibilities for $b_{j+1}$ occurs for each $j$.  Furthermore, as $v(x)$ exists, we have that $b_{j+1} = 0$ for $j$ sufficiently large.  The lemma follows.
\end{proof}

The above proof is constructive, providing a way for us to compute $v(x)$ given $u(x)$.  If we arrive at a point that $b_{j} = 0$ for $\deg u$ consecutive values of $j$, then we can stop as \eqref{multunique:prfeq} leads to $b_{j} = 0$ for all larger $j$ as well.  But what if $v(x)$ does not exist?  In that case, the algorithm will not arrive at a point where $b_{j} = 0$ for $\deg u$ consecutive values of $j$, and it becomes less clear when to know we can stop and deduce that such a $v(x)$ does not exist.  To resolve this difficulty in practice, we observe how we plan to use Theorem~\ref{factorizationthm} and what implications this has on $v(x)$.  The goal in applying Theorem~\ref{factorizationthm} is to show that a certain polynomial $u(x)$ (for Conjecture~\ref{nicolconj}, we want $u(x) = \Phi_{n}(x)+\Phi_{m}(x)$) is a product of reciprocal polynomials times an irreducible non-reciprocal polynomial.  Since we want $u(x)v(x)$ to be of the form \eqref{basicFeq}, we require $u(x)v(x)$ to be monic.  In particular, the leading coefficient of $u(x)$ cannot be divisible by $u(0)$ which is prime.  This implies $u(x)$ is not reciprocal and hence has an irreducible non-reciprocal factor.  In order then for $u(x)v(x)$ to be of the form \eqref{basicFeq}, Theorem~\ref{factorizationthm} implies $v(x)$ must be reciprocal.  
Therefore, we are only interested in the case that $v(x)$ in Lemma~\ref{multunique} is reciprocal.   

\begin{example}\label{algorithmends}
Take $u(x) = \Phi_{3\cdot 13}(x)+\Phi_{7\cdot 11}(x)$.  The computation given in the proof of Lemma~\ref{multunique} gives that if $v(x)$ exists, then the terms of $v(x)$ of degree $\le 36$ are given by 
\begin{align*}
v(x) = \cdots &+ x^{36} - x^{30} - x^{29} - x^{28} - x^{27} - x^{26} - x^{25} - x^{24} - x^{23} - x^{22} + x^{17} + x^{16} + x^{15}\\
&+ x^{14} + x^{13} + x^{12} + x^{11} + x^{10} + x^{9} + x^{8} + x^{7} + x^{6} + x^{5} + x^{4} + x^{3} + x^{2} + x + 1.
\end{align*}
Here, $u(0) = 2$, and we wish to apply Theorem~\ref{factorizationthm} to deduce that the non-reciprocal part of $u(x)$ is irreducible.  As indicated above, we require $v(x)$ to be reciprocal meaning that if $k = \deg v(x)$, then the leading terms of $v(x)$ are given by
\[
v(x) = x^{k} + x^{k-1} + \cdots + x^{k-17} - x^{k-22} - \cdots - x^{k-30} + x^{k-36} + \cdots.
\]
Note that $\deg u = 60$ so that $\deg (uv) = k+60$.  A direct computation now gives that the leading terms of $u(x) v(x)$ are given by 
\[
u(x) v(x) = x^{k+60} + x^{k+53} + x^{k+49} + x^{k+46} + x^{k+39} + x^{k+32} - x^{k+31} + x^{k+29} + x^{k+25} + \cdots.
\]
As the coefficient of $x^{k+31}$ is not in $\{ 0,1 \}$, we deduce that we cannot have $u(x) v(x)$ in the form \eqref{basicFeq}.  Hence, we cannot apply Theorem~\ref{factorizationthm} as described above.
\end{example}

\begin{example}
To emphasize an alternative use of Theorem~\ref{factorizationthm}, we show how to justify $\Phi_{3\cdot 13}(x)+\Phi_{7\cdot 11}(x)$ from the previous example has an irreducible non-reciprocal part.  To do so, we take $u(x) = \Phi_{3\cdot 13}(-x)+\Phi_{7\cdot 11}(-x)$.  The non-reciprocal part of a polynomial $f(x) \in \mathbb Z[x]$ is irreducible if and only if the non-reciprocal part of the polynomial $f(-x)$ is irreducible, so to show that $\Phi_{3\cdot 13}(x)+\Phi_{7\cdot 11}(x)$ has an irreducible non-reciprocal part, it suffices to show that $u(x)$ has an irreducible non-reciprocal part.  When we compute $v(x)$ for $u(x) = \Phi_{3\cdot 13}(-x)+\Phi_{7\cdot 11}(-x)$, we obtain
\begin{align*}
v(x) = x^{33} &- x^{32} + x^{31} + x^{26} - x^{25} + x^{24} + x^{22} - x^{21} + 2 x^{20} - x^{19} + x^{18}  \\
&+ x^{15} - x^{14} + 2 x^{13} - x^{12} + x^{11} + x^{9} - x^{8} + x^{7} + x^{2} - x + 1.
\end{align*}
A direct computation gives
\begin{align*}
u(x) v(x) = x^{93} &+ x^{90} + x^{80} + x^{77} + x^{57} + x^{50} + x^{46} + x^{39} \\
&+ x^{18} + x^{16} + x^{13} + x^{11} + x^{7} + x^{3} + 2.
\end{align*}
Therefore, Theorem~\ref{factorizationthm} applies, and we deduce that the non-reciprocal part of $u(x)$ and, hence, $\Phi_{3\cdot 13}(x)+\Phi_{7\cdot 11}(x)$ is irreducible, where we are using that since $u(x)$ is monic and $u(0) = 2$, we know $u(x)$ has a non-constant non-reciprocal part.  An application of Lemma~\ref{recipfactf(x)} now implies further that the non-cyclotomic part of $\Phi_{3\cdot 13}(x)+\Phi_{7\cdot 11}(x)$ is irreducible. 

Without even checking the explicit value of $v(x)$ given above, we know from Theorem~\ref{factorizationthm} that $v(x)$ must only have irreducible factors that are reciprocal.  One can check that in fact
\[
v(x) = \Phi_{2}(x)^{3} \,\Phi_{6}(x) \,\Phi_{14}(x) \,\Phi_{22}(x) \,\Phi_{26}(x) 
= \dfrac{(x^{3}+1)(x^{7}+1)(x^{11}+1)(x^{13}+1)}{x+1}.
\]
The factorization of $v(x)$ and its connections to the primes appearing in the subscripts of $\Phi_{3\cdot 13}(x)$ and $\Phi_{7\cdot 11}(x)$ are surely not a coincidence, particularly given that Lemma~\ref{multunique} implies that no other $v(x)$ could have been used in this argument.
\end{example}

Observe that Lemma~\ref{multunique} is implying that there is at most one multiple of $\Phi_{n}(x) + \Phi_{m}(x)$ in $\mathbb Z[x]$ that is of the form \eqref{basicFeq}, and we can replace here $\Phi_{n}(x) + \Phi_{m}(x)$ with any other polynomial.  
A particularly intriguing choice of a polynomial to consider multiples of is a polynomial $F(x)$ already of the form \eqref{basicFeq}.  Here, Lemma~\ref{multunique} implies that we cannot multiply such an $F(x)$ by a polynomial in $\mathbb Z[x]$ and create a different polynomial of the form \eqref{basicFeq}. 
Momentarily, we will be considering a multiple of $\Phi_{n}(-x) + \Phi_{m}(-x)$, as in the example above, that is potentially of the form \eqref{basicFeq}.  In moving forward, it is worth keeping in mind that in cases where we have found a multiple of a polynomial of the form \eqref{basicFeq}, it is the unique multiple with this property.  We will not need this fact moving forward, but it does mean that we have found the right choice of a multiple to consider in such cases.

We recall the function $\omega(m)$ used to denote the number of distinct prime factors of $m$.
We define
\begin{equation}\label{vnm:eq}
v_{n,m}(x) = (x+1)^{\frac{(-1)^{\omega(n)+1}+(-1)^{\omega(m)+1}}{2}} 
{\hskip -8pt}\prod_{\substack{d \mid n, \,d > 1 \\ \omega(d) \not\equiv \omega(n) {\hskip -6pt}\pmod{2}}} {\hskip -8pt}(x^{d} + 1) \times {\hskip -8pt}\prod_{\substack{d' \mid m, \,d' > 1 \\ \omega(d') \not\equiv \omega(m) {\hskip -6pt}\pmod{2}}} {\hskip -8pt}(x^{d'} + 1),
\end{equation}
and
\begin{equation}\label{Fnm:eq}
\begin{split}
F_{n,m}(x) &= \prod_{\substack{d \mid n, \,d > 1 \\ \omega(d) \equiv 0 {\hskip -6pt}\pmod{2}}} {\hskip -8pt}\big(x^{d} + 1\big) \times 
\prod_{\substack{d' \mid m \\ \omega(d') \equiv 1 {\hskip -6pt}\pmod{2}}} {\hskip -8pt}\big(x^{d'} + 1\big) \\[5pt]
&\qquad\quad + 
\prod_{\substack{d' \mid m, \,d' > 1 \\ \omega(d') \equiv 0 {\hskip -6pt}\pmod{2}}} {\hskip -8pt}\big(x^{d'} + 1\big) 
\times 
\prod_{\substack{d \mid n \\ \omega(d) \equiv 1 {\hskip -6pt}\pmod{2}}} {\hskip -8pt}\big(x^{d} + 1\big).
\end{split}
\end{equation}
In the case that the first factor on the right of \eqref{vnm:eq} is $(x+1)^{-1}$, we see that $\omega(n)$ and $\omega(m)$ are both even and the two products on the right of \eqref{vnm:eq} each have a factor of $x+1$.  We deduce that $v_{n,m}(x)$ and more clearly $F_{n,m}(x)$ are in $\mathbb Z[x]$.  
Also, observe that $F_{n,m}(x) = F_{m,n}(x)$ so that a result for $F_{n,m}(x)$ implies an analogous result for $F_{m,n}(x)$ (see, for example, Theorem~\ref{bigenoughgapresult}).  

Next, we explain why these polynomials are of interest to us.
Note that $v_{3\cdot 13,7\cdot 11}(x)$ is equal to the value of $v(x)$ given in the previous example, and we want to justify its similar role in a more general setting. 

\begin{lemma}\label{vnmFnm:lem}
Let $n > 1$ and $m > 1$ be odd squarefree integers for which $\omega(n) \equiv \omega(m) \pmod{2}$.  
Given $v_{n,m}(x)$ in \eqref{vnm:eq} and $F_{n,m}(x)$ in \eqref{Fnm:eq}, we have
\begin{equation}\label{vnmFnm:eq}
\big( \Phi_{n}(-x) + \Phi_{m}(-x) \big) \,v_{n,m}(x) = F_{n,m}(x).
\end{equation}
\end{lemma}

\begin{proof}
Writing $n = p_{1} \cdots p_{t}$ where the $p_{j}$ are distinct odd primes, we deduce that
\begin{equation*}
\begin{split}
\Phi_{n}(-x) &=\Phi_{2n}(x) 
= \prod_{d \mid (2n)} (x^{d}-1)^{\mu(2n/d)}
= \prod_{d \mid n} \bigg( \dfrac{x^{2d}-1}{x^{d}-1} \bigg)^{\mu(n/d)} \\[5pt]
&= \prod_{d \mid n} \big( x^{d} + 1 \big)^{(-1)^{\omega(n)-\omega(d)}}
= (x+1)^{(-1)^{\omega(n)}} {\hskip -8pt}\prod_{d \mid n, \,d > 1} \big( x^{d} + 1 \big)^{(-1)^{\omega(n)-\omega(d)}} \\[5pt]
&= (x+1)^{(-1)^{\omega(n)}} {\hskip -8pt}\prod_{\substack{d \mid n, \,d > 1 \\ \omega(d) \equiv \omega(n) {\hskip -6pt}\pmod{2}}} {\hskip -8pt}\big(x^{d} + 1\big) \times \prod_{\substack{d \mid n, \,d > 1 \\ \omega(d) \not\equiv \omega(n) {\hskip -6pt}\pmod{2}}} {\hskip -8pt}\big(x^{d} + 1\big)^{-1}.
\end{split}
\end{equation*}
Similarly,
\[
\Phi_{m}(-x) = (x+1)^{(-1)^{\omega(m)}} {\hskip -8pt}\prod_{\substack{d' \mid m, \,d' > 1 \\ \omega(d') \equiv \omega(m) {\hskip -6pt}\pmod{2}}} {\hskip -8pt}\big(x^{d'} + 1\big) \times \prod_{\substack{d' \mid m,, \,d' > 1 \\ \omega(d') \not\equiv \omega(m) {\hskip -6pt}\pmod{2}}} {\hskip -8pt}\big(x^{d'} + 1\big)^{-1}.
\]
Considering the two cases $\omega(n) \equiv \omega(m) \equiv 0 \pmod{2}$ and $\omega(n) \equiv \omega(m) \equiv 1 \pmod{2}$, a direct computation using the above formulas for $\Phi_{n}(-x)$ and $\Phi_{m}(-x)$ with \eqref{vnm:eq} and \eqref{Fnm:eq} gives \eqref{vnmFnm:eq}.
\end{proof}

In the case that $\omega(n) \not\equiv \omega(m) \pmod{2}$, there is an analogous lemma of some interest.

\begin{lemma}\label{vnmFnm:lem2}
Let $n > 1$ and $m > 1$ be odd squarefree integers for which $\omega(n) \not\equiv \omega(m) \pmod{2}$.  
Given $v_{n,m}(x)$ in \eqref{vnm:eq} and $F_{n,m}(x)$ in \eqref{Fnm:eq}, we have
\begin{equation}\label{vnmFnm:eq2}
\big( \Phi_{n}(-x) \,\Phi_{m}(-x) + 1 \big) \,v_{n,m}(x) = F_{n,m}(x).
\end{equation}
\end{lemma}

We omit the proof of Lemma~\ref{vnmFnm:lem2} as it is very similar to the proof of Lemma~\ref{vnmFnm:lem}.  Also, we note that we are not asserting that in general $F_{n,m}(x)$ is of the form given in \eqref{basicFeq}, but our plan is to exploit choices of $n$ and $m$ where it is. 
 
\begin{theorem}\label{fnm:thm}
Let $n' = p_{1} \cdots p_{t}$ and $m' = q_{1} \cdots q_{s}$ for positive integers $t$ and $s$ and $t+s$ distinct odd primes $p_{1}, \ldots, p_{t}, q_{1}, \ldots, q_{s}$.  Let $e$ be a nonnegative integer, and set $n = 2^{e}n'$ and $m = 2^{e}m'$.  Suppose that $F_{n',m'}(x)$ is of the form given in \eqref{basicFeq}.  Then we have the following.
\begin{enumerate}[label=(\roman*)]
\item
If $t \equiv s {\hskip -2pt}\pmod{2}$, then the non-cyclotomic part of $\Phi_{n}(x)+\Phi_{m}(x)$ is irreducible.  \label{fnmthmfirstitem}
\item
If $t \not\equiv s {\hskip -2pt}\pmod{2}$, then the non-cyclotomic part of $\Phi_{n}(x)\,\Phi_{m}(x) + 1$ is irreducible.  \label{fnmthmseconditem}
\end{enumerate} 
\end{theorem}

\begin{proof}
We apply Lemma~\ref{lem:cyclotomic-properties}.  The non-cyclotomic part of $\Phi_{n'}(x)+\Phi_{m'}(x)$ is irreducible  if and only if the non-cyclotomic part of $\Phi_{n'}(-x)+\Phi_{m'}(-x) = \Phi_{2n'}(x)+\Phi_{2m'}(x)$ is irreducible.  Also, the non-cyclotomic part of $\Phi_{n'}(x)\Phi_{m'}(x)+1$ is irreducible if and only if the non-cyclotomic part of $\Phi_{n'}(-x)\Phi_{m'}(-x)+1 = \Phi_{2n'}(x)\Phi_{2m'}(x)+1$ is irreducible.  Thus, we may and do suppose that $e \ge 1$.
We see then that
\[
\Phi_{n}(x)+\Phi_{m}(x) = \Phi_{2n'}(x^{2^{e-1}})+\Phi_{2m'}(x^{2^{e-1}}) 
= \Phi_{n'}(-x^{2^{e-1}})+\Phi_{m'}(-x^{2^{e-1}}) 
\]
and
\[
\Phi_{n}(x)\,\Phi_{m}(x) + 1 = \Phi_{2n'}(x^{2^{e-1}})\,\Phi_{2m'}(x^{2^{e-1}}) + 1
= \Phi_{n'}(-x^{2^{e-1}})\,\Phi_{m'}(-x^{2^{e-1}}) + 1.
\]
If a multiple of $\Phi_{n'}(-x)+\Phi_{m'}(-x)$ or $\Phi_{n'}(-x)\,\Phi_{m'}(-x) + 1$ in $\mathbb Z[x]$ is of the form \eqref{basicFeq}, then a multiple of $\Phi_{n'}(-x^{2^{e-1}})+\Phi_{m'}(-x^{2^{e-1}})$ or $\Phi_{n'}(-x^{2^{e-1}})\,\Phi_{m'}(-x^{2^{e-1}}) + 1$, respectively, is of the form \eqref{basicFeq}.   
By Theorem~\ref{factorizationthm}, Lemma~\ref{recipfactf(x)} and Lemma~\ref{recipfactphinphimplus1}, it suffices then to show for \ref{fnmthmfirstitem} that if $t \equiv s {\hskip -2pt}\pmod{2}$, then a multiple of $\Phi_{n'}(-x)+\Phi_{m'}(-x)$ is in the form \eqref{basicFeq}, and for \ref{fnmthmseconditem} that if $t \not\equiv s {\hskip -2pt}\pmod{2}$, then a multiple of $\Phi_{n'}(-x)\,\Phi_{m'}(-x) + 1$  is in the form \eqref{basicFeq}.  By Lemma~\ref{vnmFnm:lem}, Lemma~\ref{vnmFnm:lem2} and the conditions in the theorem, the polynomial $F_{n',m'}(x)$ is such a multiple, completing the proof.
\end{proof}

\begin{example}\label{4primeseacheg}
One can check that with 
\[
n = 3 \cdot 7 \cdot 41 \cdot 1163 
\qquad \text{ and } \qquad
m = 5 \cdot 11 \cdot 131 \cdot 21803
\]
the polynomial $F_{n,m}(x)$ is of the form given in \eqref{basicFeq}.  From \eqref{Fnm:eq}, this is equivalent to asserting that the sums of distinct elements from 
\[
S_{1} = \{ d \mid n : d > 1, \,\omega(d) \equiv 0 {\hskip -6pt}\pmod{2} \} \,\textstyle\bigcup\,
 \{ d' \mid m : \omega(d') \equiv 1 {\hskip -6pt}\pmod{2} \} 
\]
together with the sums of distinct elements from 
\[
S_{2} = \{ d' \mid m : d' > 1, \,\omega(d') \equiv 0 {\hskip -6pt}\pmod{2} \} \,\textstyle\bigcup\,
 \{ d \mid n : \omega(d) \equiv 1 {\hskip -6pt}\pmod{2} \} 
\]
are distinct, other than $1$ that occurs by considering summing elements from $\emptyset$ in both $S_{1}$ and $S_{2}$.  Let $T(n,m)$ be the set of these distinct sums other than $1$.  If we remove prime factors from $n$ or $m$, the exponents in \eqref{Fnm:eq} become a subset of $T(n,m)$.  As a result, if $n' \mid n$ and $m' \mid m$, then we see that $F_{n',m'}(x)$ also has the property that it is of the form given in \eqref{basicFeq}.  We deduce from Theorem~\ref{fnm:thm} that for each divisor $n'$ of $n$ and each divisor $m'$ of $m$ and all nonnegative integers $e$, the non-cyclotomic part of $\Phi_{2^{e}n'}(x) + \Phi_{2^{e}m'}(x)$ is irreducible if $\omega(n') \equiv \omega(m') \pmod{2}$ and the non-cyclotomic part of $\Phi_{2^{e}n'}(x)\,\Phi_{2^{e}m'}(x)+1$ is irreducible if $\omega(n') \not\equiv \omega(m') \pmod{2}$.
\end{example}

The number of prime divisors in the above example can be extended indefinitely as our next result implies.  

\begin{theorem}\label{bigenoughgapresult}
Let $n$ and $m$ be positive integers for which $F_{n,m}(x)$ is of the form given in \eqref{basicFeq}.  Suppose further that $F_{n,m}(x)$ satisfies the property:
\[
\text{No two terms of $F_{n,m}(x)$ have degrees differing by $1$.}
\tag{P}
\]
Then for a prime
\[
p >  \max \bigg\{ \sum_{\substack{d \mid n, \,d > 1 \\ \omega(d) \equiv 0 {\hskip -6pt}\pmod{2}}} d \,+ \sum_{\substack{d' \mid m \\ \omega(d') \equiv 1 {\hskip -6pt}\pmod{2}}} d', \sum_{\substack{d' \mid m, \,d' > 1 \\ \omega(d') \equiv 0 {\hskip -6pt}\pmod{2}}} d' \,+ \sum_{\substack{d \mid n \\ \omega(d) \equiv 1 {\hskip -6pt}\pmod{2}}} d \bigg\} + 1,
\]
the polynomial $F_{pn,m}(x)$ is also of the form given in \eqref{basicFeq} and no two degrees of terms in $F_{pn,m}(x)$ differ by $1$. 
\end{theorem}

\begin{comment}\label{165comment}
The condition (P) is important as the following example shows.  
Consider $n = 165 = 3 \cdot 5 \cdot 11$.  For this example, the choice of $m$ here can be somewhat arbitrary; we just require that $F_{165,m}(x)$ is of the form \eqref{basicFeq}.  Taking $m = 13$, this is the case.  We claim that for every integer $n' > 1$, the polynomial $F_{165n',13}(x)$ is not of the form \eqref{basicFeq}.  The polynomial $F_{165,13}(x)$ includes the terms $x^{15}$ and $x^{16}$ with degrees differing by $1$.  These terms come from $165$ by itself as $3 \cdot 5 = 15$ and $5 + 11 = 16$, so the first arises from the first product in \eqref{Fnm:eq} with $d = 15$
and the second from the last product in \eqref{Fnm:eq} with $d = 5$ and $d = 11$.  As a consequence, there is a term of degree $16n'$ in $F_{165n',13}(x)$ coming the divisors $5n'$ and $11n'$ of $165n'$ in \eqref{Fnm:eq} since $5n' + 11n' = 16n'$ but also coming from the divisors $n'$ and $15n'$ of $165n'$ in \eqref{Fnm:eq} since $n' + 15\cdot n' = 16n'$.  Thus, there is a term of degree $16n'$ in $F_{165n',13}(x)$ that has coefficient at least $2$ and, hence, $F_{165n',13}(x)$ cannot be of the form \eqref{basicFeq} regardless of the choice of $n'$.
\end{comment}

\begin{comment}
It is worth noting that a positive proportion of $(n,m) \in (\mathbb Z^{+})^{2}$ are such that $n$ and $m$ are odd numbers, $nm$ is squarefree, and $F_{n,m}(x)$ is not of the form \eqref{basicFeq}.  One way to see this is to start with a choice of $n_{0}$ and $m_{0}$ (the choice $n_{0} = 3 \cdot 11$ and $m_{0} = 5 \cdot 7$ will do) for which $F_{n_{0},m_{0}}(x)$ is not of the form \eqref{basicFeq} with $n_{0}$ and $m_{0}$ odd and $n_{0}m_{0}$ squarefree.  Then $F_{n,m}(x)$ will not be of the form \eqref{basicFeq} whenever $n_{0} \mid n$ and $m_{0} \mid m$, leading us to the positive proportion assertion.  The number $165$ in the comment above is a different circumstance where $F_{165,m}(x)$ can be of the form \eqref{basicFeq}, but 
whenever $n > 165$ and $165 \mid n$, independent of the positive integer $m$, the polynomial $F_{n,m}(x)$ is not of the form \eqref{basicFeq}, leading also to the positive proportion assertion. 
\end{comment}

\begin{proof}[Proof of Theorem~\ref{bigenoughgapresult}]
We rewrite \eqref{Fnm:eq} in the form $F_{n,m}(x) = F^{(1)}_{n,m}(x)+F^{(2)}_{n,m}(x)$, where  
\[
F^{(1)}_{n,m}(x) = \prod_{\substack{d \mid n, \,d > 1 \\ \omega(d) \equiv 0 {\hskip -6pt}\pmod{2}}} {\hskip -8pt}\big(x^{d} + 1\big) \times 
\prod_{\substack{d' \mid m \\ \omega(d') \equiv 1 {\hskip -6pt}\pmod{2}}} {\hskip -8pt}\big(x^{d'} + 1\big) 
\]
and
\[
F^{(2)}_{n,m}(x) = \prod_{\substack{d' \mid m, \,d' > 1 \\ \omega(d') \equiv 0 {\hskip -6pt}\pmod{2}}} {\hskip -8pt}\big(x^{d'} + 1\big) \times 
\prod_{\substack{d \mid n \\ \omega(d) \equiv 1 {\hskip -6pt}\pmod{2}}} {\hskip -8pt}\big(x^{d} + 1\big).
\]
Note that the given inequality on $p$ is simply 
\[
p > (\deg F_{n,m}) + 1 = \max\big\{\!\deg F^{(1)}_{n,m}, \deg F^{(2)}_{n,m} \big\} + 1.
\]
Furthermore, this inequality implies $p > n$ since if $n > 1$, then $n$ is a divisor of $n$ occurring in one of the products in $F^{(1)}_{n,m}(x)$ and $F^{(2)}_{n,m}(x)$.  Also, we see that
\[
F^{(1)}_{pn,m}(x) = F^{(1)}_{n,m}(x) 
{\hskip -4pt}\prod_{\substack{d \mid n \\ \omega(d) \equiv 1 {\hskip -6pt}\pmod{2}}} {\hskip -8pt}\big( x^{pd} + 1 \big) 
\quad \text{ and } \quad
F^{(2)}_{pn,m}(x) = F^{(2)}_{n,m}(x) 
{\hskip -4pt}\prod_{\substack{d \mid n \\ \omega(d) \equiv 0 {\hskip -6pt}\pmod{2}}} {\hskip -8pt}\big( x^{pd} + 1 \big),
\]
where in the latter case we include $d = 1$ so that the factor $x^{p}+1$ occurs there.  The conditions on $F_{n,m}(x)$ imply that the non-constant terms in
\[
\prod_{\substack{d \mid n \\ \omega(d) \equiv 1 {\hskip -6pt}\pmod{2}}} {\hskip -8pt}\big( x^{pd} + 1 \big) 
\,+
\prod_{\substack{d \mid n, \,d > 1 \\ \omega(d) \equiv 0 {\hskip -6pt}\pmod{2}}} {\hskip -8pt}\big( x^{pd} + 1 \big)
\]
all have coefficient $1$ as they appear as non-constant terms in $F_{n,m}(x^{p})$ and further that the degrees of these terms (including the constant term) do not differ by $p$.  It follows that the non-constant terms in 
\[
\prod_{\substack{d \mid n \\ \omega(d) \equiv 1 {\hskip -6pt}\pmod{2}}} {\hskip -8pt}\big( x^{pd} + 1 \big) 
\,+
\prod_{\substack{d \mid n \\ \omega(d) \equiv 0 {\hskip -6pt}\pmod{2}}} {\hskip -8pt}\big( x^{pd} + 1 \big)
\]
all have coefficient $1$.  Since the degrees of these terms also differ by more than $\deg F^{(1)}_{n,m}$ and $\deg F^{(2)}_{n,m}$ and since each of $F^{(1)}_{n,m}(x)$, $F^{(2)}_{n,m}(x)$ and $F^{(1)}_{n,m}(x) + F^{(2)}_{n,m}(x)$ only have terms with coefficient $1$ other than the constant term $2$ in $F^{(1)}_{n,m}(x) + F^{(2)}_{n,m}(x)$, we see that all the coefficients of $F_{pn,m}(x)$ are $1$ except the constant term which is $2$.  Note that $p > \max\big\{\!\deg F^{(1)}_{n,m}, \deg F^{(2)}_{n,m} \big\} + 1$ also ensures that degrees of all the terms in $F_{pn,m}(x)$ are different from $p-1$ modulo $p$.  This prohibits two terms in $F_{pn,m}(x)$ with degrees that differ by $1$.
\end{proof}

By definition \eqref{Fnm:eq}, for odd primes $p$ and $q$, we have
\[
F_{p,q}(x) = x^{p} + x^{q} + 2,
\]
so condition (P) holds.  
As a consequence of Theorem~\ref{fnm:thm} and Theorem~\ref{bigenoughgapresult}, we therefore obtain the following result, which can be viewed as a preliminary result before proving the stronger Theorem~\ref{arbitrarynumprimes:thm} in the next section.

\begin{corollary}
Let $t$ and $s$ be arbitrary positive integers.  If $t \equiv s \pmod{2}$, then there exist infinitely many choices of positive integers $n$ and $m$ with $\omega(n) = t$ and $\omega(m) = s$ such that the non-cyclotomic part of $\Phi_{n}(x) + \Phi_{m}(x)$ is irreducible. 
If $t \not\equiv s \pmod{2}$, then there exist infinitely many choices of positive integers $n$ and $m$ with $\omega(n) = t$ and $\omega(m) = s$ such that the non-cyclotomic part of $\Phi_{n}(x) \,\Phi_{m}(x) + 1$ is irreducible.
\end{corollary}

\section{Further applications of Theorem~\ref{factorizationthm}}

With Theorem~\ref{fnm:thm} at our disposal, we are ready to obtain some further results in the direction of Conjecture~\ref{nicolconj} and Conjecture~\ref{newconj} that arose in our development of $F_{n,m}(x)$.  In particular, we consider specific choices for $t$ and $s$ in Theorem~\ref{fnm:thm}.  We begin with a proof of Theorem~\ref{21case:thm} mentioned in the introduction.

\begin{proof}[Proof of Theorem~\ref{21case:thm}]
In Theorem~\ref{fnm:thm}, we take $n' = p_{1}p_{2}$ and $m' = q_{1}$.  From \eqref{Fnm:eq}, we obtain
\[
F_{n',m'}(x) = (x^{p_{1}p_{2}}+1)(x^{q_{1}}+1) + (x^{p_{1}}+1)(x^{p_{2}}+1).
\]
For $F_{n',m'}(x)$ to be of the form \eqref{basicFeq}, we want the six numbers
\begin{gather*}
p_{1}, \quad p_{2}, \quad q_{1}, \quad p_{1}+p_{2}, \quad p_{1}p_{2}, 
\quad \text{and} \quad p_{1}p_{2}+q_{1}
\end{gather*}
to be distinct.  Note that the only two even numbers of the six are $p_{1}+p_{2}$ and $p_{1}p_{2}+q_{1}$, and since $p_{1}p_{2} > p_{1}+p_{2}$ they are not equal.  The other ones are distinct primes and a number that is the product of two primes, so they are pairwise different as well.  Therefore, we see that $F_{n',m'}(x)$ is of the form \eqref{basicFeq}, and the result follows from Theorem~\ref{fnm:thm}~\ref{fnmthmseconditem}.
\end{proof}

\begin{theorem}\label{22case:thm}
If $p_{1}, p_{2}, q_{1}$ and $q_{2}$ are distinct odd primes and $e$ is a nonnegative integer, then the non-cyclotomic part of $\Phi_{2^{e}p_{1}p_{2}}(x) + \Phi_{2^{e}q_{1}q_{2}}(x)$ is irreducible provided $p_{2}$ is not equal to one of the eight numbers
\begin{gather*}
\dfrac{q_{1}q_{2}-q_{1}-q_{2}}{p_{1}}, \qquad
q_{1}+q_{2}-p_{1}, \qquad
\dfrac{q_{1}q_{2}+p_{1}-q_{1}}{p_{1}},\qquad
\dfrac{q_{1}q_{2}+p_{1}-q_{2}}{p_{1}}, \\[5pt]
\dfrac{q_{1}(q_{2}-1)}{p_{1}-1}, \quad 
\dfrac{q_{2}(q_{1}-1)}{p_{1}-1}, \quad
\dfrac{q_{1}q_{2}+p_{1}}{p_{1}-1}, \qquad \text{and} \qquad
\dfrac{q_{1}q_{2}+p_{1}-q_{1}-q_{2}}{p_{1}-1}.
\end{gather*}
\end{theorem}

\begin{comment}
One can see that each of eight numbers listed above can occur for $p_{2}$ 
by considering one of the choices for $(p_{1},p_{2},q_{1},q_{2})$ given by
\begin{gather*}
(3,29,5,23), \quad
(3,17,7,13), \quad
(3,11,5,7), \quad
(3,11,7,5), \\
(7,5,3,11), \quad
(7,5,11,3), \quad
(3,19,5,7), \quad
(3,13,5,7),
\end{gather*}
respectively.  Each of these examples also has the property that $p_{2}$ satisfies exactly one of the eight expressions for $p_{2}$ listed in the theorem, showing all eight are required in the theorem.  On the other hand, note that if we took $p_{2}$ to be the largest of the primes $p_{j}$ and $q_{j}$, then the possibilities $p_{2} = q_{1}(q_{2}-1)/(p_{1}-1)$ and $q_{2}(q_{1}-1)/(p_{1}-1)$ cannot occur.

Additionally, we note that Theorem~\ref{22case:thm} is not making any claims regarding Conjecture~\ref{nicolconj} when $p_{2}$ is equal to one of the eight numbers listed.  For the particular choices of $(p_{1},p_{2},q_{1},q_{2})$ in this comment, one can check using the method described in the proof of Lemma~\ref{multunique} and Example~\ref{algorithmends} that  as expected the polynomials $\Phi_{p_{1}p_{2}}(-x) + \Phi_{q_{1}q_{2}}(-x)$ have no multiples of the form \eqref{basicFeq}.  On the other hand, the polynomial $\Phi_{3 \cdot 29}(x) + \Phi_{5 \cdot 23}(x)$ does have a multiple of the form \eqref{basicFeq}.  We also note that the choice $p_{2} = (q_{1}q_{2}+p_{1}-q_{1}-q_{2})/(p_{1}-1)$ forces $\phi(2^{e} p_{1}p_{2}) = \phi(2^{e} q_{1}q_{2})$ so that the leading coefficient of $\Phi_{2^{e}p_{1}p_{2}}(\pm x) + \Phi_{2^{e}q_{1}q_{2}}(\pm x)$ is $2$ implying that there is no multiple of $\Phi_{2^{e}p_{1}p_{2}}(\pm x) + \Phi_{2^{e}q_{1}q_{2}}(\pm x)$ of the form \eqref{basicFeq}.
\end{comment}

\begin{proof}[Proof of Theorem~\ref{22case:thm}]
In Theorem~\ref{fnm:thm}, we take $n' = p_{1}p_{2}$ and $m' = q_{1}q_{2}$.  From \eqref{Fnm:eq}, we obtain
\[
F_{n',m'}(x) = (x^{p_{1}p_{2}}+1)(x^{q_{1}}+1)(x^{q_{2}}+1) + (x^{q_{1}q_{2}}+1)(x^{p_{1}}+1)(x^{p_{2}}+1).
\]
For $F_{n',m'}(x)$ to be of the form \eqref{basicFeq}, we want the fourteen numbers
\begin{gather*}
q_{1}, \quad q_{2}, \quad p_{1}p_{2}, \quad q_{1}+q_{2}, \quad p_{1}p_{2}+q_{1}, \quad p_{1}p_{2}+q_{2}, \quad p_{1}p_{2}+q_{1}+q_{2}, \\ 
p_{1}, \quad p_{2}, \quad q_{1}q_{2}, \quad p_{1}+p_{2}, \quad q_{1}q_{2}+p_{1}, \quad q_{1}q_{2}+p_{2}, \quad q_{1}q_{2}+p_{1}+p_{2}
\end{gather*}
to be distinct, eight of which are odd and six of which are even.  
To compare these, we note that the product of two odd primes is always greater than their sum and products of two primes cannot equal primes.  
By further parity and size considerations, we can reduce the possibilities that two of the fourteen numbers are equal to one of the eight equations
\begin{gather*}
q_{1}q_{2} = p_{1}p_{2}+q_{1}+q_{2}, \quad
p_{1}+p_{2} = q_{1}+q_{2}, \quad
q_{1}q_{2}+p_{1} = p_{1}p_{2}+q_{1}, \\[4pt]
q_{1}q_{2}+p_{1} = p_{1}p_{2}+q_{2}, \quad
q_{1}q_{2}+p_{2} = p_{1}p_{2}+q_{1}, \quad
q_{1}q_{2}+p_{2} = p_{1}p_{2}+q_{2}, \\[4pt]
q_{1}q_{2}+p_{1}+p_{2} = p_{1}p_{2}, \quad
q_{1}q_{2}+p_{1}+p_{2} = p_{1}p_{2}+q_{1}+q_{2}.
\end{gather*}
Solving for $p_{2}$ in each case gives the stated result.
\end{proof}

One last result we give in this direction before moving to the more general case is the following. 

\begin{theorem}\label{31case:thm}
If $p_{1}, p_{2}, p_{3}$ and $q_{1}$ are distinct odd primes and $e$ is a nonnegative integer, then the non-cyclotomic part of $\Phi_{2^{e}p_{1}p_{2}p_{3}}(x) + \Phi_{2^{e}q_{1}}(x)$ is irreducible provided the quadruple $(p_{1}, p_{2}, p_{3}, q_{1})$ does not satisfy one of 39 equations which are linear in each variable, each of which leads to a unique rational number $p_{3}$ for every choice of primes $p_{1}, p_{2}$ and $q_{1}$ with at least one such choice giving a prime $p_{3}$.
\end{theorem}

One can list the $39$ equations indicating what $p_{3}$ equals explicitly.  However, it is easier to write most of these equations in terms of what $q_{1}$ equals.  The 39 equations can be written as
\begin{gather*}
q_{1} = p_{1}+p_{2}+p_{3},  \quad
q_{1} = p_{1}p_{2}+p_{1}p_{3}+p_{2}p_{3},   \quad
q_{1} = p_{1}p_{2}p_{3}+p_{2}+p_{3},  \\[3pt]
q_{1} = p_{1}p_{2}p_{3}+p_{1}+p_{3},   \quad
q_{1} = p_{1}p_{2}p_{3}+p_{1}+p_{2},   \quad
q_{1} = p_{1}p_{2}-p_{1}p_{3}-p_{2}p_{3},  \\[3pt]
q_{1} = -p_{1}p_{2}+p_{1}p_{3}-p_{2}p_{3},  \quad
q_{1} = -p_{1}p_{2}-p_{1}p_{3}+p_{2}p_{3},  \quad
q_{1} = -p_{1}p_{3}+p_{1}+p_{2},  \\[3pt]
q_{1} = -p_{2}p_{3}+p_{1}+p_{2},  \quad
q_{1} = -p_{1}p_{2}+p_{1}+p_{3},  \quad
q_{1} = -p_{2}p_{3}+p_{1}+p_{3},  \\[3pt]
q_{1} = -p_{1}p_{2}+p_{2}+p_{3},  \quad
q_{1} = -p_{1}p_{3}+p_{2}+p_{3},  \quad
q_{1} = -p_{1}p_{2}+p_{1}p_{3}+p_{2}p_{3},  \\[3pt]
q_{1} = p_{1}p_{2}-p_{1}p_{3}+p_{2}p_{3},  \quad
q_{1} = p_{1}p_{2}+p_{1}p_{3}-p_{2}p_{3},  \quad
q_{1} = p_{1}p_{2}p_{3}-p_{1}p_{2}+p_{3},  \\[3pt]
q_{1} = p_{1}p_{2}p_{3}-p_{1}p_{2}-p_{1}p_{3}-p_{2}p_{3}+p_{3},  \quad
q_{1} = p_{1}p_{2}p_{3}-p_{1}p_{3}+p_{2},  \\[3pt]
q_{1} = p_{1}p_{2}p_{3}-p_{1}p_{2}-p_{1}p_{3}-p_{2}p_{3}+p_{2},  \quad
q_{1} = p_{1}p_{2}p_{3}-p_{2}p_{3}+p_{1},  \\[3pt]
q_{1} = p_{1}p_{2}p_{3}-p_{1}p_{2}-p_{1}p_{3}-p_{2}p_{3}+p_{1},  \quad
q_{1} = p_{1}p_{2}p_{3}-p_{1}p_{2}+p_{1}+p_{2}+p_{3},  \\[3pt]
q_{1} = p_{1}p_{2}p_{3}-p_{1}p_{3}+p_{1}+p_{2}+p_{3},  \quad
q_{1} = p_{1}p_{2}p_{3}-p_{2}p_{3}+p_{1}+p_{2}+p_{3},  \\[3pt]
q_{1} = p_{1}p_{2}p_{3}-p_{1}p_{3}-p_{2}p_{3}+p_{2}+p_{3},  \quad
q_{1} = p_{1}p_{2}p_{3}-p_{1}p_{2}-p_{2}p_{3}+p_{2}+p_{3},  \\[3pt]
q_{1} = p_{1}p_{2}p_{3}-p_{1}p_{2}-p_{1}p_{3}+p_{2}+p_{3},  \quad
q_{1} = p_{1}p_{2}p_{3}-p_{1}p_{3}-p_{2}p_{3}+p_{1}+p_{3},  \\[3pt]
q_{1} = p_{1}p_{2}p_{3}-p_{1}p_{2}-p_{2}p_{3}+p_{1}+p_{3},  \quad
q_{1} = p_{1}p_{2}p_{3}-p_{1}p_{2}-p_{1}p_{3}+p_{1}+p_{3},  \\[3pt]
q_{1} = p_{1}p_{2}p_{3}-p_{1}p_{3}-p_{2}p_{3}+p_{1}+p_{2},  \quad
q_{1} = p_{1}p_{2}p_{3}-p_{1}p_{2}-p_{2}p_{3}+p_{1}+p_{2},  \\[3pt]
q_{1} = p_{1}p_{2}p_{3}-p_{1}p_{2}-p_{1}p_{3}+p_{1}+p_{2},  \quad
q_{1} = p_{1}p_{2}p_{3}-p_{1}p_{2}-p_{1}p_{3}-p_{2}p_{3}+p_{1}+p_{2}+p_{3}, \\[4pt]
p_{3} = p_{1}p_{2}-p_{1}-p_{2},  \quad
p_{3} = \dfrac{p_{1}+p_{2}}{p_{1}-1},  \quad \text{ and } \quad
p_{3} = \dfrac{p_{1}+p_{2}}{p_{2}-1}.
\end{gather*}

The wording at the end of this last theorem is to emphasize a change that occurs when we extend these results further.  As we extend the approach further, considering $F_{n,m}(x)$ with $n$ and $m$ having more prime factors, we no longer can take one of the prime unknowns and make an exceptional list of equations in that unknown in terms of the others as in Theorem~\ref{22case:thm} and Theorem~\ref{31case:thm}.  To see this suppose  $n = p_{1}p_{2}p_{3}$ and $m = q_{1}q_{2}$.  Then observe that the equations
\begin{gather*}
p_{1} + p_{2} = q_{1} + q_{2}, \qquad
p_{1} + p_{3} = q_{1} + q_{2}, \qquad
p_{2} + p_{3} = q_{1} + q_{2}, \\[4 pt]
p_{1} + p_{2} + p_{3} = q_{1} \qquad \text{and} \qquad
p_{1} + p_{2} + p_{3} = q_{2}
\end{gather*}
arise from equating exponents in $F_{n,m}(x)$, making it impossible to always solve for one fixed prime variable in terms of the others.  For $n = p_{1}p_{2}p_{3}p_{4}$ and $m = q_{1}$, we similarly have the equations
\begin{gather*}
p_{1} + p_{2} + p_{4} = q_{1}, \qquad
p_{1} + p_{3} + p_{4} = q_{1}, \\[4 pt]
p_{2} + p_{3} + p_{4} = q_{1} \qquad \text{and} \qquad
p_{1} + p_{2} + p_{3} = p_{1} p_{2}
\end{gather*}
arising from equating exponents in $F_{n,m}(x)$.  Thus, similar results to Theorem~\ref{22case:thm} and Theorem~\ref{31case:thm} begin to take a slightly different form when we include more prime variables.  There is another phenomenon that occurs as we move forward that is associated with the example with $n$ divisible by $165$ in Comment~\ref{165comment}.  For example, in the case where $n = p_{1}p_{2}p_{3}p_{4}$ and $m = q_{1}$, the equation 
\begin{equation}\label{moreon165:eq}
1 + p_{1} p_{2} = p_{2} + p_{3},
\end{equation}
which has solutions in odd primes, does \textit{not} arise from equating exponents in $F_{n,m}(x)$ because of the summand $1$, but when we multiply both sides by $p_{4}$ we obtain an equation that does arise when equating exponents in $F_{n,m}(x)$.  With these comments in mind, some care is needed in establishing the following.

\begin{lemma}\label{maintsfnmthm}
Fix positive integers $t$ and $s$.  Set $m = p'_{1} \cdots p'_{s}$ and $n = p'_{s+1} \cdots p'_{s+t}$ where $p'_{1}, \ldots, p'_{s+t}$ are variables representing $s+t$ distinct odd primes.  Then there are sets of equations $\mathcal E'$ and $\mathcal E'_{2}, \mathcal E'_{3}, \ldots, \mathcal E'_{s+t}$ satisfying 
\[
\mathcal E' = \mathcal E'_{2} \,\textstyle\bigcup\, \mathcal E'_{3} \,\textstyle\bigcup\, \cdots \,\textstyle\bigcup\, \mathcal E'_{s+t},
\]
\[
|\mathcal E'| \le 2^{2^{t} + 2^{s}-1} + 2^{2^{t} + 1} + 2^{2^{s} + 1},
\]
for each equation in $\mathcal E'_{j}$, there are $A(p'_{1}, \ldots, p'_{j-1})$ and $B(p'_{1}, \ldots, p'_{j-1})$ in $\mathbb Z[p'_{1},\ldots,p'_{j-1}]$, both linear in each variable, such that the equation takes the form
\begin{equation}\label{eqform}
p'_{j} = \dfrac{A(p'_{1}, \ldots, p'_{j-1})}{B(p'_{1}, \ldots, p'_{j-1})},
\end{equation}
and $F_{n,m}(x)$ is of the form \eqref{basicFeq} unless the primes $p'_{1} \cdots p'_{s+t}$ satisfy an equation in $\mathcal E'$ of the from \eqref{eqform} with $B(p'_{1}, \ldots, p'_{j-1}) \ne 0$.
\end{lemma}

\begin{comment}
We do not care about the relative sizes of the $p'_{j}$ in this lemma.
We treat these primes as unknowns and seek equations they must satisfy if $F_{n,m}(x)$ is \textit{not} of the form \eqref{basicFeq}.  We do not mind including more equations than needed for this purpose, as long as solutions to our equations include all possible choices of $n$ and $m$ for which $F_{n,m}(x)$ is not of the form \eqref{basicFeq}.  In particular, we will certainly consider equations below that have no solutions in odd primes.  For example, in the proof of Theorem~\ref{22case:thm}, one might naturally have started with $\binom{14}{2} = 91$ equations, but only $8$ were needed in the end.  We will be content with considering all $91$ in this proof.
\end{comment}

\begin{proof}[Proof of Lemma~\ref{maintsfnmthm}]
For $1 \le j \le t+s$, we let $\mathcal S_{j}$ denote those sum of divisors of $n$ and $m$ occurring as non-zero exponents in $F_{n,m}(x)$ which only involve the first $j$ primes in our ordering.  Thus, $\mathcal S_{1} = \{ p'_{1} \}$ and either $\mathcal S_{2} = \{ p'_{1}, p'_{2}, p'_{1}+p'_{2}, p'_{1}p'_{2} \}$ or $\mathcal S_{2} = \{ p'_{1}, p'_{2} \}$, depending on whether $s > 1$ or $s = 1$, respectively.  We set $\mathcal E^{(0)}_{j}$ to be the equations obtained by setting two distinct elements of $\mathcal S_{j}$ equal.  
We will want some additional equations that use the additional factor $x+1$ that comes from considering the divisor $d = 1$ not present in \eqref{Fnm:eq}.  These additional equations account for equations like the one given in \eqref{moreon165:eq} that will occur in the argument below.  So we define $\mathcal E^{(1)}_{j}$ to be equations, only involving the first $j$ primes, formed by equating two distinct non-zero exponents from 
\[
\prod_{\substack{d \mid n \\ \omega(d) \equiv 0 {\hskip -6pt}\pmod{2}}} {\hskip -8pt}\big( x^{d} + 1 \big) 
\,+ \prod_{\substack{d \mid n \\ \omega(d) \equiv 1 {\hskip -6pt}\pmod{2}}} {\hskip -8pt}\big( x^{d} + 1 \big)
\]
and $\mathcal E^{(2)}_{j}$ to be equations, only involving the first $j$ primes, formed by equating two distinct non-zero exponents from 
\[
\prod_{\substack{d' \mid m \\ \omega(d') \equiv 0 {\hskip -6pt}\pmod{2}}} {\hskip -8pt}\big( x^{d'} + 1 \big) 
\,+ \prod_{\substack{d' \mid m \\ \omega(d') \equiv 1 {\hskip -6pt}\pmod{2}}} {\hskip -8pt}\big( x^{d'} + 1 \big).
\]
Every exponent in the terms of the expansion of the products above defining $\mathcal E^{(1)}_{j}$ and $\mathcal E^{(2)}_{j}$ either includes $1$ coming from the $x$ in the factor $x+1$ or else appears also in $\mathcal S_{t+s}$, that is as an exponent in a term of $F_{n,m}(x)$. 

An important observation for us is that if $\sigma_{1}$ and $\sigma_{2}$ are both a sum of divisors of $n$ (similarly for $m$), each coming from an exponent of one of the products in the definition of $\mathcal E^{(1)}_{j}$, and each term of $\sigma_{1}-\sigma_{2}$ is divisible by the prime $p'_{j}$, then there are $\sigma'_{1}$ and $\sigma'_{2}$ which are both a sum of divisors of $n$, each coming from one of the products in the definition of $\mathcal E^{(1)}_{j}$, such that $\sigma_{1}-\sigma_{2} = p'_{j}\sigma'_{1}-p'_{j}\sigma'_{2}$ and the equation $\sigma'_{1} = \sigma'_{2}$ is in $\mathcal E^{(1)}_{j-1}$.
To see this notice that there must be a $\sigma_{0}$ that is a sum of divisors of $n$ such that $\sigma_{1} = p_{j} \sigma'_{1} + \sigma_{0}$ and $\sigma_{2} = p_{j} \sigma'_{2} + \sigma_{0}$.  Here, we will have $\sigma_{0} = 0$ if $\sigma_{1}$ and $\sigma_{2}$ are exponents coming from different products in the definition of $\mathcal E^{(1)}_{j}$.  Each term in $p_{j} \sigma'_{1}$ is necessarily squarefree and all terms have the same parity in its number of prime factors, so the same is true for the terms in $\sigma'_{1}$.  Similarly, the terms in $\sigma'_{2}$ are squarefree and all have the same parity in the number of their prime factors.  Every prime factor of any term in $\sigma'_{1}$ or $\sigma'_{2}$ is necessarily from among $p'_{1}, \ldots, p'_{j-1}$, so it follows that $\sigma'_{1} = \sigma'_{2}$ is in $\mathcal E^{(1)}_{j-1}$.

We define $\mathcal E_{j} = \mathcal E^{(0)}_{j} \cup \mathcal E^{(1)}_{j} \cup \mathcal E^{(2)}_{j}$.  In way of examples, we note that $\mathcal E_{1} = \{  \}$, and we see that in the case that $s > 1$ we have
\begin{align*}
\mathcal E_{2} &=
\{ p'_{1} = p'_{2}, p'_{1} = p'_{1}+p'_{2}, p'_{1} = p'_{1}p'_{2}, p'_{2} = p'_{1}+p'_{2}, p'_{2} = p'_{1}p'_{2},  p'_{1}+p'_{2} = p'_{1}p'_{2}, \\
&\qquad 1 = p'_{1}, 1 = p'_{2}, 1 = p'_{1}+p'_{2}, 1 = p'_{1}p'_{2}, 1 = 1 + p'_{1}p'_{2},
p'_{1} =  1 + p'_{1}p'_{2},  \\
&\qquad p'_{2} = 1 + p'_{1}p'_{2}, p'_{1}+p'_{2} = 1 + p'_{1}p'_{2}, p'_{1}p'_{2} = 1 + p'_{1}p'_{2} \}, 
\end{align*}
and if $s = 1$ we have
\[
\mathcal E_{2} =
\{ p'_{1} = p'_{2}, 1 = p'_{2}, 1 = p'_{1} \}.
\]
None of these equations have solutions in distinct odd primes $p'_{1}$ and $p'_{2}$, though one can expect many equations in $\mathcal E_{j}$ with solutions as $j$ increases. 
We also note that
\[
\mathcal E_{1} \subset \mathcal E_{2} \subset \cdots \subset \mathcal E_{t+s},
\]
and the set $\mathcal E_{t+s}$ includes the set of equations $\mathcal E^{(0)}_{t+s}$ which have the property that if none of them are satisfied for a choice of $p'_{j}$, then $F_{n,m}(x)$ will be of the form \eqref{basicFeq}.  

We bound the number of elements in $\mathcal E_{t+s}$.  
With an eye toward counting first the number of elements of $\mathcal S_{t+s}$ coming from multiplying out the first two products in \eqref{Fnm:eq}, observe that there are $2^{t-1}-1$ divisors of $n$ that are not $1$ and contain an even number of prime divisors and $2^{s-1}$ divisors of $m$ that contain an odd number of prime divisors.  If we expand the first two products in \eqref{Fnm:eq}, each term will consist of a sum of a subset of these $2^{t-1}-1 + 2^{s-1}$ divisors, resulting in $2^{2^{t-1}-1 + 2^{s-1}}$ terms including the term $1$ of degree $0$.  This then accounts for $2^{2^{t-1}-1 + 2^{s-1}}-1$ elements of $\mathcal S_{t+s}$.  Similarly, we get $2^{2^{s-1}-1 + 2^{t-1}}-1$, the same number, of elements of $\mathcal S_{t+s}$ arising from the last two products in \eqref{Fnm:eq}.  Therefore, we deduce that
\[
\big| \mathcal S_{t+s} \big| = 2 \big( 2^{2^{t-1} + 2^{s-1} -1}-1 \big)
= 2^{2^{t-1} + 2^{s-1}}-2.
\]
We deduce that the number of equations arising from setting two different elements of $\mathcal S_{t+s}$ equal is 
\[
\big| \mathcal E^{(0)}_{t+s} \big| = \binom{|\mathcal S_{t+s}|}{2} \le 2^{2^{t} + 2^{s}-1}.
\]
Similarly, we see that
\[
\big| \mathcal E^{(1)}_{t+s} \big| = \binom{2^{2^{t-1}}+2^{2^{t-1}}-2}{2} \le 2^{2^{t} + 1}
\qquad
\text{and}
\qquad
\big| \mathcal E^{(2)}_{t+s} \big| = \binom{2^{2^{s-1}}+2^{2^{s-1}}-2}{2} \le 2^{2^{s} + 1}.
\]
Many of the equations appearing in $\mathcal E^{(1)}_{t+s}$ and $\mathcal E^{(2)}_{t+s}$ already appear in $\mathcal E^{(0)}_{t+s}$, but we do not mind the overcounting.  
Thus, we only use that
\[
\big| \mathcal E_{t+s} \big| \le 2^{2^{t} + 2^{s}-1} + 2^{2^{t} + 1} + 2^{2^{s} + 1}.
\]

We proceed, as $j$ varies from $1$ to $t+s$, by successively considering equations in $\mathcal E_{j}$, solving each new equation formed, with each equation involving only the first $j$ primes.  As noted above, there are no solutions for the equations in $\mathcal E_{1}$ and $\mathcal E_{2}$ involving the first two odd primes in our ordering.  Observe that the proof of Theorem~\ref{21case:thm} implies that, if $s = 1$, then there are also no solutions for the equations in $\mathcal E^{(0)}_{3}$, and the same holds for $\mathcal E^{(1)}_{3}$ and $\mathcal E^{(2)}_{3}$ and so $\mathcal E_{3}$.  By symmetry,  if $s = 2$, then also there are no solutions to equations in $\mathcal E_{3}$.  On the other hand, the list of equations arising from Theorem~\ref{31case:thm} indicates that if $s \ge 3$, then this is no longer true as the last three of the 39 equations listed after Theorem~\ref{31case:thm} lead to the possible equations 
\[
p'_{1} + p'_{2} + p'_{3} = p'_{1} p'_{2}, \qquad
p'_{1} + p'_{2} + p'_{3} = p'_{1} p'_{3}, \qquad \text{and} \qquad
p'_{1} + p'_{2} + p'_{3} = p'_{2} p'_{3}
\]
coming from $\mathcal E^{(0)}_{3}$, which will have solutions in primes, seen by considering some arrangement of $3$, $7$ and $11$ for each equation.  Furthermore, there are additional equations in $\mathcal E^{(2)}_{3}$ in this case with solutions, namely
\begin{gather*}
1 + p'_{1}p'_{2} = p'_{1} + p'_{3}, \quad
1 + p'_{1}p'_{2} = p'_{2} + p'_{3}, \quad
1 + p'_{1}p'_{3} = p'_{1} + p'_{2}, \\[4pt]
1 + p'_{1}p'_{3} = p'_{2} + p'_{3}, \quad
1 + p'_{2}p'_{3} = p'_{1} + p'_{2} \quad \text{and} \quad
1 + p'_{2}p'_{3} = p'_{1} + p'_{3}, 
\end{gather*}
seen by considering some arrangement of $3$, $5$ and $11$ for each equation.

For each new $j \in [1, t+s]$ that we consider, we look at the equations in $\mathcal E_{j}$ and discard ones that have already occurred when we dealt with $\mathcal E_{i}$ for $i < j$.  Necessarily then every new equation will involve $p'_{j}$.  The idea is to solve for $p'_{j}$ in each such equation.  However, when we solve for $p'_{j}$ in each such equation, we need to treat the possibility that the resulting denominator can be $0$ depending on the prime choices for $p'_{1}, \ldots, p'_{j-1}$.  
By rearranging, we can write each new equation as
\begin{equation}\label{apjeqb:eq}
B(p'_{1},\ldots,p'_{j-1}) \,p'_{j} = A(p'_{1},\ldots,p'_{j-1}),
\end{equation}
where $A(p'_{1}, \ldots, p'_{j-1})$ and $B(p'_{1}, \ldots, p'_{j-1})$ are in $\mathbb Z[p'_{1},\ldots,p'_{j-1}]$ and both are linear in each of the variables $p'_{1},\ldots,p'_{j-1}$.  
We can view \eqref{apjeqb:eq} as one of the no more than $2^{2^{t} + 2^{s}-1} + 2^{2^{t} + 1} + 2^{2^{s} + 1}$ equations mentioned in the theorem.  
As \eqref{apjeqb:eq} came from setting two different exponents equal, we do not have that both $B(p'_{1},\ldots,p'_{j-1})$ and $A(p'_{1},\ldots,p'_{j-1})$ as polynomials are identically $0$.  If the terms involving $p'_{j}$ cancelled when rearranging, then $B(p'_{1},\ldots,p'_{j-1})$ would be identically $0$ and \eqref{apjeqb:eq} is simply $A(p'_{1},\ldots,p'_{j-1}) = 0$, which would have been an equation considered in $\mathcal E_{j-1}$.  So we are only interested in the case where $B(p'_{1},\ldots,p'_{j-1})$ is not identically $0$.  We do not need to treat the special case where $A(p'_{1},\ldots,p'_{j-1})$ is identically $0$ separately from other cases for $A(p'_{1},\ldots,p'_{j-1})$.
If $B(p'_{1},\ldots,p'_{j-1}) p'_{j}$ is expanded, all of its terms involve $p'_{j}$. 
As a consequence the terms in $B(p'_{1},\ldots,p'_{j-1}) p'_{j}$ are only divisors of $n$ if $j > s$ or are only divisors of $m$ if $j \le s$.  
Since each term includes the factor $p'_{j}$, clearly $1$ is not a term in $B(p'_{1},\ldots,p'_{j-1}) p'_{j}$.  We deduce that $B(p'_{1},\ldots,p'_{j-1}) p'_{j}$ arises as a difference of two exponents appearing in $F_{n,m}(x)$ and consequently $B(p'_{1},\ldots,p'_{j-1}) = 0$ is an equation in $\mathcal E^{(1)}_{j-1}$ if $p'_{j} \mid n$ or in $\mathcal E^{(2)}_{j-1}$ if $p'_{j} \mid m$ (see the discussion involving $\sigma_{1}, \sigma_{2}, \sigma'_{1}$ and $\sigma'_{2}$ above).  As a consequence, we would have already accounted for the possibility that $B(p'_{1},\ldots,p'_{j-1}) = 0$ and hence that $B(p'_{1},\ldots,p'_{j-1}) = A(p'_{1},\ldots,p'_{j-1}) = 0$ in \eqref{apjeqb:eq} when we worked with equations in $\mathcal E_{i}$ for some $i < j$.  In other words, solutions to \eqref{apjeqb:eq} where $B(p'_{1},\ldots,p'_{j-1}) = 0$ have already been eliminated when solving $B(p'_{1},\ldots,p'_{j-1}) = 0$ in $\mathcal E_{i}$ for some $i < j$.  We can therefore deduce from \eqref{apjeqb:eq} that we necessarily have \eqref{eqform}, 
where we can ignore choices of $p'_{1},\ldots,p'_{j-1}$ where $B(p'_{1},\ldots,p'_{j-1}) = 0$.  
By setting $\mathcal E'_{j}$ to be the equations from $\mathcal E_{j}$ obtained of the form \eqref{eqform} and noting that $|\mathcal E'| \le |\mathcal E|$, the lemma follows.
\end{proof}

\begin{proof}[Proof of Theorem~\ref{arbitrarynumprimes:thm}]
By Theorem~\ref{fnm:thm}, it suffice to show the number of choices for distinct odd primes $p'_{1},\ldots,p'_{s+t}$ satisfying one of the equations in $\mathcal E'$ in Lemma~\ref{maintsfnmthm}, with each $p'_{j} \le X$, is bounded by $O_{t,s}(\pi(X)^{t+s-1})$.  With equations of the form \eqref{eqform} for $2 \le j \le t+s$, we proceed as follows to bound the number of odd primes $p'_{j} \le X$ satisfying at least one of the equations in $\mathcal E'$.  We consider each $j_{0} \in \{ 2, \ldots, t+s \}$ in turn and restrict to those equations in \eqref{eqform} corresponding to $j = j_{0}$.  
We allow arbitrary distinct odd prime values for $p'_{i}$ with $1 \le i \le t+s$ and $i \ne j_{0}$ in $\binom{\pi(X)}{t+s-1}$ ways and then take $p'_{j_{0}}$ to be one of the solutions given by \eqref{eqform} with $j = j_{0}$ and with the denominator $B(p'_{1},\ldots,p'_{j-1}) \ne 0$.  Such a solution for $p'_{j_{0}}$ need not be an odd prime distinct from our choices for $p'_{i}$ with $1 \le i \le t+s$ and $i \ne j_{0}$ and it might not even be an integer, but this will only mean that using such a solution in our count gives us an overestimate.  Each equation in \eqref{eqform} then gives rise to at most $\binom{\pi(X)}{t+s-1}$ distinct odd prime choices for $p'_{1}, \ldots, p'_{t+s}$ satisfying at least one of the equations in $\mathcal E'$.  Given the bound on $|\mathcal E'|$, the proof of the theorem is complete.
\end{proof}

We turn now to establishing Theorem~\ref{maingeneralthm2}.
For $t$ and $s$ positive integers, let $\mathcal Q$ denote the pairs of odd positive integers $(n,m)$ with $n \le X$, $m \le X$, $nm$ squarefree, $\omega(n)= t$ and $\omega(m) = s$.
To establish Theorem~\ref{maingeneralthm2}, we show that for every $\varepsilon > 0$, there is a $\delta > 0$ such that if $t+s \le \delta \sqrt{\log\log X}$, then the non-cyclotomic part of $\Phi_{n}(x)+\Phi_{m}(x)$ is irreducible if $t \equiv s \pmod{2}$ and the non-cyclotomic part of $\Phi_{n}(x)\Phi_{m}(x)+1$ is irreducible if $t \not\equiv s \pmod{2}$ for all but at most $\varepsilon |\mathcal Q|$ pairs $(n,m)$ in $\mathcal Q$.  From Theorem~\ref{fnm:thm}, it suffices to show that if $t+s \le \delta \sqrt{\log\log X}$, then $F_{n,m}(x)$ is of the form \eqref{basicFeq} for all but at most $\varepsilon |\mathcal Q|$ pairs $(n,m)$ in $\mathcal Q$.  To accomplish this task, we need the order of magnitude of $|\mathcal Q|$.  
For this, one can take advantage of a theorem of A.~Selberg \cite{selberg1954} (also, see \cite[Section~7.4]{montgomeryvaughan} and \cite[Chapter~II.6]{tenenbaum}).  Estimates for $\pi_{t}(X)$, which is the number of squarefree numbers $n \le X$ consisting of $t$ (distinct) prime factors go back to Landau (cf.~\cite[Section~56]{landau}) and L.~G.~Sathe's work, beginning with \cite{sathe}, prior to Selberg's work helped develop interest in similar estimates.  These works specifically address estimates for $\pi_{t}(X)$ (see also \cite[Section~7.4.1,~Problem~4]{montgomeryvaughan}).  Borrowing from these, we have the following.

\begin{lemma}\label{landautypelem}
For every $\varepsilon > 0$, there is an $X_{0} = X_{0}(\varepsilon)$ and a $\delta = \delta(\varepsilon) > 0$ such that if $X \ge X_{0}(\varepsilon)$ and $t$ is a positive integer $< \delta \log\log X$, then the number $\pi_{t}(X)$ of squarefree numbers $n \le X$ for which $\omega(n) = t$ satisfies 
\[
(1-\varepsilon) \dfrac{X \,(\log\log X)^{t-1}}{(t-1)!\,\log X} < \pi_{t}(X) < (1+\varepsilon) \dfrac{X \,(\log\log X)^{t-1}}{(t-1)!\,\log X}.
\]
\end{lemma}

We are interested in a similar estimate for $\pi^{o}_{t}(X)$ which we use to denote the number of \textit{odd} squarefree $n \le X$ satisfying $\omega(n) = t$.  Since every even squarefree $n = 2m \le X$ with $\omega(n) = t$ is such that $m$ is squarefree and $\omega(m) = t-1$, we see that there are no more than $\pi_{t-1}(X/2)$ such $n$, which by 
Lemma~\ref{landautypelem}, is small for small $\delta$ compared to $\pi_{t}(X)$.  We deduce then from Lemma~\ref{landautypelem} that the same result holds with $\pi_{t}(X)$ replaced by $\pi^{o}_{t}(X)$.  More precisely, we have the following. 

\begin{lemma}\label{newlandautypelem}
For every $\varepsilon > 0$, there is an $X_{0} = X_{0}(\varepsilon)$ and a $\delta = \delta(\varepsilon) > 0$ such that if $X \ge X_{0}$ and $t$ is a positive integer $< \delta \log\log X$, then the number $\pi^{o}_{t}(X)$ of odd squarefree numbers $n \le X$ for which $\omega(n) = t$ satisfies 
\begin{equation}\label{landautypelem:eq1}
(1-\varepsilon) \dfrac{X (\log\log X)^{t-1}}{(t-1)!\,\log X} < \pi^{o}_{t}(X) < (1+\varepsilon) \dfrac{X (\log\log X)^{t-1}}{(t-1)!\,\log X}.
\end{equation}
\end{lemma}

We extend the above lemma and its argument to a more general situation as follows.

\begin{lemma}\label{newlandautypelem2}
Fix $\varepsilon > 0$.  Let $X_{0} = X_{0}(\varepsilon)$ be sufficiently large and $\delta = \delta(\varepsilon) > 0$ be sufficiently small.
Let $t$ be a positive integer $< \delta \sqrt{\log\log X}$, and let $\mathcal P = \mathcal P(X) \subset [1,X]$ be a finite subset of the primes for which $|\mathcal P| \le \exp(\exp(\sqrt{\log\log X}\,))$ where $X \ge X_{0}$.
Then the number $\pi^{\mathcal P}_{t}(x)$ of squarefree numbers $n \le X$ for which $n$ has no prime factors in $\mathcal P$ and $\omega(n) = t$ satisfies 
\begin{equation}\label{landautypelem:eq2}
(1-\varepsilon) \dfrac{X (\log\log X)^{t-1}}{(t-1)!\,\log X} < \pi^{\mathcal P}_{t}(X) < (1+\varepsilon) \dfrac{X (\log\log X)^{t-1}}{(t-1)!\,\log X}.
\end{equation}
\end{lemma}

\begin{proof}
Since $\pi^{\mathcal P}_{t}(X) \le \pi_{t}(X)$, the upper bound follows from Lemma~\ref{landautypelem}.  
Furthermore, if $\mathcal P'$ is a subset of $\mathcal P$, then $\pi^{\mathcal P'}_{t}(X) \ge \pi^{\mathcal P}_{t}(X)$, so we need only consider $|\mathcal P|$ large (near the upper bound on $|\mathcal P|$ given in the lemma).  We do so.
Observe that $\exp(\exp(\sqrt{\log\log X}\,))$ is small compared to $X/\log X$. 
If $t =1$, then 
\[
\pi^{\mathcal P}_{t}(X) \sim \pi(X) \sim X/\log X,
\]
and \eqref{landautypelem:eq2} follows.  We suppose then that $t \ge 2$.  
The number of primes 
$\le 2\,|\mathcal P|\,\log|\mathcal P|$ is at least $|\mathcal P|$.  Hence, by Lemma~\ref{landautypelem}, the number of squarefree numbers $n \le X$ for which $n$ has a prime factor in $\mathcal P$ and $\omega(n) = t$ is bounded by
\[
\sum_{p \in \mathcal P} \pi_{t-1}\bigg( \dfrac{X}{p} \bigg) 
\le \sum_{p \le 2|\mathcal P|\log|\mathcal P|} \pi_{t-1}\bigg( \dfrac{X}{p} \bigg)
\le \dfrac{1+\varepsilon}{(t-2)!} \sum_{p \le 2|\mathcal P|\log|\mathcal P|} \dfrac{X (\log\log X)^{t-2}}{p \log(X/p)}.
\]
The upper bound on $|\mathcal P|$ implies that for $p \le 2|\mathcal P|\log|\mathcal P|$ we have
\begin{align*}
\log (X/p) 
&\ge \log X - 2\exp\big(\sqrt{\log\log X}\,\big) \\[5pt]
&= \log X \,\bigg( 1 - \dfrac{2\exp\big(\sqrt{\log\log X}\,\big)}{\log X} \bigg) = (1+o(1)) \log X,
\end{align*}
where, as defined in the introduction, $o(1)$ indicates a function that tends to $0$ as $X$ goes to infinity.
Also, we see that
\[
\sum_{p \le 2|\mathcal P|\log|\mathcal P|} \dfrac{1}{p} \le 2 \log\log |\mathcal P| \le 2 \sqrt{\log\log X}
= \dfrac{2 \delta \log\log X}{\delta \sqrt{\log\log X}} \le \dfrac{2 \delta \log\log X}{t-1}.
\]
With $X_{0}$ large enough and $\delta$ small enough, we obtain that
\[
\sum_{p \in \mathcal P} \pi_{t-1}\bigg( \dfrac{X}{p} \bigg) 
\le \dfrac{\varepsilon}{2} \cdot \dfrac{X (\log\log X)^{t-1}}{(t-1)!\,\log X}.
\]
Taking $\varepsilon$ in Lemma~\ref{landautypelem} to be $\varepsilon/2$, we deduce the lemma.
\end{proof}

The bound on $|\mathcal P|$ in Lemma~\ref{newlandautypelem2} is much more than we will need.  We only require $|\mathcal P| < \delta \sqrt{\log\log X}$.  Our interest is in the following.

\begin{lemma}\label{Qsizelemma}
For $t$ and $s$ positive integers, let $\mathcal Q$ denote the pairs of odd positive integers $(n,m)$ with $n \le X$, $m \le X$, $nm$ squarefree, $\omega(n)= t$ and $\omega(m) = s$.  If $t+s = o\big( \sqrt{\log\log X} \,\big)$, then
\[
|\mathcal Q| \sim \dfrac{X^{2} (\log\log X)^{t+s-2}}{(t-1)! (s-1)!\,\log^{2} X}.
\]
\end{lemma}

\begin{proof}
By Lemma~\ref{newlandautypelem}, the number of ways we can choose odd $m$ with $m \le X$, $m$ squarefree and $\omega(m) = s$ is
\[
\pi_{s}^{o}(X) \sim \dfrac{X (\log\log X)^{s-1}}{(s-1)! \log X}.
\]
For each such choice for $m$, we take $\mathcal P(m)$ to be the $s$ prime divisors of $m$ and set $\mathcal P = \mathcal P(m) \cup \{ 2 \}$.   Then the inequality on $|\mathcal P|$ in Lemma~\ref{newlandautypelem2} is satisfied.  Therefore, we deduce from Lemma~\ref{newlandautypelem2} that for each chosen $m$, there are 
\[
\pi^{\mathcal P}_{t}(X) \sim \dfrac{X (\log\log X)^{t-1}}{(t-1)! \log X}
\]
choices for $n$ odd with $n \le X$, $nm$ squarefree and $\omega(n) = t$.  The lemma follows.
\end{proof}

\begin{proof}[Proof of Theorem~\ref{maingeneralthm2}]
From Theorem~\ref{fnm:thm}, we want to show that for every $\varepsilon > 0$, there is a $\delta > 0$ such that if $t+s \le \delta \sqrt{\log\log X}$, then $F_{n,m}(x)$ is of the form \eqref{basicFeq} for all but at most $\varepsilon |\mathcal Q|$ pairs $(n,m)$ in $\mathcal Q$, where we are taking $X$ to be sufficiently large.  We consider $\delta < 1/2$.  We make use of Lemma~\ref{maintsfnmthm} to estimate the number of $(n,m)$ in $\mathcal Q$ for which $F_{n,m}(x)$ is not of the form \eqref{basicFeq}.  We use the notation in Lemma~\ref{maintsfnmthm} but order the primes so that
\[
p'_{s} > p'_{s-1} > \cdots > p'_{1}
\qquad \text{and} \qquad
p'_{t+s} > p'_{t+s-1} > \cdots > p'_{s+1}.
\]  

We start with the primes dividing $m$.
For each $j \in \{ 2, 3, \ldots, s \}$, we bound the number of pairs $(n,m) \in \mathcal Q$ for which $p'_{j}$ satisfies an equation in $\mathcal E'_{j}$ by choosing $n \le X$ odd and squarefree with $\omega(n) = t$ in $\pi^{o}_{t}(X)$ ways and then counting the number of ways of choosing $m = p'_{j}m' \le X$ odd and squarefree with $\omega(m') = s-1$ and $p'_{j}$ satisfying an equation in $\mathcal E'_{j}$.  We do not mind overcounting so we do not need concern ourselves with the coprimality of $p'_{j}$, $m'$ and $n$ and whether $p'_{j}$ is between the $(j-1)^{\rm st}$ and $j^{\rm th}$ prime factors of $m'$.  

Note that $2 \le j \le s$, so $s \ge 2$ and $t \ge 1$.  We see then that
\[
|\mathcal E'_{j}| \le |\mathcal E'| \le 2^{2^{t} + 2^{s}-1} + 2^{2^{t} + 1} + 2^{2^{s} + 1}
\le 2^{2^{t+s} -3} + 2^{2^{t+s} - 5} + 2^{2^{t+s} - 3} \le 2^{2^{t+s}}.
\]
Set $D = 2^{2^{\delta\sqrt{\log\log X}+1}}$.  Observe that
\[
\log D < 2^{\delta\sqrt{\log\log X}+1} = o(\log X)
\]
as $X \rightarrow \infty$.
We consider $m' \le X/D$ and $m' > X/D$ separately.   For $m' \le X/D$, we count the number of pairs $(p'_{j},m')$, where $p'_{j}$ satisfies an equation in $\mathcal E'_{j}$.   
For each of the $\pi^{o}_{s-1}(X/D)$ choices of $m' \le X/D$, there are at most $|\mathcal E'_{j}|$ primes $p'_{j}$ satisfying an equation in $\mathcal E'_{j}$.  Thus, the total number of such pairs $(p'_{j},m')$ is bounded by
\begin{align*}
\pi^{o}_{s-1}(X/D) \cdot 2^{2^{t+s}} &\le \dfrac{X (\log\log (X/D))^{s-2} \cdot 2^{2^{\delta\sqrt{\log\log X}}}}{(s-2)!\,\log (X/D)\,2^{2^{\delta\sqrt{\log\log X}+1}}} \\[5pt]
&\le \dfrac{2X (\log\log X)^{s-2}}{(s-2)!\,(\log X)\,2^{2^{\delta\sqrt{\log\log X}}}} \\[5pt]
&\le \dfrac{X (\log\log X)^{s-2}}{(s-1)!\,\log X}.
\end{align*}

In the case that $m' > X/D$, observe that $p'_{j} \le D$.  In this case, we do not concern ourselves with whether $p'_{j}$ satisfies an equation in $\mathcal E'_{j}$ but instead take advantage of the fact that $p'_{1} \cdots p'_{j} \le D^{j}$, which is small compared to $X$.  It suffices to bound the number of positive odd squarefree integer pairs $(m'_{1},m'_{2})$ where $m'_{1} \le D^{j}$, $\omega(m'_{1}) = j$, $m'_{2} \le X/m'_{1}$, and $\omega(m'_{2}) = s-j$, which is again an overcount of what is truly needed.  
For each of the $\pi^{o}_{j}(D^{j})$ choices for $m'_{1}$, we see that there are at most $\pi^{o}_{s-j}(X/m'_{1})$ choices for $m'_{2}$.  Here, the interpretation in the case $j = s$ of $\pi^{o}_{s-j}(X/m'_{1})$ is $\pi^{o}_{0}(X/m'_{1}) = 1$ as necessarily $m'_{2} = 1$ in this case.  
For $2 \le j \le s-1$, we can bound the number of pairs $(m'_{1},m'_{2})$ by
\[
\sum_{\substack{m'_{1} \le D^{j} \\ m'_{1} \text{ odd squarefree} \\ \omega(m'_{1}) = j}} \pi^{o}_{s-j}(X/m'_{1})
\le \sum_{\substack{m'_{1} \le D^{j} \\ m'_{1} \text{ odd squarefree} \\ \omega(m'_{1}) = j}} \dfrac{2X (\log\log X)^{s-j-1}}{(s-j-1)!\,m'_{1} \,\log X},
\]
where we are using that $X$ is large and
\[
\log m'_{1} \le j \log D \le \delta\sqrt{\log\log X} \,2^{\delta\sqrt{\log\log X}+1} = o(\log X).
\]
For an appropriate large $U_{0}$ and $X$ sufficiently large, we obtain
\begin{align*}
\sum_{\substack{m'_{1} \le D^{j} \\ m'_{1} \text{ odd squarefree} \\ \omega(m'_{1}) = j}} \dfrac{1}{m'_{1}}
&\le U_{0} + \int_{U_{0}}^{D^{j}} \dfrac{1}{u} \,d\,\pi^{o}_{j}(u) \\[5pt]
&\le U_{0}+1 + 2\int_{U_{0}}^{D^{j}} \dfrac{(\log\log u)^{j-1}}{(j-1)!\,u \,\log u} \,du \\[5pt]
&\le \dfrac{3\big(\log\log (D^{j})\big)^{j}}{j!}.
\end{align*}
We can bound the numerator by using
\[
\log\log (D^{j}) = (\log\log D) \bigg( 1+\dfrac{\log j}{\log \log D} \bigg) \le 2 \log\log D.
\]
Now, since $s \le \sqrt{\log \log X}$ and $\log \log D \le 2\delta \sqrt{\log \log X}$, we deduce that
\begin{align*}
\sum_{\substack{m'_{1} \le D^{j} \\ m'_{1} \text{ odd squarefree} \\ \omega(m'_{1}) = j}} \pi^{o}_{s-j}(X/m'_{1})
&\le \dfrac{6 X (\log\log X)^{s-j-1} (2 \log \log D)^{j}}{(s-j-1)!\,j!\,\log X} \\
&\le \dfrac{6\cdot 2^{j} X (\log\log X)^{s-j-1} (2 \delta \sqrt{\log \log X}\,)^{j}}{(s-j-1)!\,j!\,\log X} \\[5pt]
&\le \dfrac{6\delta \cdot 4^{j} X (\log\log X)^{s-1}}{(\log \log X)^{j/2} (s-j-1)!\,j!\,\log X} \\[5pt]
&\le \dfrac{6\delta \cdot 4^{j} X (\log\log X)^{s-1}}{s^{j} (s-j-1)!\,j!\,\log X} \\[5pt]
&\le \dfrac{6\delta \cdot 4^{j} X (\log\log X)^{s-1}}{j! \,(s-1)!\,\log X}.
\end{align*}
For $j = s$, we claim the same bound on the number of pairs $(m'_{1},m'_{2})$ is satisfied.  In this case, $m'_{2} = 1$ and $m'_{1} \le D^{s}$, so it suffices to show
\begin{equation}\label{jequalsscase}
D^{s} \le \dfrac{6\delta \cdot 4^{s} X (\log\log X)^{s-1}}{s! \,(s-1)!\,\log X}.
\end{equation}
Since $\delta < 1/2$ and $X$ is large, we have $s < \sqrt{\log\log X}$ and
\[
\log D < 2^{\sqrt{\log\log X}} < \exp\bigg( \dfrac{\log\log X}{\sqrt{\log\log X}} \bigg) = (\log X)^{1/\sqrt{\log\log X}}.
\]  
Hence,
\[
\log (s^{2s} D^{s}) = 2s \log s + s \log D = o(\log X).
\]
Now, after multiplying both sides of \eqref{jequalsscase} by $s! \,(s-1)! \le s^{2s}$ and taking logarithms, one can see that \eqref{jequalsscase} holds.

Combining the above bounds, we see that for $t+s \le \delta \sqrt{\log\log X}$, the number of pairs $(n,m)$ in $\mathcal Q$ for which $F_{n,m}(x)$ is not of the form \eqref{basicFeq} because one of the equations in $\mathcal E'_{2} \cup \mathcal E'_{3} \cup \cdots \cup \mathcal E'_{s}$ holds is bounded by 
\[
\sum_{j=2}^{s} \dfrac{X (\log\log X)^{s-2}}{(s-1)!\,\log X} + \sum_{j=2}^{s} \dfrac{6\delta \cdot 4^{j} X (\log\log X)^{s-1}}{j! \,(s-1)!\,\log X}.
\]
From
\[
s \le \sqrt{\log\log X} 
\qquad \text{and} \qquad
\sum_{j=0}^{\infty} \dfrac{4^{j}}{j!} = e^{4},
\]
this bound is in turn at most
\[
\dfrac{X (\log\log X)^{s-2} \sqrt{\log\log X}}{(s-1)!\,\log X} + \dfrac{6 e^{4} \delta X (\log\log X)^{s-1}}{(s-1)!\,\log X}
< \dfrac{\varepsilon}{4} \cdot \dfrac{X (\log\log X)^{s-1}}{(s-1)!\,\log X},
\]
for $\delta < \varepsilon/(25 e^{4})$ and $X$ sufficiently large depending on $\varepsilon$.
Taking into account the $\pi^{o}_{t}(X)$ choices for $n \le X$, we obtain from \eqref{landautypelem:eq1} that the number of pairs $(n,m) \in \mathcal Q$ for which one of the equations in $\mathcal E'_{2} \cup \mathcal E'_{3} \cup \cdots \cup \mathcal E'_{s}$ holds is bounded by
\[
\dfrac{\varepsilon}{3} \cdot \dfrac{X^{2} (\log\log X)^{t+s-2}}{(t-1)!\,(s-1)!\,\log^{2} X}
\]
provided $\varepsilon \le 1/3$, which we may assume is the case.

For equations in $\mathcal E'_{j}$ with $s+1 \le j \le s+t$ involving the prime divisors of $n$, we want to use a similar analysis.  We choose $m \le X$ odd and squarefree with $\omega(m) = s$ in $\pi^{o}_{s}(X)$ ways and then count the number of ways of choosing $n = p'_{j}n' \le X$ odd and squarefree with $\omega(n') = t-1$ and $p'_{j}$ satisfying an equation in $\mathcal E'_{j}$.  Before we had $s \ge 2$, and using the analogous argument as above, with $D$ the same and with $t \ge 2$, the number of pairs $(p'_{j},n')$ with $n' \le X/D$ is bounded by
\[
\pi^{o}_{t-1}(X/D) \cdot 2^{2^{t+s}} \le \dfrac{X (\log\log X)^{t-2}}{(t-1)!\,\log X}.
\]
In the case that $t = 1$, we again have the interpretation $\pi^{o}_{0}(X/D) = 1$ (that is, one must have $n' = 1$).  In this case, as $2^{2^{t+s}}$ is small compared to say $\sqrt{X}$, the inequality above still holds.
Also, for $n' > X/D$, we can write $n = n'_{1}n'_{2}$ where $n'_{1} \le D^{j-s}$, $\omega(n'_{1}) = j-s$, $n'_{2} \le X/n'_{1}$, and $\omega(n'_{2}) = t-j+s$.  The number of such pairs $(n'_{1},n'_{2})$ is, by a similar argument to that given for $m' > X/D$, bounded by
\begin{align*}
\sum_{\substack{n'_{1} \le D^{j-s} \\ n'_{1} \text{ odd squarefree} \\ \omega(n'_{1}) = j-s}} \pi_{t-j+s}(X/n'_{1})
&\le \sum_{\substack{n'_{1} \le D^{j-s} \\ n'_{1} \text{ odd squarefree} \\ \omega(n'_{1}) = j-s}} \dfrac{2 X (\log\log X)^{t-j+s-1}}{(t-j+s-1)!\,n'_{1}\,\log X} \\[5pt]
&\le \dfrac{6\delta \cdot 4^{j-s} X (\log\log X)^{t-1}}{(j-s)! \,(t-1)!\,\log X},
\end{align*}
where the case $j = t+s$ is handled differently as the case $j = s$ was handled before.  Summing over $j \in [s+1,s+t]$ and taking into account the $\pi^{o}_{s}(X)$ choices for $m \le X$, we obtain from \eqref{landautypelem:eq1} that the number of pairs $(n,m) \in \mathcal Q$ for which one of the equations in $\mathcal E'_{s+1} \cup \mathcal E'_{s+2} \cup \cdots \cup \mathcal E'_{s+t}$ holds is bounded by
\[
\dfrac{\varepsilon}{3} \cdot \dfrac{X^{2} (\log\log X)^{t+s-2}}{(t-1)!\,(s-1)!\,\log^{2} X}.
\]
The theorem now follows from Lemma~\ref{Qsizelemma}.
\end{proof}

\section{Concluding remarks}

Our main goal in this paper was to highlight the new Theorem~\ref{factorizationthm} and its implications on the factorization of $\Phi_{n}(x)+\Phi_{m}(x)$ and $\Phi_{n}(x)\,\Phi_{m}(x)+1$.  We also aimed to bring attention to the nice 26-year-old conjecture of Charles Nicol from \cite{nicol}.  There are certainly other directions that we could have explored related to Theorem~\ref{factorizationthm}.  For example, similar to the analysis of $\Phi_{\ell p}(x) + \Phi_{\ell q}(x)$ given in the introduction, one can see that if $\ell$ and $q$ are primes and $p_{1}, \ldots, p_{q-1}$ and $p_{q}$ are distinct primes $> \ell$, then Theorem~\ref{factorizationthm} implies that the \textit{non-reciprocal} part of
\begin{equation}\label{sumofmorepj}
\Phi_{\ell p_{1}}(x) + \cdots + \Phi_{\ell p_{q}}(x)
\end{equation}
is irreducible.  Also, the analogous result holds if $\ell = 1$ or if $\ell$ in \eqref{sumofmorepj} is replaced by $k = \ell^{t}$ for some $t \ge 2$.

\subsection*{Acknowledgments}
The authors express their gratitude to Pieter Moree for encouraging the collaboration of the authors on this topic and introducing Conjecture~\ref{nicolconj} to the second author.  As noted in Section~\ref{sect2}, the authors also acknowledge and express their appreciation to Alexander Borisov for simplifying their original argument for Theorem~\ref{factorizationthm}.


\begin{thebibliography}{10}

\bibitem{berlekamp}
E.~R.~Berlekamp, Algebraic coding theory, revised 1984 edition, 
Aegean Park Press, Laguna Hills, CA, 1984.

\bibitem{elliott}
P.~D.~T.~A.~Elliott, Probabilistic number theory I, Grundlehren der Mathematischen Wissenschaften, 239, Springer, New York-Berlin, 1979.

\bibitem{fil1999} 
M.~Filaseta,
\textit{On the factorization of polynomials with small {E}uclidean norm},
In Number theory in progress, Vol. 1 (Zakopane-Ko\'scielisko, 1997), pages 143--163, de Gruyter, Berlin, 1999.
  
\bibitem{filsol1999}
M.~Filaseta and I.~Solan,
\textit{An extension of a theorem of Ljunggren},
Math.~Scand.~84 (1999), 5--10.  

\bibitem{hungerford}
T.~W.~Hungerford, Algebra, Holt, Rinehart and Winston, Inc., New York-Montreal, Que.-London, 1974.

\bibitem{landau}
E.~Landau, Handbuch der Lehre von der Verteilung der Primzahlen, Chelsea, New York, 1953.

\bibitem{ljunggren}
W.~Ljunggren, \textit{On the irreducibility of certain trinomials and quadrinomials}, Math.~Scand.~8 (1960), 65--70.

\bibitem{montgomeryvaughan}
H.~L. Montgomery and R.~C. Vaughan, Multiplicative number theory. I. Classical theory, Cambridge Studies in Advanced Mathematics, 97, Cambridge Univ. Press, Cambridge, 2007.

\bibitem{nicol}
C.~Nicol, Google Groups, August 25, 2000, \url{https://groups.google.com/g/sci.math.research/c/H4UJSXZNJac?pli=1}.

\bibitem{sathe}
L.~G.~Sathe, \textit{On a problem of Hardy on the distribution of integers having a given number of prime factors, I}, J.~Indian Math.~Soc.~(N.S.) 17 (1953), 63--82.

\bibitem{selberg1954} 
A.~Selberg, \textit{Note on a paper by L.~G.~Sathe}, J. Indian Math. Soc. (N.S.) 18 (1954), 83--87.

\bibitem{sheilsmall}
T.~B.~Sheil-Small, Complex polynomials, Cambridge Studies in Advanced Mathematics, 75, Cambridge Univ. Press, Cambridge, 2002.

\bibitem{tenenbaum}
G.~Tenenbaum, Introduction to analytic and probabilistic number theory, third edition Translated from the 2008 French edition by Patrick D.~F.~Ion,  Graduate Studies in Mathematics, 163, Amer. Math. Soc., Providence, RI, 2015.
  
\end{thebibliography}
\end{document}